\documentclass{gtpart}

\usepackage{pinlabel}
\usepackage[all]{xy}
\usepackage{stmaryrd}

\agtart

\input diagxy
\let\to\rightarrow
\let\barrsquare\square
\let\square\undefined

\title{A rank inequality for the knot Floer homology of double branched covers}

\author{Kristen Hendricks}
\givenname{Kristen}
\surname{Hendricks}
\address{Department of Mathematics\\
Columbia University\\\newline
2990 Broadway\\
New York NY 10027\\
USA}
\email{hendricks@math.columbia.edu}
\urladdr{math.columbia.edu/~hendricks/}

\volumenumber{12}
\issuenumber{4}
\publicationyear{2012}
\papernumber{81}
\startpage{2127}
\endpage{2178}

\doi{}
\MR{}
\Zbl{}

\keyword{Heegaard Floer}
\keyword{double branched covers}
\keyword{knot theory}
\keyword{Floer cohomology}
\keyword{localization}
\subject{primary}{msc2010}{53D40}
\subject{primary}{msc2010}{57M25}
\subject{primary}{msc2010}{57M27}
\subject{primary}{msc2010}{57R58}

\received{11 July 2011}
\revised{18 June 2012}
\accepted{12 July 2012}
\published{26 Dec 2012}
\publishedonline{26 Dec 2012}
\editor{Zolt\'{a}n Szab\'{o}/Nicholas}
\version{2}
\proposed{}
\seconded{}
\corresponding{}

\arxivreference{math/1107.2154}
\arxivpassword{4dv3nture}

\let\xysavmatrix\xymatrix
\def\xymatrix{\disablesubscriptcorrection\xysavmatrix}
\AtBeginDocument{\let\bar\wbar\let\tilde\wtilde}

\newcommand{\HFK}{\mathit{HFK}}
\newcommand{\CF}{\mathit{CF}}
\makeop{rk}
\newcommand{\llparen}{(\!(}
\newcommand{\rrparen}{)\!)}
\newcommand{\lbracket}{[}
\newcommand{\rbracket}{]}
\newcommand{\inv}{\mathrm{inv}}
\makeop{Ext}
\makeop{Hom}
\makeop{Supp}
\makeop{Sym}
\newcommand{\anti}{\mathrm{anti}}
\makeop{Id}
\makeop{ch}

\newtheorem{theorem}{Theorem}[section]
\newtheorem{lemma}[theorem]{Lemma}
\newtheorem{corollary}[theorem]{Corollary}
\newtheorem{proposition}[theorem]{Proposition} 

\theoremstyle{definition}
\newtheorem{definition}[theorem]{Definition}
\newtheorem*{remark}{Remark}

\begin{document}

\begin{asciiabstract}
Given a knot K in S^3, let Sigma(K) be the double branched cover of
S^3 over K. We show there is a spectral sequence whose E^1 page is
(^HFK(Sigma(K),K) otimes V^{otimes(n-1)}) otimes Z_2((q)), for V a
Z_2-vector space of dimension two, and whose E^{infinity} page is
isomorphic to (^HFK(S^3,K) otimes V^{otimes(n-1)}) otimes Z_2((q)), as
Z_2((q))-modules. As a consequence, we deduce a rank inequality
between the knot Floer homologies ^HFK(Sigma(K),K) and ^HFK(S^3,K).
\end{asciiabstract}

\begin{abstract}  
Given a knot $K$ in $S^3$, let $\Sigma(K)$ be the double branched cover of
$S^3$ over $K$. We show there is a spectral sequence whose $E^1$ page is
$(\widehat{\mathit{HFK}}(\Sigma(K), K) \otimes V^{\otimes(n-1)}) \otimes
\mathbb Z_2(\!(q)\!)$, for $V$ a $\mathbb Z_2$--vector space of dimension two,
and whose $E^{\infty}$ page is isomorphic to $(\widehat{\mathit{HFK}}(S^3,
K) \otimes V^{\otimes(n-1)}) \otimes \mathbb Z_2(\!(q)\!)$, as $\mathbb
Z_2(\!(q)\!)$--modules. As a consequence, we deduce a rank inequality
between the knot Floer homologies $\widehat{\mathit{HFK}}(\Sigma(K),
K)$ and $\widehat{\mathit{HFK}}(S^3, K)$.
\end{abstract}

\maketitle

\section{Introduction}

Heegaard Floer knot homology is a powerful invariant of a knot $K$ in a
three-manifold $Y$ introduced in 2003 by Ozsv{\'a}th and
Szab{\'o}~\cite{MR2065507} and independently by
Rasmussen~\cite{MR2704683}. The theory associates to $(Y,K)$ a bigraded
abelian group $\widehat{\HFK}(Y,K)$ which arises as the homology
of the Floer chain complex of two Lagrangian tori in the symmetric product
of a punctured Heegaard surface for $(Y,K)$. Amongst other properties, for
a knot $K$ in the three-sphere this theory detects the genus of $K$ (see
Ozsv{\'a}th and Szab{\'o}~\cite{MR2023281}) and whether $K$ is fibred (see
Ghiggini~\cite{MR2450204} and Ni~\cite{MR2357503}); its graded Euler characteristic is the Alexander polynomial of $K$~\cite{MR2065507}.

If $K$ is a knot in $S^3$, then we may construct the double branched cover $\Sigma(K)$ of $S^3$ over $K$. The preimage of $K$ under the branched cover map $\pi \co \Sigma(K) \rightarrow S^3$ is a nullhomologous knot in $\Sigma(K)$, also called $K$. The relationship between the knot Floer homology groups $\widehat{\HFK}(S^3,K)$ and $\widehat{\HFK}(\Sigma(K),K)$ has been studied by Grigsby~\cite{MR2253451} and Levine~\cite{MR2443111,MR2794779}.

We prove the following theorem, conjectured by Levine~\cite[Conjecture
4.5]{MR2794779} after being proved in the case of two-bridge knots by
Grigsby~\cite[Theorem 4.3]{MR2253451}. Let $n$ be the number of basepoints
on some Heegaard diagram $\mathcal D$ for $(Y,K)$ whose underlying surface
is $S^2$. Let $\wtilde{\mathcal D}$ be a double branched cover of
$\mathcal D$ which is an $n$--pointed Heegaard diagram for
$(\Sigma(K),K)$. (We will introduce this construction more explicitly in
\fullref{HeegaardFloerSection}).  We will work with a variant of knot
Floer homology, $\widetilde{\HFK}(\wtilde{\mathcal D})$, which is dependent on $n$ and equal to $\widehat{\HFK}(Y,K) \otimes V^{\otimes (n-1)}$, for $V$ a dimension 2 vector space over $\mathbb F_2$.

\begin{theorem} \label{existsspecseq}
There is a spectral sequence whose $E^1$ page is
$$(\widehat{\HFK}(\Sigma(K),K) \otimes V^{\otimes (n-1)}) \otimes \mathbb
Z_2\llparen q\rrparen$$
and whose $E^{\infty}$ page is isomorphic to
$$(\widehat{\HFK}(S^3,K) \otimes V^{\otimes (n-1)}) \otimes \mathbb
Z_2\llparen q\rrparen$$
as $\mathbb Z_2\llparen q\rrparen$--modules.
\end{theorem}

Here $\mathbb Z_2\llparen q\rrparen$ denotes the ring $\mathbb Z_2\llbracket q\rrbracket[q^{-1}]$ of Laurent series in the variable $q$. In particular, we have the following rank inequality.

\begin{corollary} \label{main}
Given $K$ a knot in $S^3$ and $\Sigma(K)$ the double branched cover of $(S^3,K)$ then the following rank inequality holds:
\begin{align*}
\rk(\widehat{\HFK}(S^3,K)) \leq \rk(\widehat{\HFK}(\Sigma(K),K)).
\end{align*}
\end{corollary}

Knot Floer homology admits two gradings, the Maslov or homological grading and the Alexander grading. We will see that the spectral sequence of \fullref{existsspecseq} is generated by a double complex whose two differentials each preserve the Alexander grading on the basepoint-dependent invariant $\widetilde{\HFK}(\mathcal D)$, inducing a splitting of the spectral sequence which the isomorphism of \fullref{existsspecseq} fails to disrupt. Knot Floer homology also splits along the ${\mathrm {spin}}^{\mathrm{c}}$ structures $\mathfrak s$ of $Y$, such that we have
\[
\widehat{\HFK}(Y,K) = \oplus_{\mathfrak s} \widehat{\HFK}(Y,K, \mathfrak s).
\]
The extra factors of $V$ in $\widetilde{\HFK}(D)$ respect this splitting, such that
\[
\widetilde{\HFK}(\mathcal D, \mathfrak s) = \widehat{\HFK}(Y,K, \mathfrak s) \otimes V^{\otimes (n-1)}.
\]
The differentials on the spectral sequence of \fullref{main} interchange pairs of conjugate ${\mathrm{spin}}^{\mathrm{c}}$ structures, preserving a single \textit{canonical ${\mathrm{spin}}^{\mathrm{c}}$ structure} $\mathfrak s_0$ on $\Sigma(K)$, which is moreover the only ${\mathrm{spin}}^{\mathrm{c}}$ structure to survive to the $E^{\infty}$ page of the spectral sequence. Therefore we can sharpen \fullref{main} to the following.

\begin{corollary} \label{sharper}
Given a knot $K$ in $S^3$ and $\mathfrak s_0$ the canonical ${\mathrm{spin}}^{\mathrm{c}}$ structure on its double branched cover $\Sigma(K)$, then we have the rank inequality
\[
\rk \big( \widehat{\HFK}(\Sigma(K), K, \mathfrak s_0)\big) \geq \rk
\big(\widehat{\HFK}(S^3, K)\big).
\]
\end{corollary} 

\begin{corollary} \label{Alexander corollary}
Given a knot $K$ in $S^3$ and $\mathfrak s_0$ the canonical ${\mathrm{spin}}^{\mathrm{c}}$ structure on its double branched cover $\Sigma(K)$, then we have the rank inequality
\[
\rk \big( \widetilde{\HFK}(\wtilde{\mathcal D}, \mathfrak s_0, i)\big)
\geq \rk \big(\widetilde{\HFK}({\mathcal D}, i)\big)
\]
for $i$ any Alexander grading. In particular, for $g$ the top Alexander grading such that $\widehat{\HFK}(\Sigma(K), K, \mathfrak s_0, i)$ is nonzero, this is an inequality of the hat invariant:
\[
\rk \big( \widehat{\HFK}(\Sigma(K), K, \mathfrak s_0, g)\big) \geq \rk
\big(\widehat{\HFK}(S^3, K, g)\big).
\]
\end{corollary}

The key technical tool in these proofs is a recent result of Seidel and Smith concerning equivariant Floer cohomology. Let $M$ be an exact symplectic manifold, convex at infinity, containing exact Lagrangians $L_0$ and $L_1$ and equipped with an involution $\tau$ preserving $(M, L_0, L_1)$. Let $(M^{\inv}, \smash{L_0^{\inv}}, \smash{L_1^{\inv}})$ be the submanifolds of each space fixed by $\tau$. Then under certain stringent conditions on the normal bundle $N(M^{\inv})$ of $M^{\inv}$ in $M$, there is a rank inequality between the Floer cohomology $\mathit{HF}(L_0,L_1)$ of the two Lagrangians $L_0$ and $L_1$ in $M$ and the Floer cohomology $\mathit{HF}(\smash{L_0^{\inv}}, \smash{L_1^{\inv}})$ of $\smash{L_0^{\inv}}$ and $\smash{L_1^{\inv}}$ in $M^{\inv}$. More precisely, they consider the normal bundle $N(M^{\inv})$ to $M^{\inv}$ in $M$ and its pullback $\Upsilon(M^{\inv})$ to $M^{\inv} \times [0,1]$. We ask that $M$ satisfy a $K$--theoretic condition called \textit{stable normal triviality} relative to two Lagrangian subbundles over $\smash{L_0^{\inv}} \times \{0\}$ and $\smash{L_1^{\inv}} \times \{1\}$. Seidel and Smith prove the following.

\begin{theorem}[Seidel and Smith~{\cite[Section 3f]{MR2739000}}]
\label{SeidelSmith}
If $\Upsilon(M^{\inv})$ carries a stable normal trivialization, there is a spectral sequence whose $E^1$ page is $\mathit{HF}(L_0,L_1) \otimes \mathbb Z_2\llparen q\rrparen$ and whose $E^{\infty}$ page is isomorphic to $\mathit{HF}(\smash{L_0^{\inv}}, \smash{L_1^{\inv}}) \otimes \mathbb Z_2\llparen q\rrparen$ as $\mathbb Z_2\llparen q\rrparen$ modules. 
\end{theorem}

In particular, there is the following useful corollary.

\begin{corollary}[Seidel and Smith~{\cite[Theorem 1]{MR2739000}}]
\label{Smith Inequality}
If $\Upsilon(M^{\inv})$ carries a stable normal trivialization, the Floer theoretic version of the Smith inequality holds:
\[
\rk(\mathit{HF}(L_0, L_1)) \geq \rk(\mathit{HF}(\smash{L_0^{\inv}}, \smash{L_1^{\inv}})).
\]
\end{corollary}

This paper is organized as follows: In Section 2 we review the set up for Floer cohomology as used by Seidel and Smith and state the results that imply \fullref{SeidelSmith}. In Section 3 we review the basics of Heegaard Floer knot homology, discuss previous work concerning the knot Floer homology of branched double covers, and state \fullref{stable}, which states that a manifold used to compute knot Floer homology carries a stable normal trivialization, and, together with some symplectic structural data, implies \fullref{existsspecseq}. We also discuss how Corollaries~\ref{sharper} and~\ref{Alexander corollary} follow from \fullref{existsspecseq}. In Section 4 we show that the spaces involved in the computation of knot Floer homology satisfy the basic symplectic structural requirements of Seidel and Smith's theory. In Section 5 we further examine the homotopy type and cohomology of these spaces, producing results we will need to prove \fullref{stable}. In Section 6 we review some important concepts from $K$--theory necessary to the proof of \fullref{stable}, and in Section 7 we finally give a proof of this theorem, completing the proof of \fullref{existsspecseq}. We afterward summarize the deduction of Corollaries~\ref{main}, \ref{sharper}, and~\ref{Alexander corollary} from \fullref{existsspecseq}, already touched on earlier, for the reader's convenience. We finish with some optimistic remarks concerning possible future work. Section 8 is an appendix containing a proof on charts that an important inclusion map of symmetric products, \eqref{embedding}, is holomorphic.

\subsection{Acknowledgements}

I am grateful to Peter Ozsv{\'a}th for suggesting this problem, and to Robert Lipshitz for providing guidance and reading a draft of this paper. Thanks also to Dylan Thurston, Rumen Zarev, and especially Adam Levine for helpful conversations, and to Dan Ramras, who pointed out the argument used in \fullref{AtiyahHirzebruch}. Furthermore, thank you to the reviewer for clear and helpful comments, and in particular for pointing out a gap in the argument of \fullref{SymplecticGeometrySection}.

I was very privileged to be able to attend MSRI's Spring 2010 program Homology Theories of Knots and Links; I would like to express appreciation both to the Institute for providing such a stimulating program in an idyllic setting, and also to the mathematicians in attendance, who were uniformly extremely generous with their time and knowledge. In addition to those noted above, particular mention is due of Jen Hom, Eli Grigsby, Allison Gilmore, Jon Bloom, and Paul Melvin.

I was partially supported by an NSF grant number DMS-0739392.

\section{Spectral sequences for Floer cohomology}

Floer cohomology is an invariant for Lagrangian submanifolds in a symplectic manifold introduced by Floer~\cite{MR965228,MR933228,MR948771}. Many versions of the theory exist; in this section we briefly introduce Seidel and Smith's setting for Floer cohomology before stating the hypotheses and results of their main theorem for equivariant Floer cohomology. Let $M$ be a manifold equipped with an exact symplectic form $\omega = d\theta$ and a compatible almost complex structure $J$. Let $L_0$ and $L_1$ be two exact Lagrangian submanifolds of $M$. For our purposes we can restrict to the case that $L_0$ and $L_1$ intersect transversely. 

\begin{definition}
The Floer chain complex $\CF(L_0,L_1)$ is an abelian group with generators the finite set of points $L_0 \cap L_1$.
\end{definition}

The differential $d$ on $\CF(L_0,L_1)$ counts holomorphic disks
whose boundary lies in $L_0 \cup L_1$ which run from $x_-$ to $x_+$. More
precisely, we choose $\mathbf{J} = J_t$ a time-dependent perturbation of
$J$ and let $\mathcal M(x_{-}, x_{+})$ be the moduli space of Floer
trajectories $u\co \mathbb R \times [0,1] \rightarrow M$ where
\begin{align*}
u(s,0) &\in L_0, &
\partial_s u + J_t(u)\partial_t(u) &= 0, \\
u(s,1) &\in L_1, &
\lim_{s\rightarrow \pm \infty} &= x_{\pm}.
\end{align*}
This moduli space carries a natural action by $\mathbb R$ corresponding to
translation on the coordinate $s$; we let the quotient by this action be
$\widehat{\mathcal M}(x_{-}, x_{+}) = \smash{\frac{\mathcal M(x_{-},
x_+)}{\mathbb R}}$ the set of unparametrized holomorphic curves from $x_-$ to $x_+$.

Before formally defining the differential on $\CF(L_0,L_1)$, we need to impose one further technical condition on $M$ to ensure both that there are only finitely many equivalence classes of holomorphic curves between any two intersection points $x_+,x_{-} \in L_0 \cap L_1$ and that the image of any holomorphic curve $u\co\mathbb R \times [0,1] \rightarrow M$ is contained in some compact set in $M$. We say $\phi\co M \rightarrow \mathbb R$ is \textit{exhausting} if it is smooth, proper, and bounded below. We consider the one-form $d^{\mathbb C}(\phi) = d\phi \circ J$ and the two-form $\omega_{\phi} = -dd^{\mathbb C}(\phi)$. We say that $\phi$ is \textit{$J$--convex} or \textit{plurisubharmonic} if $\omega_{\phi}$ is compatible with the complex structure on $M$, that is, if $\omega_{\phi}(Jv,Jw)=\omega_{\phi}(v,w)$ and $\omega_{\phi}(v,Jv) > 0$ for all $v \in TM$. This ensures that $\omega_{\phi}$ is a symplectic form on $M$.(The term plurisubharmonic indicates that the restriction of $\phi$ to any holomorphic curve in $M$ is subharmonic, hence satisfies the maximum modulus principle.) A noncompact symplectic manifold $M$ with this structure is called \textit{convex at infinity}. If $\omega_{\phi}$ is in fact the symplectic form $\omega$ on $M$, then $M$ is said to be \textit{strictly convex at infinity}.

We use an index condition to determine which strips $u$ count for the differential. Given any $u 
\in \mathcal M (x_-,x_+)$, we can associate to $u$ a Fredholm operator
$D_{\mathbf{J}}u \co \mathcal W^1_u \to \mathcal W^0_u$ from $\mathcal
W^1_u = \{X \in W^{1,p}(u^*TM) \co X(\cdot, 0) \in u^*TL_0, X(\cdot, 1)
\in u^*TL_1\}$ to $W^0_u$. (Here $p>2$ is a fixed real number.) This
operator describes the linearization of Floer's equation, $\partial_s u +
J_t(u)\partial_t(u) =0$, near $u$. We say that $\mathbf{J}$ is
\textit{regular} if $D_{\mathbf{J}} u$ is surjective for all finite energy holomorphic strips $u$.

\begin{lemma}
If $M$ is an exact symplectic manifold with a compatible almost convex
structure $J$ which is convex at infinity and $L_0,L_1$ are exact
Lagrangian submanifolds also convex at infinity, then a generic choice of
$\mathbf{J}$ perturbing $J$ is regular.
\end{lemma}

Floer's original proof of this result~\cite[Proposition~2.1]{MR965228}
and Oh's revision~\cite[Proposition~3.2]{MR1223659} were for compact
manifolds, but, as observed by McDuff and Salamon~\cite[Section~9.2]{MR2045629} and indeed by Sikorav~\cite{Sikorav} in his review of
Floer's paper~\cite{MR965228}, the proof carries through identically for
noncompact manifolds which are convex at infinity. Choose such a generic
regular $\mathbf{J}$. We let $\mathcal M_1(x_{-}, x_{+})$ be the set of
trajectories $u$ in $\mathcal M(x_{-},x_{+})$ such that the Fredholm
index of $D_{\mathbf{J}}u = 1$.

\begin{lemma}[Floer~{\cite[Lemma 3.2]{MR965228}}]
If $\mathbf{J}$ is regular, $\mathcal M_1(x_{-},x_{+})$ is a smooth, compact one-manifold such that $\#\widehat{ \mathcal M}_1(x_{-}, x_{+}) = n^{x_{-} x_+}$ is finite. Moreover, for any $x_-, x_+ \in \CF(L_0,L_1)$, the sum
\[
\sum_{x \in \CF(L_0,L_1)}n^{x_- x} n^{xx_+}
\]
is zero modulo two. 

\end{lemma}

Therefore we make the following definition.

\begin{definition}[Floer~{\cite[Definition~3.2]{MR965228}}]
\label{FloerCohomologyDefn}
The Floer cohomology $\mathit{HF}(L_0,L_1)$ is the homology of $\CF(L_0,L_1)$ with respect to the differential 
\begin{align} 
\delta(x_{-}) = \sum_{x_+ \in \CF(L_0,L_1)} \# \wwhat{\mathcal M}_1(x_{-},x_{+})x_+ \label{differential}
\end{align}
with respect to a regular family of almost complex structures
$\mathbf{J}$ perturbing $J$. 

\end{definition}

Now suppose that $M$ carries a symplectic involution $\tau$ preserving
$(M, L_0, L_1)$ and the forms $\omega$ and $\theta$. Let the submanifold
of $M$ fixed by $\tau$ be $M^{\inv}$, and similarly for
$L_i^{\inv}$ for $i=0,1$. We can define the Borel (or equivariant)
cohomology of $(M, L_0, L_1)$ with respect to this involution. Seidel and
Smith give a geometric description of the cochain complexes used to
produce equivariant Floer cohomology; we'll content ourselves with an
algebraic description, referring the reader to their paper~\cite[Section~3]{MR2739000} for further geometric detail. Notice that the usual Floer
chain complex $\CF(L_0, L_1)$ carries an induced involution
$\tau^{\#}$ which takes an intersection point $x \in L_0 \cap L_1$ to the
intersection point $\tau(x) \in L_0 \cap L_1$. This map $\tau^{\#}$ is not
a chain map with respect to a generic family of complex structures on $M$.
However, suppose that we are in the nice case that we can find a suitable
family of complex structures $\mathbf{J}$ on $M$ such that $\tau^{\#}$
commutes with the differential on $\CF(L_0, L_1)$. (Part of Seidel
and Smith's use of their technical conditions on the bundle
$\Upsilon(M^{\inv})$ is to establish that such a $\mathbf{J}$ exists~\cite[Lemma~19]{MR2739000}.) Then $\CF(L_0, L_1)$ is a chain complex over $\mathbb F_2[\mathbb Z_2] = \mathbb F_2[\tau^{\#}]/\langle (\tau^{\#})^2 = 1\rangle$. Indeed, $(1 + \tau^{\#})^2 = 0$, so there is a chain complex
\begin{align*}
\xymatrix{
0 \rightarrow \CF(L_0, L_1) \ar[r]^-{1 + \tau^{\#}} & \CF(L_0, L_1) \ar[r]^-{1 + \tau^{\#}} & \CF(L_0, L_1)\cdots
}
\end{align*}

\begin{definition}

If $\CF(L_0, L_1)$ is the Floer chain complex and $\tau^{\#}$ is a chain map with respect to the complex structure on $M$, $\mathit{HF}_{\mathrm{borel}}(L_0, L_1)$ is the homology of the complex $\CF(L_0,L_1) \otimes \mathbb Z_2\llbracket q\rrbracket$ with respect to the differential $\delta + (1 + \tau^{\#})q$. 
\end{definition}

Therefore the double complex
\label{doublecomplex}
$$\bfig
  \morphism[0`\CF(L_0,L_1);]
  \morphism(500,0)<800,0>[\phantom{\CF(L_0,L_1)}`\CF(L_0,L_1);1+\tau^{\#}]
  \morphism(350,0)/@{<-}@(ul,ur)/<300,0>[\phantom{\CF(L_0,L_1)}`\phantom{\CF(L_0,L_1)};\delta]
  \morphism(1300,0)<800,0>[\phantom{\CF(L_0,L_1)}`\CF(L_0,L_1);1+\tau^{\#}]
  \morphism(1150,0)/@{<-}@(ul,ur)/<300,0>[\phantom{\CF(L_0,L_1)}`\phantom{\CF(L_0,L_1)};\delta]
  \morphism(2100,0)<600,0>[\phantom{\CF(L_0,L_1)}`\cdots;1+\tau^{\#}]
  \morphism(1950,0)/@{<-}@(ul,ur)/<300,0>[\phantom{\CF(L_0,L_1)}`\phantom{\CF(L_0,L_1)};\delta]
  \efig$$
induces a spectral sequence whose first page is $\mathit{HF}(L_0,L_1) \otimes \mathbb Z_2\llbracket q\rrbracket$ and which converges to $\mathit{HF}_{\mathrm{borel}}(L_0,L_1)$.

There is another more algebraic method of generating this complex. We may
begin by considering the Floer \textit{homology} complex on $L_0$ and
$L_1$, which is constructed identically to the cohomology complex except
in counting holomorphic strips $u$ of Maslov index 1 with $u(s,0) \in L_1$
and $u(s,1) \in L_0$; that is, it is equal to the Floer cohomology
$\mathit{HF}(L_1,L_0)$. Let $n_{x_- x_+}$ be the number of such Floer
trajectories after quotienting by the translation action on $\mathbb R$.
Let $d$ be the Floer homology differential on $\CF(L_1, L_0)$. If
$\mathbf{J}$ is a time-dependent perturbation of $J$ which is regular for Floer cohomology, it is also regular for Floer homology. Moreover, notice that $\CF(L_1,L_0)$ carries an involution $\tau_{\#} = \tau^{\#}$, also not generically a chain map.

The following relationship between the two theories is well-known.

\begin{lemma}
The Floer cohomology complex $\CF(L_0,L_1)$ is canonically isomorphic to the complex $\mathrm{Hom}_{\mathbb Z_2}(\CF(L_1,L_0), \mathbb Z_2)$ with the dual differential $d^{\dagger}$ as chain complexes.
\end{lemma}

\begin{proof}
Since the group $\CF(L_0, L_1) = \CF(L_1,L_0)$ has a canonical set of generators in the intersection points of $L_0$ and $L_1$, the group $\mathrm{Hom}(\CF(L_1,L_0), \mathbb Z_2)$ is canonically isomorphic to $\CF(L_0,L_1)$ as abelian groups. (Indeed, in some moral sense the space of maps ought to be the chain complex for Floer cohomology.) It remains to be shown that $d^{\dagger} = \delta$. First observe that $n_{x_- x_+} = n^{x_+ x_ -}$: if $u \co \mathbb R \times [0,1] \rightarrow M$ is a Floer trajectory of index 1 from $x_-$ to $x_+$ which counts for the differential $\delta$, then $v\co \mathbb R \times [0,1] \rightarrow M$ defined by $v(s,t) = u(-s, 1-t)$ is a Floer trajectory from $x_+$ to $x_-$ which counts for the differential $\delta$. Let $x$ be an intersection point of $L_0$ and $L_1$, and $x^*$ its dual in $\mathrm{Hom}(\CF(L_1,L_0), \mathbb Z_2)$. Then if $y$ is another intersection point, we have
$$\langle d^{\dagger} x^*, y \rangle  = \langle x^*, d y^* \rangle
= \bigg\langle x^*, \sum_{z \in L_1 \cap L_0} n_{yz}z \bigg\rangle
=n_{yx}
=n^{xy}$$
So $y^*$ appears in $d^{\dagger}x^*$ with coefficient $n^{xy}$. Since $y$ appears in $\delta x$ with coefficient $n^{xy}$, the two chain complexes are isomorphic, as promised.
\end{proof}

A similar argument applies to $\CF(L_1,L_0)$ and
$\mathrm{Hom}_{\mathbb Z_2}(\CF(L_0,L_1), \mathbb Z_2)$. In
particular, if $\mathbf{J}$ is a perturbation of $J$ with respect to which
$\tau^{\#}$ is a chain map, there is a chain map $(\tau^{\#})^{\dagger}$
on $\mathrm{Hom}_{\mathbb Z_2}(\CF(L_0,L_1), \mathbb Z_2)$ which
is identified with $\tau_{\#}$ with respect to the isomorphism. Ergo
$\tau_{\#}$ is also a chain map with respect to $\mathbf{J}$.

This leads to a more algebraic definition of equivariant Floer cohomology.

\begin{lemma} \label{homspace}

The equivariant Floer cohomology $\mathit{HF}_{\mathrm{borel}}(L_0,L_1)$ is isomorphic to
\[
\Ext_{\mathbb F_2[\mathbb Z_2]}(\CF(L_1,L_0), \mathbb F_2).
\]
\end{lemma}

Here we regard $\mathbb F_2$ as the trivial module over $\mathbb F_2[\mathbb Z_2]$.

\begin{proof}
We will show that the double complex that computes $\Ext_{\mathbb
F_2[\mathbb Z_2]}(\CF(`L_1`,`L_0`),`\mathbb F_2`)$ is isomorphic to the double complex from which our spectral sequence arises. Consider the following free resolution of $\mathbb F_2$ over $\mathbb F_2[\mathbb Z_2]$.
\begin{align*}
\xymatrix{
\cdots \ar[r]^{1 + \tau_{\#}} & \mathbb F_2[\mathbb Z_2] \ar[r]^{1 + \tau_{\#}} & \mathbb F_2[\mathbb Z_2] \ar[r] & 0
}
\end{align*}
We may obtain a free resolution of $\CF(L_1,L_0)$ by tensoring it with the chain complex above over $\mathbb F_2$. This produces a double complex 
%
$$\bfig
  \morphism<1000,0>[\cdots`\CF(L_1,L_0){\otimes_{\mathbb{F}_2}}
    \mathbb{F}_2\lbracket\mathbb{Z}_2\rbracket;
    1{\otimes}(1{+}\tau_{\#})]
  \morphism(1000,0)<1400,0>[
    \phantom{\CF(L_1,L_0){\otimes_{\mathbb{F}_2}}
    \mathbb{F}_2\lbracket\mathbb{Z}_2\rbracket}`
    \CF(L_1,L_0){\otimes_{\mathbb{F}_2}}
    \mathbb{F}_2\lbracket\mathbb{Z}_2\rbracket;
    1{\otimes}(1{+}\tau_{\#})]
  \morphism(2400,0)<800,0>[
    \phantom{\CF(L_1,L_0){\otimes_{\mathbb{F}_2}}
    \mathbb{F}_2\lbracket\mathbb{Z}_2\rbracket}`0;]
  \morphism(900,0)/@{<-}@(ul,ur)/<300,0>[
    \phantom{\CF(L_1,L_0){\otimes_{\mathbb{F}_2}}
    \mathbb{F}_2\lbracket\mathbb{Z}_2\rbracket}`
    \phantom{\CF(L_1,L_0){\otimes_{\mathbb{F}_2}}
    \mathbb{F}_2\lbracket\mathbb{Z}_2\rbracket}; d]
  \morphism(2300,0)/@{<-}@(ul,ur)/<300,0>[
    \phantom{\CF(L_1,L_0){\otimes_{\mathbb{F}_2}}
    \mathbb{F}_2\lbracket\mathbb{Z}_2\rbracket}`
    \phantom{\CF(L_1,L_0){\otimes_{\mathbb{F}_2}}
    \mathbb{F}_2\lbracket\mathbb{Z}_2\rbracket}; d]
  \efig$$
To compute $\Ext_{\mathbb F_2[\mathbb Z_2]}(\CF(L_0,L_1), \mathbb F_2)$, we must take the homology of the double complex $\mathrm{Hom}_{\mathbb F_2[\mathbb Z_2]}(\CF(L_1,L_0) \otimes_{\mathbb Z_2} \mathbb F_2[\mathbb Z_2], \mathbb F_2)$ with respect to the duals of the maps $d$ and $1 + \tau_{\#}$.

However, suppose $\phi \in \mathrm{Hom}_{\mathbb F_2[\mathbb Z_2]}(\CF(L_1,L_0) \otimes_{\mathbb F_2} \mathbb F_2[\mathbb Z_2], \mathbb F_2)$. Then since $\phi$ is equivariant with respect to the action of $\tau_{\#}$ on $\CF(L_1,L_0) \otimes_{\mathbb F_2} \mathbb F_2[\mathbb Z_2]$, we see $\phi(x \otimes \tau_{\#}) = \phi(\tau_{\#}x \otimes 1)$, that is, $\phi$ is determined by its behavior as a $\mathbb F_2$--linear map on $\CF(L_1,L_0) \otimes \{1\}$. Hence there is a canonical isomorphism 
\[
\mathrm{Hom}_{\mathbb F_2[\mathbb Z_2]}(\CF(L_1,L_0) \otimes_{\mathbb F_2} \mathbb F_2[\mathbb Z_2], \mathbb F_2) \cong \mathrm{Hom}_{\mathbb F_2}(\CF(L_1,L_0), \mathbb F_2)
\]
Since this isomorphism is natural, we can compute $\Ext_{\mathbb F_2[\mathbb Z_2]}(\CF(L_1,L_0), \mathbb F_2)$ from the double complex
%
$$\bfig
  \morphism<900,0>[\cdots`\Hom(\CF(L_1,L_0),\mathbb{F}_2)_{d^{\dagger}};
    (1{+}\tau_{\#})^{\dagger}]
  \morphism(900,0)<1400,0>[
    \phantom{\Hom(\CF(L_1,L_0),\mathbb{F}_2)_{d^{\dagger}}}`
    \Hom(\CF(L_1,L_0),\mathbb{F}_2)_{d^{\dagger}};
    (1{+}\tau_{\#})^{\dagger}]
  \morphism(2300,0)<900,0>[
    \phantom{\Hom(\CF(L_1,L_0),\mathbb{F}_2)_{d^{\dagger}}}`0.;
    (1{+}\tau_{\#})^{\dagger}]
  \morphism(800,0)/@{<-}@(ul,ur)/<300,0>[
    \phantom{\Hom(\CF(L_1,L_0),\mathbb{F}_2)_{d^{\dagger}}}`
    \phantom{\Hom(\CF(L_1,L_0),\mathbb{F}_2)_{d^{\dagger}}};]
  \morphism(2200,0)/@{<-}@(ul,ur)/<300,0>[
    \phantom{\Hom(\CF(L_1,L_0),\mathbb{F}_2)_{d^{\dagger}}}`
    \phantom{\Hom(\CF(L_1,L_0),\mathbb{F}_2)_{d^{\dagger}}};]
  \efig$$
We saw in the proof of \fullref{homspace} that $d^{\dagger} = \delta$; moreover, since $\tau_{\#}$ and $\tau^{\#}$ are in point of fact the same map on the generators of $\CF(L_0,L_1)$, $(\tau_{\#})^{\dagger} = \tau^{\#}$. Therefore this is precisely the double complex we used to define equivariant Floer cohomology.\end{proof}

Seidel and Smith's result concerns the existence of a localization map
$$\mathit{HF}_{\mathrm{borel}} (L_0, L_1)\longrightarrow
\mathit{HF}(\smash{L_0^{\inv}}, \smash{L_1^{\inv}}),$$
where the second
space is the Floer cohomology of the two Lagrangians $\smash{L_0^{\inv}}$
and $\smash{L_1^{\inv}}$ in $M^{\inv}$. The main goal is to
produce a family of $\tau$--invariant complex structures on $M$ such that,
for $u\co \mathbb R \times [0,1] \rightarrow M^{\inv}$, the index
of the operator $D_{\mathbf{J}}u$ of $u$ with respect to $\mathbf{J}$ in
$M$ differs from the index of the operator $D_{\mathbf{J}^{\inv}}$
of $u$ with respect to $\mathbf{J}^{\inv}$ in $M^{\inv}$ by a constant.

Consider the normal bundle $N(M^{\inv})$ to $M^{\inv}$ in $M$ and its Lagrangian subbundles $N(L_i^{\inv})$ the normal bundles to each $L_i^{\inv}$ in $L_i$. The construction requires one additional degree of freedom, achieved by pulling back the bundle $N(M^{\inv})$ along the projection map $M^{\inv} \times [0,1] \rightarrow M^{\inv}$. Call this pullback $\Upsilon(M^{\inv})$. This bundle is constant with respect to the interval $[0,1]$. Its restriction to each $M^{\inv} \times \{t\}$ is a copy of $N(M^{\inv})$ which will occasionally, by a slight abuse of notation, be called $N(M^{\inv}) \times \{t\}$; similarly, for $i=0,1$ the copy of $N(L_i^{\inv})$ above $L_i^{\inv} \times \{t\}$ will be referred to as $N(L_i^{\inv}) \times \{t\}$.

We make a note here of the correspondence between our notation and Seidel and Smith's original usage. Our bundle $\Upsilon(M^{\inv})$ is their $TM^{\anti}$; while our $N(\smash{L_0^{\inv}}) \times \{0\}$ is their $T\smash{L_0^{\inv}}$ and our $N(\smash{L_1^{\inv}}) \times \{1\}$ is their $TL_1^{\anti}$. (The name $TL_1^{\anti}$ is also used for the bundle that we denote $N(\smash{L_1^{\inv}}) \times \{0\}$, using the obvious isomorphism between the bundles.)

We are now ready to introduce the notion of a stable normal trivialization of $\Upsilon(M^{\inv})$. We denote the trivial bundle $X \times \mathbb C^n \rightarrow X$ by $\mathbb R^n$, whenever the base space $X$ is clear from context, and similarly for $\mathbb R^n$.

\begin{definition}[Seidel and Smith~{\cite[Definition~18]{MR2739000}}]
\label{stablenormaltriv}
A stable normal trivialization of the vector bundle $\Upsilon(M^{\inv})$ over $M^{\inv} \times [0,1]$ consists of the following data:

\begin{itemize}
\item a stable trivialization of unitary vector bundles $\phi \co \Upsilon(M^{\inv}) \oplus \mathbb C^{K} \rightarrow \mathbb C^{k_{\anti} + K}$ for some $K$;

\item a Lagrangian subbundle $\Lambda_0 \subset
(\Upsilon(M^{\inv}))|_{[0,1] \times L^{\inv}_0}$ such that
$\Lambda_0|_{\{0\} \times L^{\inv}_0} = (N(\smash{L_0^{\inv}})\times
\{0\})\oplus \mathbb R^K$ and $\phi(\Lambda_0|_{\{1\} \times
\smash{L_0^{\inv}}}) = \mathbb R^{k_{\anti} + K}$; and

\item a Lagrangian subbundle $\Lambda_1 \subset (\Upsilon(M^{\inv}))|_{[0,1] \times L^{\inv}_1}$ such that $\Lambda_1|_{\{0\} \times L^{\inv}_1} = (N(\smash{L_1^{\inv}})\times\{0\})\oplus \mathbb R^K$ and $\phi(\Lambda_1|_{\{1\} \times \smash{L_1^{\inv}}}) = i\mathbb R^{k_{\anti} + K}$.
\end{itemize}
\end{definition}

The crucial theorem of~\cite{MR2739000}, proved through extensive geometric analysis and comparison with the Morse theoretic case, is as follows.

\begin{theorem}[Seidel and Smith~{\cite[Theorem~20]{MR2739000}}]
\label{Localization}
If $\Upsilon(M^{\inv})$ carries a stable normal trivialization, then $\mathit{HF}_{\mathrm{borel}}(L_0,L_1)$ is well-defined and there are localization maps
\[
\Delta^{(m)}\co \mathit{HF}_{\mathrm{borel}} \rightarrow \mathit{HF}(\smash{L_0^{\inv}}, \smash{L_1^{\inv}})\llbracket q\rrbracket
\]
defined for $m\gg0$ and satisfying $\Delta^{(m+1)} = q\Delta^{(m)}$. Moreover, after tensoring over $\mathbb Z_2\llbracket q\rrbracket$ with $\mathbb Z_2\llparen q\rrparen$ these maps are isomorphisms. 
\end{theorem}

This implies \fullref{SeidelSmith}. Because the localization maps are well-behaved with respect to the $\mathbb Z_2\llparen q\rrparen$ module structures, we also deduce \fullref{Smith Inequality}.

\section{Heegaard Floer homology preliminaries} \label{HeegaardFloerSection}

We pause to review the construction of knot Floer homology, first defined
by Oszv\'ath and Szab\'o~\cite{MR2065507} and Rasmussen~\cite{MR2704683},
along with some of its interactions with the double branched cover
construction. We work in coefficients modulo two.

\begin{definition}\label{HeegaardDefinition}
A \textit{multipointed Heegaard diagram} $\mathcal D = (S, \boldsymbol
\alpha, \boldsymbol \beta, \mathbf{w}, \mathbf{z})$ consists of the following data.

\begin{itemize}
\item An oriented surface $S$ of genus $g$ 
\item Two sets of basepoints $\mathbf{w} = (w_1,\ldots,w_n)$ and $\mathbf{z} =
(z_1,\ldots,z_n)$ \item Two sets of closed embedded curves
$$\boldsymbol \alpha = \{\alpha_1,\ldots , \alpha_{g+n-1}\}
\quad\text{and}\quad
\boldsymbol \beta = \{\beta_1,\ldots ,\beta_{g+n-1}\}$$
such that each of $\boldsymbol \alpha$
and $\boldsymbol \beta$ spans a $g$--dimensional subspace of $H_1(S)$,
$\alpha_i \cap \alpha_j = \emptyset = \beta_i \cap \beta_j$ for $i\neq j$,
each $\alpha_i$ and $\beta_j$ intersect transversely, and each component
of $S - \cup \alpha_i$ and of $S - \cup \beta_i$ contain exactly one point
of $\mathbf{w}$ and one point of $\mathbf{z}$. 

\end{itemize}

\end{definition}

We use $\mathcal D$ to obtain an oriented three-manifold $Y$ by attaching
two-handles to $S \times I$ along the curves $\alpha_i \times \{0\}$ and
$\beta_i \times \{1\}$ and filling in $2n$ three-balls to close the
resulting manifold. This yields a handlebody decomposition $Y =
H_{\boldsymbol \alpha} \cup_{S} H_{\boldsymbol \beta}$ of $Y$. The Heegaard
diagram $\mathcal D$ furthermore determines a knot or link in $Y$: connect
the $z$ basepoints to the $w$ basepoints in the complement of the curves
$\alpha_i$ and push these arcs into the handlebody $H_{\boldsymbol \alpha}$,
and connect the $w$ basepoints to the $z$ basepoints in the complement of
the curves $\beta_i$ and push these arcs into the $H_{\boldsymbol \beta}$ handlebody. In this paper we will be concerned only with the case that this produces a knot $K$.

We impose one further technical requirement on $\mathcal D$. A
\textit{periodic domain} is a 2--chain on
$S\backslash\{\mathbf{w},\mathbf{z}\}$ whose boundary may be expressed as a sum of the $\alpha$ and $\beta$ curves. The set of periodic domains on $S$ is in bijection with $H_2(Y)$. We say that $D$ is \textit{weakly admissible} if every periodic domain on $S$ has both positive and negative local multiplicities, and require that any Heegaard diagram we use to compute knot Floer homology have this property.

Given a pair $(Y,K)$, we may produce a Heegaard diagram $\mathcal D$ for
$(Y,K)$ via the following strategy.  Let $f\co(Y,K) \rightarrow [0,3]$ be
a self-indexing Morse function with $n$ critical points each of index zero
and index three, and $g+n-1$ critical points each of index one and index
two. Furthermore, insist that $K$ is a union of flowlines between critical
points of index zero and index three, and passes once through each such
critical point.  Then we have $S = f^{-1}(\frac{3}{2})$ is a surface of
genus $g$.  Draw $\alpha$ curves at the intersection of $S$ with the
ascending manifolds of the critical points of index one, and $\beta$
curves at the intersection of $S$ with the descending manifolds of the
critical points of index two.  Finally, let the $w$ basepoints be the
intersection of flowlines in $K$ from index zero critical points to index
three critical points, and the $z$ basepoints be the intersection of
flowlines in $K$ from index three critical points to index zero critical
points.  This produces a $\mathcal D = (S, {\boldsymbol \alpha},
{\boldsymbol \beta}, \mathbf{w}, \mathbf{z})$ satisfying the conditions of \fullref{HeegaardDefinition}; moreover we may choose $f$ such that weak admissibility is also satisfied.

The construction of the knot Floer homology $\widehat{\HFK}(Y,K)$ uses the
symmetric product $\mathrm{Sym}^{g+n-1}(S)$ consisting of all unordered
$(g+n-1)$--tuples of points in $S$. This space is the quotient of
$(S)^{g+n-1}$ by the action of the symmetric group $S_{g+n-1}$ permuting
the factors of $(S)^{g+n-1}$, and its holomorphic structure is defined by
insisting that the quotient map $(S)^{g+n-1} \rightarrow
\mathrm{Sym}^{g+n-1}(S)$ be holomorphic. In particular, if $j$ is a
complex structure on $S$, there is a natural complex structure
$\mathrm{Sym}^{g+n-1}(j)$ on the symmetric product. There are two
transversely intersecting submanifolds of $\mathrm{Sym}^{g+n-1}(S)$ of
especial interest, namely the two totally real embedded tori $\mathbb
T_{\boldsymbol \alpha} = \alpha_1 \times\cdots \times \alpha_{g+n-1}$ and
$\mathbb T_{\boldsymbol \beta} = \beta_1 \times \cdots \times
\beta_{g+n-1}$. The chain complex $\widehat{\mathit{CFK}}(\mathcal D)$ for
knot Floer homology is generated by the finite set of intersection points
of $\mathbb{T}_{\boldsymbol{\alpha}}$ and
$\smash{\mathbb{T}_{\boldsymbol{\beta}}}$. More
concretely, a generator of $\widehat{\mathit{CFK}}(\mathcal D)$ is a point
$\mathbf{x} = (x_1 \cdots x_{g+n-1}) \in \mathrm{Sym}^{g+n-1}(S)$ such that each $\alpha$ or $\beta$ curve contains a single $x_i$.

In its original form, knot Floer homology is computed as follows: let
$\mathbf{x}, \mathbf{y}$ be two intersection points in
$\widehat{\mathit{CFK}}(\mathcal D)$. Denote by $\pi_2(\mathbf{x},\mathbf{y})$ the set
of Whitney disks $\phi \co B_1(0) \rightarrow \mathrm{Sym}^{g+n-1}(S)$
from the unit disk in the complex plane to our symmetric product such that
$\phi(-i) = \mathbf{x}$, $\phi(i) = \mathbf{y}$ and $\phi$ maps the portion of
the boundary of the unit disk with positive real part into $\mathbb
T_{\boldsymbol \alpha}$ and the portion with negative real part into
$\mathbb T_{\boldsymbol \beta}$. The most common method of studying such maps $\phi$ is to use the following familiar construction of Ozsv{\'a}th and Szab{\'o} to associate to any homotopy class of Whitney disks in $\pi_2(x,y)$ a domain in $S$. There is a $(g+n-1)$--fold branched cover 
\[
S \times \mathrm{Sym}^{g+n-2}(S) \rightarrow \mathrm{Sym}^{g+n-1}(S).
\]
The pullback of this branched cover along $\phi$ is a $(g+n-1)$--fold branched cover of $B_1(0)$ which we shall denote $\Sigma(B_1(0))$. Consider the induced map on $\Sigma(B_1(0))$ formed by projecting the total space of this fibration to $S$.
\begin{align*}
\xymatrix{
\Sigma(B_1(0)) \ar[r] \ar[d] & S \times \mathrm{Sym}^{g+n-2}(S) \ar[r] \ar[d] & S \\
B_1(0) \ar[r]^-{\phi} & \mathrm{Sym}^{g+n-1}(S)
}
\end{align*}
We associate to $\phi$ the image of this projection counted with
multiplicities; to wit, we let $D = \Sigma a_i D_i$ where $D_i$ are the
closures of the components of $S - \cup \alpha_i - \cup \beta_i$ and $a_i$
is the algebraic multiplicity of the intersection of the holomorphic
submanifold $V_{x_i} = \{x_i\} \times \mathrm{Sym}^{g+n-2}(S)$ with
$\phi(B_1(0))$ for any interior point $x_i$ of $D_i$. The boundary of $D$
consists of $\alpha$ arcs from points of $\mathbf{x}$ to points of
$\mathbf{y}$ and $\beta$ arcs from points of $\mathbf{y}$ to points of
$\mathbf{x}$. If $D_i$ contains a basepoint $z_j$, then we introduce some additional notation by letting $a_i = n_{z_j}(\phi)$ be the algebraic intersection number of $z_j \times \mathrm{Sym}^{g+n-2}(S)$ with the image of $\phi$.

The Maslov index $\mu(\phi)$ can be computed using the associated domain
$\Sigma a_i D_i$ in a formula of Lipshitz's~\cite[Proposition
4.2]{MR2240908}. For each domain $D_i$, let $e(D_i)$ be the Euler measure
of $D_i$; in particular, if $D$ is a convex $2k$--gon, $e(D_i) = 1-
\frac{k}{2}$. Let $p_{\mathbf{x}}(D)$ be the sum of the average of the
multiplicities of $D$ at the four corners of each point in $\mathbf{x}$
and likewise for $p_{\mathbf{y}}(D)$. Then the Maslov index is
\[
\mu(\phi) = \sum a_i e(D_i) + p_{\mathbf{x}}(D) + p_{\mathbf{y}}(D).
\]
The differential $\partial$ on $\widehat{\mathit{CFK}}(\mathcal D)$ counts
the dimension of the moduli spaces of pseudo-holomorphic curves of Maslov
index one in $\pi_2(\mathbf{x}, \mathbf{y})$.
\begin{align*}
\partial(\mathbf{x}) = \sum_{\mathbf{y} \in \mathbb T_{\boldsymbol \alpha} \cap
\mathbb T_{\boldsymbol \beta}} \sum_{\substack{\phi \in \pi_2(\mathbf{x},
\mathbf{y}) \co \\
	 \mu(\phi) =1 \\
	  n_{w_i}(\phi) = 0 \\
	  n_{z_j}(\phi) = 0 
	  }}
	  \#\left(\frac {M(\phi)}{\mathbb R}\right) \mathbf{y}
\end{align*}
Ozsv{\'a}th and Szab{\'o} have shown that this is a well-defined differential in~\cite{MR2065507}. Indeed, once we show that the homology of $\widehat{\mathit{CFK}}(\mathcal D)$ with respect to $\partial$ can be seen as the Floer cohomology of a suitable manifold, this will be a special case of the well-definedness of the differential of \fullref{FloerCohomologyDefn}.

The chain complex $(\widehat{\mathit{CFK}}(D), \partial)$ splits along the $\mathrm{spin}^{\mathrm{c}}$ structures on $Y$. A Mayer-Vietoris argument (plus Poincar\'{e} duality) shows that
\begin{align*}
H^2(Y) \cong H_1(Y) \cong \frac{H_1(\mathrm{Sym}^{g+n-1}(S))} {H_1(\mathbb
T_{\boldsymbol \alpha}) \oplus H_1(\mathbb T_{\boldsymbol \beta})} \cong
\frac{H_1(S)}{\{[\boldsymbol \alpha_1`], \ldots , [\boldsymbol \alpha_{g+n-1}`],
[\boldsymbol \beta_1`], \ldots , [\boldsymbol \beta_{g+n-1}`]\}}
\end{align*}
(We will discuss the cohomology rings of certain symmetric products at far greater length in \fullref{HomotopyCohomologySection}.)

Any two intersection points $\mathbf{x}$ and $\mathbf{y}$ in
$\widehat{\mathit{CFK}}(\mathcal D)$ can be connected by a one-cycle
$\gamma_{\mathbf{x}, \mathbf{y}}$ of $\alpha$ arcs from points in
$\mathbf{x}$ to
$\mathbf{y}$ and $\beta$ arcs from points in $\mathbf{y}$ to $\mathbf{x}$; there is
then a Whitney disk $\phi$ between $\mathbf{x}$ and $\mathbf{y}$ exactly when
the image $\epsilon(\mathbf{x}, \mathbf{y})$ of $\gamma_{\mathbf{x},
\mathbf{y}}$ is
trivial in
$$\frac{H_1(S)}{\{[\alpha_1], \ldots , [\alpha_{g+n-1}], [\beta_1], \ldots , [
\beta_{g+n-1}]\}} \cong H_1(Y),$$
or when the Poincar\'{e} dual of
$\epsilon(\mathbf{x}, \mathbf{y})$ is trivial in $H^2(Y)$. Therefore the
chain complex splits along an affine copy of $H^2(Y)$, or along the
$\mathrm{spin}^{\mathrm{c}}$ structures of $Y$. To pin down the
$\mathrm{spin}^{\mathrm{c}}$ structure corresponding to a generator
$\mathbf{x}$, let $f$ be a Morse function compatible with $\mathcal D$ and
let $N_{\mathbf{x}}$ be the closures of regular neighborhoods of the
flowlines of $f$ through the points $\mathbf{x}$, which connect index 1
critical points to index 2 critical points and the flowlines through the
$w_i$, which connect index 3 critical points to index 0 critical points.
Then the gradient vector field $\nabla f$ does not vanish on $Y \backslash
N_{\mathbf{x}}$ and defines a $\mathrm{spin}^{\mathrm{c}}$ structure $\mathfrak s$ on $Y$.

The complex $\widehat{\mathit{CFK}}(\mathcal D, \mathfrak s)$ also carries
a (relative, for our purposes) homological grading called the Maslov
grading and, when $K$ is nullhomologous in $Y$, an additional grading
known as the Alexander grading. Suppose $\mathbf{x}$ and $\mathbf{y}$ are connected by a Whitney disk $\phi$. Then the relative gradings are determined by
\begin{align*}
M(\mathbf{x}) - M(\mathbf{y}) &= \mu(\phi) - 2\sum_i n_{w_i}(\phi) \\
A(\mathbf{x}) - A(\mathbf{y}) &= \sum_i n_{z_i}(\phi) - \sum_i n_{w_i}(\phi).
\end{align*}
The relative Alexander grading may also be computed as the linking number
of $\gamma_{\mathbf{x}, \mathbf{y}}$ with $K$. If $Y$ is a rational
homology sphere, we may pin down this grading precisely by letting
$Y_0(K)$ be the manifold obtained by zero-surgery along $K$ and
$\underline{\mathfrak s}$ be the $\mathrm{spin}^{\mathrm{c}}$ structure
obtained by extending the $\mathrm{spin}^{\mathrm{c}}$ structure
associated to $\mathbf{x}$ over $Y_0(K)$. Choose $F$ a Seifert surface of
$K$ in $Y$ and let $\widehat{F}$ be the closed surface resulting from
capping off $F$ in $Y_0(K)$. In this case, $A(\mathbf{x}) = \langle c_1(\underline{\mathfrak s}), [\widehat{F}] \rangle$. There is also a formula for absolute Maslov gradings which is somewhat too complicated to discuss here.

The differential $\partial$ lowers the Maslov grading by one and preserves the Alexander grading. Therefore $\widehat{\mathit{CFK}}(\mathcal D)$ also splits along Alexander grading in each $\mathrm{spin}^{\mathrm{c}}$ structure. A further useful technical tool is the $\delta$ grading, defined as the difference between the Maslov and Alexander gradings.

The homology of $\widehat{\mathit{CFK}}(\mathcal D)$ with respect to the differential $\partial$ is very nearly the knot Floer homology of $(Y,K)$. There is, however, a slight subtlety having to do with the number of pairs of basepoints $z_i$ and $w_i$ on $\mathcal D$. Let $V$ be a vector space over $\mathbb F_2$ with generators in gradings $(M,A) = (0,0)$ and $(M,A) = (-1,-1)$, and suppose $\mathcal D$ carries $n$ pairs of basepoints.

\begin{definition}
The homology of the complex $\widehat{\mathit{CFK}}(\mathcal D)$ with respect to the differential $\partial$ is
\[
\widetilde{\HFK}(\mathcal D) = \widehat{\HFK}((Y,K)) \otimes V^{\otimes (n-1)}.
\]

\end{definition}

To explicate the mysterious appearance of the vector space $V$, consider
that we can eliminate any pair of basepoints $z_i$ and $w_i$ (or
$w_{i-1}$) by adding to the Heegaard surface a tube connecting a small
neighborhood of one basepoint to a small neighborhood of the other. The
resulting surface $D'$ with the same $\alpha$ and $\beta$ curves as
previously is a Heegaard diagram for $(Y \# (S^1 \times S^2)), K')$, where
$K'$ is a knot running over the new tube instead of between the two former
basepoints and otherwise identical to $K$. If $t_0$ is the single torsion
$\mathrm{spin}^{\mathrm{c}}$ structure on $\#^{n-1} (S^1 \times S^2)$,
then $\widetilde{\HFK}(\mathcal D, s) = \widehat{\HFK}(Y
\# (\#^{n-1} S^1 \times S^2), K, s \# t_0)$ (see Ozsv\'ath and
Szab\'o~\cite[Theorem 4.5]{MR2443092}).

Perutz has shown~\cite[Theorem 1.2]{MR2509747} that there is a symplectic
form $\omega$ on $\mathrm{Sym}^{g+n-1}(S)$ which is compatible with the
induced complex structure, and with respect to which the submanifolds
$\mathbb T_{\alpha}$ and $\mathbb T_{\beta}$ are in fact Lagrangian and
the various Heegaard Floer homology theories are their Lagrangian Floer
cohomologies. In particular, the knot Floer homology is the Floer
cohomology of these two tori in the ambient space $\mathrm{Sym}^{g+n-1}(S
\backslash \{\mathbf{w},\mathbf{z}\})$, where the removal of the basepoints accounts for the restriction that holomorphic curves not be permitted to intersect the submanifolds $V_{w_i}$ and $V_{z_j}$ of the symmetric product.

\begin{proposition}
There is a symplectic structure on $\mathrm{Sym}^{g-n-1}(S^3 \backslash
\{\mathbf{w},\mathbf{z}\})$ with respect to which the submanifolds $\mathbb
T_{\boldsymbol \alpha}$ and $\mathbb T_{\boldsymbol \beta}$ are Lagrangian and
\[
\widetilde{\HFK}(D) \cong \widehat{\HFK}(S^3,K)\otimes V^{\otimes (n-1)}
\cong \mathit{HF}(\mathbb T_{\boldsymbol \beta}, \mathbb T_{\boldsymbol \alpha}).
\]

\end{proposition}

This is essentially~\cite[Theorem~1.2]{MR2509747} adjusted for Heegaard diagrams with multiple basepoints. From now on we will work with this method of computing knot Floer homology; in \fullref{SymplecticGeometrySection} we will show that the symplectic form produced by Perutz's construction meets the requirements of Seidel and Smith's theorem.

\subsection{Heegaard diagrams of double branched covers}

Consider the double branched cover $\Sigma(K)$ of the three-manifold $Y$ over the knot $K$; that is, the unique complex manifold with an involution $\tau \co \Sigma(K) \rightarrow \Sigma(K)$ such that the quotient of $\Sigma(K)$ by the action of $\tau$ is $S^3$, and such that if $\pi \co \Sigma(K) \rightarrow Y$ is the quotient map, $\pi^{-1}(K)$ is exactly the set of fixed points of $\tau$.
One way of constructing this manifold is to choose a Seifert surface $F$ of $K$ and remove a bicollar $F \times [-1,1]$ from $Y$. We then take two copies of $Y \backslash (F \times [-1,1])$ and identify the positive side of the bicollar in the one copy with the negative side in the other.

Suppose $f \co (Y,K)\rightarrow [0,3]$ is a self-indexing Morse function
with respect to which $K$ is a collection of flowlines between critical
points of index zero and index three.  Let $\mathcal D$ be the Heegaard
surface for $(Y,K)$ constructed from $f$.  That is, $\mathcal D$ consists
of the surface $S = f^{-1}(\frac{3}{2})$, curves $\boldsymbol{\alpha} =
\{\alpha_1,\ldots,\alpha_{g+n-1}\}$ at the intersection of the ascending
manifolds of critical points of index one with $S$, curves $\boldsymbol{\beta} =
\{\beta_1,\ldots,\beta_{g+n-1}\}$ at the intersection of the descending
manifolds of critical points of index two with $S$, and basepoints
$\mathbf{w} = (w_1,\ldots,w_n)$ (resp. $\mathbf{z}=(z_1,\ldots,z_n)$) at the
negatively (resp. positively) oriented points of $S \cup K$.  Now consider
the map $\wtilde{f} = f \circ \pi$, which is also self-indexing Morse
since $\pi$ is proper.  From $\wtilde{f}$ we obtain a Heegaard diagram
$\wtilde{\mathcal D}$ for $(\Sigma(K), K)$ which has surface $\wtilde{S} =
\pi|_{S}^{-1}(S)$, the branched double cover of $S$ over the basepoints
$\{\mathbf{w},\mathbf{z}\}$.  Moreover, each $\alpha_i$ lifts to two closed curves
$\wtilde{\alpha}_i$ and $\tau(\wtilde{\alpha}_i)$, each of which is the
attaching circle of a one-handle in $\Sigma(K)$, and similarly for the
$\beta$ curves and two-handles. Let $\wtilde{\boldsymbol{\alpha}} =
\{\wtilde{\alpha}_1,\tau(\wtilde{\alpha}_1),\ldots
,\wtilde{\alpha}_{g+n-1}, \tau(\wtilde{\alpha}_{g+n-1}))$ and
likewise for $\wtilde{\boldsymbol{\beta}}$. This leads us to the following lemma.

\begin{lemma}
If $\mathcal D=(S, \boldsymbol \alpha, \boldsymbol \beta, \mathbf{w},\mathbf{z})$ is a weakly admissible Heegaard surface for $(Y,K)$, then
$\wtilde{\mathcal D} = (\Sigma(S), \wtilde{\boldsymbol{\alpha}},
\wtilde{\boldsymbol{\beta}}, \mathbf{w},\mathbf{z})$ is a weakly admissible Heegaard surface for $(\Sigma(K), K)$.
\end{lemma} 

\begin{proof}
The only thing left to check is admissibility. Yet if there is a two-chain
$F$ in $\Sigma(S)$ with boundary some collection of the curves in
$\wtilde{\boldsymbol{\alpha}}$ and $\wtilde{\boldsymbol{\beta}}$ with only
positive (or only negative) local multiplicities, then $\pi(F)$ is a
two-chain in $S$ with boundary some of the curves in $\boldsymbol \alpha$
and $\boldsymbol \beta$ with only positive (or negative) local
multiplicities. Hence $\wtilde{\mathcal D}$ is weakly admissible if $\mathcal D$ is.\end{proof}

The generators of $\widehat{\mathit{CFK}}(D)$ have been studied by Grigsby~\cite{MR2253451} and Levine~\cite{MR2443111}; we give a quick sketch of their proofs of the following lemmas before proceeding to discuss the Heegaard diagrams we will use in this paper.

\begin{lemma}[Levine~{\cite[Lemma 3.1]{MR2443111}}]
Any generator $\mathbf{x}$ of $\widehat{\mathit{CFK}}( \wtilde{\mathcal D})$ admits a non-unique 
decomposition as
$\wtilde{\mathbf{x}}_{1}\wtilde{\mathbf{x}}_{2}$ where
$\wtilde{\mathbf{x}}_{i}$ is a lift of a generator $\mathbf{x}_i \in \widehat{\mathit{CFK}}(\mathcal D)$ for $i=1,2$. 
\end{lemma}

\begin{proof}
Consider the image of $\mathbf{x}$ under the natural map
$\mathrm{Sym}^{2(g+n-1)}(\Sigma(S)) \to \mathrm{Sym}^{2(g+n-1)}(S)$; this
is a collection of $2(g+n-1)$ points in $S$ such that each $\alpha$ circle
and each $\beta$ circle contains precisely two points. We can partition
the image into two subsets each of which is a generator in
$\widehat{\mathit{CFK}}(\mathcal D)$. To see this, begin with some $x_0$
on $\alpha_{i_1}$ and $\beta_{j_1}$ in $S$. Construct an oriented
one-cycle on $S$ as follows: moving along $\beta_{j_1}$ to the second
point in the image of $\mathbf{x}$ on that curve, which must also lie on some $\alpha_{i_2}$ (where $i_1$ and $i_2$ are not necessarily distinct). Move along $\alpha_{i_2}$ to the second point on that curve, and so on. Eventually this process terminates at $x_0$, which must be reached along $\alpha_{i_1}$. Hence there are an even number of edges in the resulting one-cycle on $S$, and we have collected an even number of vertices along the way. Put those vertices sitting at the start of a portion of this one-cycle lying on a $\beta$ curve in one set labelled $B$ and those vertices sitting at the start of a portion of this one-cycle lying on an $\alpha$ curve in another labelled $A$. Then each $\alpha$ and each $\beta$ curve which contributes an arc to the one-cycle contains exactly one point in $A$ and one point in $B$. Choose a vertex not yet assigned to $A$ or $B$ and repeat. This choice of partition is not at all unique.
\end{proof}

Of particular interest are the generators of the form
$\wtilde{\mathbf{x}}\tau(\wtilde{\mathbf{x}})$ in $\widehat{\mathit{CFK}}(\wtilde{\mathcal
D})$; that is, the generators which consist of all lifts of the points of
a generator $\mathbf{x}$ in $\widehat{\mathit{CFK}}(\mathcal D)$. These
points are exactly the invariant set of the induced involution $\tau^{\#}$
on $\widehat{\mathit{CFK}}(\wtilde{\mathcal D})$.

\begin{lemma}[Grigsby~{\cite[Proposition 3.2]{MR2253451}}]
All generators of $\widehat{\mathit{CFK}}(\wtilde{\mathcal D})$ of the
form $\wtilde{\mathbf{x}}\tau(\wtilde{\mathbf{x}})$ are in the same $\mathrm{spin}^{\mathrm{c}}$ structure, hereafter denoted $\mathfrak s_0$ and called the \textit{canonical $\mathrm{spin}^{\mathrm{c}}$ structure} on the double branched cover.  
\end{lemma}

\begin{proof}
Given two such generators $\wtilde{\mathbf{x}}\tau(\wtilde{\mathbf{x}})$ and
$\wtilde{\mathbf{y}}\tau(\wtilde{\mathbf{y}})$, let
$\gamma_{\mathbf{x},\mathbf{y}}$ be a
one-cycle in $S$ connecting $\mathbf{x}$ and $\mathbf{y}$ chosen as previously
and $\wtilde{\gamma}_{\mathbf{x},\mathbf{y}}$ be any lift to $\Sigma(S)$. Then
$\wtilde{\gamma}_{\mathbf{x},\mathbf{y}} +
\tau^{\#}(\wtilde{\gamma}_{\mathbf{x},\mathbf{y}})$ is
a suitable one-cycle running from
$\wtilde{\mathbf{x}}\tau(\wtilde{\mathbf{x}})$
to $\wtilde{\mathbf{y}}\tau(\wtilde{\mathbf{y}})$. Moreover, since $\tau^{\#}$
acts by multiplication by $-1$ on
$$H_1(\Sigma(Y)) \cong
\frac{H_1(\Sigma(S))}{\langle [\wtilde \alpha_1],
[\tau(\wtilde{\alpha}_1)],\ldots ,[\wtilde{\beta}_{g+n-1}],
[\tau(\wtilde{\beta}_{g+n-1})] \rangle},$$
the image
$\epsilon(\wtilde{\mathbf{x}}\tau(\wtilde{\mathbf{x}}),
\wtilde{\mathbf{y}}\tau(\wtilde{\mathbf{y}}))$ of
$\wtilde{\gamma}_{\mathbf{x},\mathbf{y}} +
\tau^{\#}(\wtilde{\gamma}_{\mathbf{x},\mathbf{y}})$ in $H^1(\Sigma(Y))$ is trivial.\end{proof}

At this juncture we pause to discuss the action of the induced involution
$\tau^{\#}$ on $\mathrm{spin}^{\mathrm{c}}$ structures on
$\widehat{\HFK}(\Sigma(K),K)$. The $\mathrm{spin}^{\mathrm{c}}$
structures on $\Sigma(K)$ are an affine copy of $H^2(\Sigma(K)) \cong
H_1(\Sigma(K))$; setting $\mathfrak s_0 = 0$ removes the ambiguity of the
identification between the set of $\mathrm{spin}^{\mathrm{c}}$ structures
on $\Sigma(K)$ and $H^2(\Sigma(K))$. Moreover, we shall see that for
suitable choice of Heegaard diagram $\mathcal D$, including both the
spherical bridge diagrams used in this paper and the toroidal grid
diagrams of Levine~\cite{MR2443111}, and additionally any $\mathcal D$
which is nice in the sense of Sarkar and Wang~\cite{MR2630063},
$\tau^{\#}$ is a chain map. The induced involution $\tau^*$ which on the
first homology of $\Sigma(K)$ acts by multiplication by $-1$, as does
conjugation of $\mathrm{spin}^{\mathrm{c}}$ structures on
$H_1(\Sigma(K))$. Ergo $\tau^*(\mathfrak s) = \wbar{\mathfrak s}$.
Thus the action  of $\tau$ to $\wtilde{\mathcal D}$ induces an isomorphism
\[
\widehat{\HFK}(\Sigma(K), K, \mathfrak s) \cong \widehat{\HFK}(\Sigma(K),
K, \wbar{\mathfrak s}).
\]
In particular, the action of $\tau$ on
$\widehat{\mathit{CFK}}(\wtilde{\mathcal D})$ preserves the canonical $\mathrm{spin}^{\mathrm{c}}$ structure.

\begin{lemma}
If $\mathbf{x} = \wtilde{\mathbf{x}}_{1}\wtilde{\mathbf{x}}_{2}$, the
Alexander grading of $\mathbf{x}$ is the average of the Alexander gradings
of $\mathbf{x}_1$ and $\mathbf{x}_2$. 
\end{lemma}

We refer the reader to Levine~\cite[Lemma 3.4]{MR2443111} for a lovely proof of this lemma.

We are now ready to consider the equivariant knot Floer homology of a
double branched cover of $S^3$ over a knot. Let $\pi \co (\Sigma(K),K)
\rightarrow (Y,K)$ be the branched double cover map and $\tau \co
\Sigma(K) \rightarrow \Sigma(K)$ be the involution interchanging the two
not necessarily distinct preimages of a point $x \in Y$. Assume that for
our particular Heegaard diagram $\mathcal D$, the map $\tau^{\#}$ is a
chain map on $\widehat{\mathit{CFK}}(\wtilde{\mathcal D})$.

\begin{definition}
The equivariant knot Floer homology $\widehat{\HFK}_{\mathrm{borel}}(\Sigma(K),K)$ is given by
\[
\widetilde{\HFK}_{\mathrm{borel}}(\wtilde{\mathcal D}) = \widehat{\HFK}_{\mathrm{borel}}(\Sigma(K),K) \otimes V^{\otimes (n-1)}
\]
where $\widetilde{\HFK}_{\mathrm{borel}}(\mathcal D)$ is the
homology of $\widehat{\mathit{CFK}}(\wtilde{\mathcal D}) \otimes \mathbb Z_2\llbracket q\rrbracket$ with respect to the differential $\partial + (1 + \tau^{\#})q$.

\end{definition}
The spectral sequence derived from the double complex
\begin{align*}
\xymatrix{
0 \ar[r] \ar[d] & \widehat{\mathit{CFK}}_{i+1}(\wtilde{\mathcal D})
\ar[d]^-{\partial} \ar[r]^-{1 + \tau^{\#}}  &
\widehat{\mathit{CFK}}_{i+1}(\wtilde{\mathcal D}) \ar[d]^-{\partial}
\ar[r]^-{1 + \tau^{\#}} & \widehat{\mathit{CFK}}_{i+1}(\wtilde{\mathcal D}) \ar[d]^-{\partial} \cdots \\
0 \ar[r] \ar[d] & \widehat{\mathit{CFK}}_i(\wtilde{\mathcal D})
\ar[d]^-{\partial} \ar[r]^-{1 + \tau^{\#}}  &
\widehat{\mathit{CFK}}_i(\wtilde{\mathcal D}) \ar[d]^-{\partial}
\ar[r]^-{1 + \tau^{\#}} & \widehat{\mathit{CFK}}_i(\wtilde{\mathcal D}) \ar[d]^-{\partial} \cdots \\
0 \ar[r] & \widehat{\mathit{CFK}}_{i-1}(\wtilde{\mathcal D}) \ar[r]^-{1 +
\tau^{\#}}  & \widehat{\mathit{CFK}}_{i-1}(\wtilde{\mathcal D}) \ar[r]^-{1
+ \tau^{\#}} & \widehat{\mathit{CFK}}_{i-1}(\wtilde{\mathcal D}) \cdots \\
}
\end{align*}
which converges to
$\widetilde{\HFK}_{\mathrm{borel}}(\wtilde{\mathcal D})$ has been a source
of interest for some time; a popular conjecture has been that its
$E^{\infty}$ page is isomorphic modulo torsion to
$\widetilde{\HFK}(\mathcal D) \otimes \mathbb Z_2\llbracket q\rrbracket$. We will show a
similar statement, namely \fullref{existsspecseq}. The $E^1$ page of this
spectral sequence (after computing the homology of the vertical
differentials) is $(\widetilde{\HFK}(\wtilde{\mathcal D}) \otimes V^{\otimes (n-1)}) \otimes \mathbb Z_2\llbracket q\rrbracket$, and an application of \fullref{Smith Inequality} will show that after tensoring with $\mathbb Z\llparen q\rrparen$, the $E^{\infty}$ page of the spectral sequence is isomorphic to $(\widetilde{\HFK}(\mathcal D) \otimes V^{\otimes (n-1)}) \otimes \mathbb Z\llparen q\rrparen$.

Notice that since both $\partial$ and $\tau^{\#}$ preserve the Alexander
grading, the spectral sequence splits along the Alexander grading; it
moreover splits along pairs of conjugate $\mathrm{spin}^{\mathrm{c}}$
structures. Ergo it is interesting to consider not only the full spectral
sequence but also its restriction to $\widehat{\mathit{CFK}}(\Sigma(K), K,
\mathfrak s_0, i)$ for any Alexander grading $i$. Since the localization
maps of Seidel and Smith~\cite[Section 2c]{MR2739000} referenced in
\fullref{SeidelSmith} are defined by counting holomorphic disks between
equivariant and nonequivariant generators of $\CF(L_0,L_1)
= \CF(\mathbb T_{\beta}, \mathbb T_{\alpha})$ and by
multiplication and division by $q$, the Alexander grading on
$\widetilde{\HFK}(\wtilde{\mathcal D})$ is preserved by the
isomorphism of \fullref{existsspecseq}.

In order to apply \fullref{Smith Inequality} to the case of $(S^3,K)$ and
its double branched cover $(\Sigma(K), K)$ we will require a Heegaard
diagram $D$ for $(S^3, K)$ lying on the sphere $S^2$. Choose a bridge
presentation of $K$ in $S^2$; that is, a diagram of $K$ in $S^2 = \mathbb
R^2 \cup \{\infty\}$ such that there are a finite number of line segments
$b_1,\ldots ,b_n$ in the image of $K$ in the plane such that at every
crossing in $K$ the overcrossing arc is a portion of the $b_i$ and neither
of the undercrossing arcs are. Distribute basepoints $\mathbf{w} =
(w_1,\ldots ,w_n)$ and $\mathbf{z} = (z_1,\ldots ,z_n)$ along the image of $K$ in the bridge presentation at the endpoints of the line segments $b_i$ such that as one moves along $K$ in the direction of the orientation starting with the line segment $b_1$, these basepoints are encountered in the order $w_1, z_1, w_2, z_2,\ldots ,w_n, z_n$. For $1 \leq i \leq n-1$, let $\beta_i$ be a closed curve in the plane encircling whichever bridge $b_j$ has endpoints $w_i$ and $z_i$, and let $\alpha_i$ be a closed curve in the plane encircling the arc of $K$ which contains none of the bridges $b_j$ and has endpoints $z_i$ and $w_{i+1}$. Both sets of curves will be oriented counterclockwise with respect to their interiors in the plane $S^2 \backslash \{z_{n}\}$.

\begin{figure}[ht!]
\begin{center}
\includegraphics{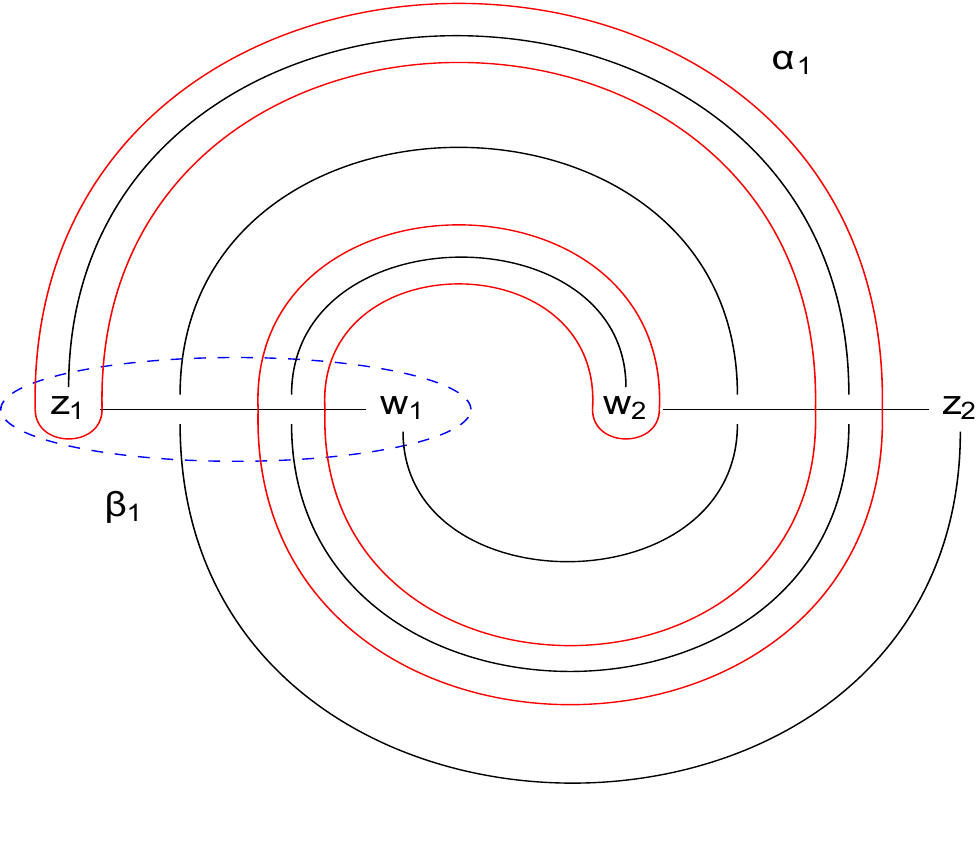}
\caption{A Heegaard diagram on the sphere derived from a two-bridge presentation of the trefoil.}
\end{center}
\end{figure}

Choose complex coordinates on $S^2 \backslash \{\mathbf{w},\mathbf{z}\}$ and use these
to induce complex charts on $\Sigma(S^2) \backslash \{\mathbf{w},\mathbf{z}\}$ and
subsequently on the symmetric products $\mathrm{Sym}^{n-1}(S^2 \backslash
\{\mathbf{w},\mathbf{z}\})$ and $\mathrm{Sym}^{2n-2}(\Sigma(S^2) \backslash
\{\mathbf{w},\mathbf{z}\})$. (Later on we will want to be a little more precise about this
original choice, but for now we allow ourselves considerable latitude.) If
$j$ is the almost complex structure on $S^2 \backslash \{\mathbf{w},\mathbf{z}\}$,
then $\pi^*j = \wtilde{j}$ is the almost complex structure on $\Sigma(S^2)
\backslash \{\mathbf{w},\mathbf{z}\}$ and $\mathrm{Sym}^{n-1}(j)$ and
$\mathrm{Sym}^{2n-2}(\wtilde{j})$ are the almost complex structures on the two symmetric products.

For $x_i$ a point on $S^2 \backslash \{\mathbf{w},\mathbf{z}\}$, let
$\wtilde{x}_i$ denote one of the two preimages of $x_i$ on
$\Sigma(S^2)\backslash\{\mathbf{w},\mathbf{z}\}$. Consider the map
\begin{align}
i \co \mathrm{Sym}^{n-1}(S^2 \backslash \{\mathbf{w},\mathbf{z}\}) & \hookrightarrow  \mathrm{Sym}^{2n-2}(\Sigma(S^2)\backslash \{\mathbf{w},\mathbf{z}\}) \label{embedding} \\
(x_1x_2 \ldots x_{n-1}) &\mapsto (\wtilde{x}_1\tau(\wtilde{x}_1) \ldots
\wtilde{x}_{n-1}\tau(\wtilde{x}_{n-1})) \nonumber
\end{align}
This map is a holomorphic embedding; for a proof of this fact on charts, see Appendix 1. The involution $\tau$ interchanges the two (not necessarily distinct) preimages of a point under the restriction of the double branched cover map $\pi \co \Sigma(K) \rightarrow S^3$.

We make the following suggestive choices of notation: let
$L_0 = \mathbb{T}_{\wtilde{{\boldsymbol \beta}}}$, let
$L_1 =\mathbb{T}_{\wtilde{{\boldsymbol \alpha}}}$, and let
$M = \mathrm{Sym}^{2n-2}(\Sigma(S^2)\backslash
\{\mathbf{z},\mathbf{w}\})$. In \fullref{SymplecticGeometrySection}
we will show that $M$ is convex at
infinity and can be equipped with an exact $\tau$--invariant symplectic
form with respect to which $\mathbb T_{\boldsymbol \alpha}$ and $\mathbb
T_{\boldsymbol \beta}$ are exact Lagrangian submanifolds.

The following is immediate from the definitions.

\begin{lemma}

With respect to the induced involution on $M$ (also called $\tau$), we have the following invariant sets.
\begin{align*}
M^{\inv} &= \mathrm{Sym}^{2n-2}(\Sigma(S^2) \backslash
\{\mathbf{z},\mathbf{w}\})^{\inv} = i(\mathrm{Sym}^{n-1}(S^2
\backslash \{\mathbf{z},\mathbf{w}\}))\\
\smash{L_0^{\inv}} &= \mathbb{T}_{\wtilde{{\boldsymbol
\beta}}}^{\inv} = i(\mathbb{T}_{{\boldsymbol \beta}}) \\
\smash{L_1^{\inv}} &= \mathbb{T}_{\wtilde{{\boldsymbol
\alpha}}}^{\inv} = i(\mathbb{T}_{{\boldsymbol \alpha}}).
\end{align*}
\end{lemma}

\begin{corollary}
With respect to our choice of symplectic manifold $M$ and Lagrangians $L_0$ and $L_1$, we have the following Floer homology groups. 
\begin{align*}
\mathit{HF}(L_0, L_1) &= \mathit{HF}\big(\mathbb{T}_{\wtilde{{\boldsymbol
\beta}}}, \mathbb{T}_{\wtilde{{\boldsymbol \alpha}}}\big)\\&=
\widetilde{\HFK}(\wtilde{D})\\ &= \widehat{\HFK}((\Sigma(K),K)) \otimes V^{\otimes (n-1)} \\
\mathit{HF}(\smash{L_0^{\inv}}, \smash{L_1^{\inv}}) &=
\mathit{HF}\big(\mathbb{T}_{\wtilde{{\boldsymbol \beta}}}^{\inv},
\mathbb{T}_{\wtilde{{\boldsymbol \alpha}}}^{\inv}\big) \\ &=
\mathit{HF}(\mathbb{T}_{{\boldsymbol \beta}}, \mathbb{T}_{{\boldsymbol \alpha}}) \\ &= \widetilde{\HFK}(D) \\ &= \widehat{\HFK}(S^3,K) \otimes V^{\otimes (n-1)}.
\end{align*}
\end{corollary}

As before, let $\Upsilon(M^{\inv}) \rightarrow M^{\inv} \times [0,1]$ be the pullback of the normal bundle to $M^{\inv}$ in $M$ along the projection map $M^{\inv} \times [0,1] \rightarrow M^{\inv}$. Let $N(L_i^{\inv}) \times \{t\}$ for $i=0,1$ be the copy of the Lagrangian normal bundle to $L_i^{\inv}$ in $L_i$ sitting above the subspace $L_i^{\inv} \times \{t\} \subset M^{\inv} \times [0,1]$.

Our goal will be to prove the following.

\begin{theorem} \label{stable}

The bundle $\Upsilon(M^{\inv})$ carries a stable normal trivialization with respect to the involution $\tau$.

\end{theorem}

The proof is given in \fullref{StableTrivSection}.

\fullref{stable}, combined with convexity at infinity of $M$ and the
existence of an exact $\tau$--invariant form on $M$ with respect to which
$\mathbb T_{\boldsymbol \alpha}$ and $\mathbb T_{\boldsymbol \beta}$ are exact Lagrangians, implies \fullref{existsspecseq}. We deduce a rank inequality between $\widetilde{\HFK}(\Sigma(K), K, \mathfrak s_0)$ and $\widetilde{\HFK}(S^3, K)$. But since each of these Heegaard diagrams contains $n$ pairs of basepoints, we obtain the rank inequality in \fullref{main}. Moreover, our previous remarks concerning the splitting of the spectral sequence along $\mathrm{spin}^{\mathrm{c}}$ structures and Alexander gradings will then imply Corollaries~\ref{sharper} and~\ref{Alexander corollary}.

\section{Symplectic geometry of $M$ and $M^{\inv}$} \label{SymplecticGeometrySection}

Thus far we have not shown that the complex manifold $M =
\mathrm{Sym}^{2n-2}(\Sigma(S^2) \backslash \{\mathbf{z},\mathbf{w}\})$ and its totally
real submanifolds $L_0 = \mathbb T_{\wtilde{\boldsymbol \beta}}$ and $L_1
= \mathbb T_{\wtilde{\boldsymbol \alpha}}$ satisfy the basic symplectic structural requirements imposed by Seidel and Smith's theory. Before we move to the more complex task of demonstrating that $\Upsilon(M^{\inv})$ carries a stable normal trivialization, we pause to show that $(M, L_0, L_1)$ can be equipped with a symplectic form such that $M$ is exact and convex at infinity and $L_0$, $L_1$ are exact Lagrangian submanifolds of $M$.

Regard the punctured sphere $S^2 \backslash \{\mathbf{w},\mathbf{z}\}$ as a
$(2n-1)$--punctured plane
$$\mathbb C \backslash \{\mathbf{w},z_1,\ldots ,z_{n-1}\}.$$
We can insist that our original choice of complex coordinates
on $S^2$, which we used to induce complex coordinates on $\Sigma(S^2)$ and
subsequently all of the relevant symmetric products, was compatible with
this embedding. Let $x=u+iv$ be a point on $\mathbb C$, and let $\phi_0$
be a smooth function on $S^2 \backslash \{\mathbf{w},\mathbf{z}\}$ defined as follows.
\begin{align*}
\phi_0 \co \mathbb C \backslash \{\mathbf{w},z_1,\ldots ,z_{n-1}\}  &\rightarrow \mathbb R \\
x &\mapsto |x|^2 + \sum_{i=1}^{n}\frac{1}{|x-w_i|^2}+  \sum_{i=1}^{n-1} \frac{1}{|x-z_i|^2}
\end{align*}
We claim this map is $j$--convex, where $j$ is as before the complex structure on the punctured sphere. The corresponding symplectic form is
\begin{align*}
\omega_{\phi_0} = -dd^{\mathbb C}(\phi) = \bigg(4 +
\sum_{i=1}^{n}\frac{4}{|x-w_i|^4} +\sum_{i=1}^{n-1}
\frac{4}{|x-z_{i}|^4}\bigg) du\wedge dv.
\end{align*}
This is compatible with the almost complex structure on the punctured
plane, so $\phi_0$ is a $j$--convex function on $S^2 \backslash
\{\mathbf{w},\mathbf{z}\}$.

The smooth map $\phi_0$ will not quite be our final choice of $j$--convex
function on the punctured plane, since we would like each $\alpha$ and
$\beta$ circle on $S^2 \backslash \{\mathbf{w},\mathbf{z}\}$ to be an
exact Lagrangian with respect to the our choice of symplectic form on $S^2
\backslash \{\mathbf{w},\mathbf{z}\}$. We will replace $\phi_0$ with a closely related $j$--convex function $\phi_1$ and slightly isotope the $\alpha$ and $\beta$ circles with the result that $\int_{\alpha_i} -d^{\mathbb C} \phi_1 = \int_{\beta_j} -d^{\mathbb C} \phi_1 = 0$. Then the restriction of the form $-d^{\mathbb C}\phi_1$ to either an $\alpha$ or $\beta$ curve will be exact as desired.

We begin by observing that since the curves $\alpha_i$ and $\beta_i$ were all chosen to be oriented counterclockwise in the plane formed by deleting $z_{n}$, if $\psi(z) = |z|^2$ then the integral of $-d^{\mathbb C}(\psi) = 2x dy - 2ydx$ around one of the attaching circles is four times the area enclosed by that circle in the plane, and in particular is strictly positive. Therefore there is some large constant $C$ such that if we replace the function $\phi_0$ with $\phi_1$ as below, each of $\int_{\alpha_i} -d^{\mathbb C}(\phi_1)$ and $\int_{\beta_j} -d^{\mathbb C}(\phi_1)$ is nonnegative. Let
\begin{align*}
\phi_1 \co \mathbb C \backslash \{\mathbf{w},z_1, \ldots ,z_{n-1}\}  &\rightarrow \mathbb R \\
x &\mapsto C|x|^2 + \sum_{i=1}^{n}\frac{1}{|x-w_i|^2} +\sum_{i=1}^{n-1} \frac{1}{|x-z_i|^2}.
\end{align*}
This is still a $j$--convex function by the same argument as for $\phi_0$. Now for each attaching circle, a look at the construction shows there is a straight line segment from the circle to one of the punctures it encircles which intersects no other attaching circles. We can isotope any of the attaching circles inward along these arcs without changing the intersection points of the $\alpha$ and $\beta$ curves or altering any of our computations concerning the cohomology of this space and its symmetric product. Since $\omega_{\phi_1}$ is an area form on the punctured plane, by Stokes' theorem changing one of the attaching circles in this fashion subtracts the area of the small counterclockwise-oriented region removed from the interior of the attaching circle from the integral of $-d^{\mathbb C}\phi_1$ around the curve. Since $\phi_1$ goes to infinity along the arc connecting the attaching circle to the puncture, we can make this area arbitrarily large, and in fact choose it precisely so that the integral of $-d^{\mathbb C}(\phi_1)$ along our isotoped attaching circle is zero. From now on we will assume we have performed such an isotopy and that each $\alpha$ and $\beta$ curve is an exact Lagrangian with respect to $\omega_{\phi_1}$. We take $\omega_{\phi_1}$ to be our choice of exact symplectic form on the punctured sphere.

Now let $\pi\co\Sigma(S^2) \backslash \{\mathbf{w},\mathbf{z}\} \rightarrow S^2
\backslash \{\mathbf{w},\mathbf{z}\}$ be the restriction of the double branched
covering map, and consider the map $\phi_1 \circ \pi$. We claim that this
map is $\wtilde{j}$--convex: it is clearly smooth and bounded below, and
since $\pi$ is a branched covering map, hence proper, $\phi_1 \circ \pi$
is also proper. Moreover, as $\pi$ is holomorphic by definition,
$d^{\mathbb C}(\phi_1 \circ \pi) = \pi^*(d^{\mathbb C}(\phi_1))$.
Therefore $\omega_{\phi_1 \circ \pi} = \pi^*(\omega_{\phi_1})$ is
compatible with the complex structure on $\Sigma(S^2) \backslash
\{\mathbf{w},\mathbf{z}\}$ and is $\wtilde{j}$--convex. We let $\wtilde{\phi}_1 = \phi_1
\circ \pi$ and take $\omega_{\wtilde{\phi}_1}$ to be our choice of
symplectic form on $\Sigma(S)\backslash \{\mathbf{w},\mathbf{z}\}$. Notice that the
lifts of the $\alpha$ and $\beta$ curves in the punctured sphere to the
$\alpha$ and $\beta$ curves in $\Sigma(S^2) \backslash \{\mathbf{w},\mathbf{z}\}$ are
necessarily exact Lagrangian with respect to this $\omega_{\wtilde{\phi}_1}$.

On the product space $(\Sigma(S^2) \backslash \{\mathbf{w},\mathbf{z}\})^{2n-2}$ there
is a corresponding $\mathrm{Sym}^{2n-2}(\wtilde{j})$--convex function: if
$p_k$ is the projection of the product space to its $k$th factor, then let
$\wtilde{\phi}\co (\Sigma(S^2) \backslash \{\mathbf{w},\mathbf{z}\})^{2n-2}$ be given
by $\wtilde{\phi} = \wtilde{\phi}_1 \circ p_1 + \cdots + \wtilde{\phi}_1
\circ p_{2n-2}$. This is proper and bounded below; moreover,
$$\omega_{\wtilde{\phi}} = -dd^{\mathbb C}(\wtilde{\phi}) =
\omega_{\wtilde{\phi}_1} \otimes 1 \otimes \cdots \otimes 1 + \cdots + 1
\otimes \cdots \otimes 1\otimes \omega_{\wtilde{\phi}_1},$$
which is
necessarily compatible with the induced complex structure on the product
space. We take $\omega_{\wtilde{\phi}}$ to be the symplectic form on the
product $(\Sigma(S^2) \backslash \{\mathbf{w},\mathbf{z}\})^{2n-2}$, noting that with
respect to this form, any of the $(2n-2)!$ lifts of $\mathbb
T_{\wtilde{\boldsymbol \alpha}}$ or $\mathbb T_{\wtilde{\boldsymbol \beta}}$ to the product space is an exact Lagrangian submanifold of the product.

Now consider the symmetric product $\mathrm{Sym}^{2n-2}(\Sigma(S^2)
\backslash \{\mathbf{w},\mathbf{z}\})$. The map $\wtilde{\phi}$ induces a (possibly singular) continuous function 
\begin{align*}
\psi \co \mathrm{Sym}^{2n-2}(\Sigma(S)\backslash \{\mathbf{w},\mathbf{z}\}) &\rightarrow \mathbb R \\
(x_1 \cdots x_{2n-2}) &\mapsto \sum_{\sigma \in S_{2n-2}} \wtilde{\phi}(x_{\sigma(1)},\ldots ,x_{\sigma(2n-2)})
\end{align*}
which is proper, bounded below, and smooth outside a
neighborhood of the fat diagonal $\{(x_1 \cdots x_{2n-2}) \in
\mathrm{Sym}^{2n-2}(S^2 \backslash \{\mathbf{w},\mathbf{z}\}): x_i = x_j
\text{ for some } i\neq j\}$. Perutz observes in~\cite{MR2509747} that this function
is strictly plurisubharmonic in the sense of non-smooth functions, that
is, that the two-current $-dd^{\mathbb C}\psi$ is strictly positive. This
gives us a continuous exhausting function on
$\mathrm{Sym}^{2n-2}(\Sigma(S^2) \backslash \{\mathbf{w},\mathbf{z}\})$.  We may apply
the following lemma of Richberg~\cite{MR0222334}, quoted by Cielebak and
Eliashberg~\cite[Lemma 3.10]{SteinBook}.

\begin{lemma}
Let $\psi$ be a continuous $J$--convex function on an integrable complex manifold $(V,J)$. Then for every positive function $h \co V \rightarrow \mathbb R_+$, there exists a smooth $J$--convex function $\psi'$ such that $|\psi'(x) - \psi(x)| < h(x)$. If $\phi$ is already smooth on a neighborhood of a compact subset $A$, then we can achieve $\phi = \phi'$ on $A$.
\end{lemma}

In particular, we may take $h\co\text{Sym}^{2n-2}(\Sigma(S^2) \backslash
\{\mathbf{w},\mathbf{z}\})\rightarrow \mathbb R_+$ to be a constant function $h(x) =
\epsilon$, and apply the lemma to our map $\psi$ and $h$.  Then we may
produce $\psi'\co \mathrm{Sym}^{2n-2}(\Sigma(S^2) \backslash
\{\mathbf{w},\mathbf{z}\}) \rightarrow \mathbb R$ such that $|\psi'(x) - \psi(x)|<\epsilon$,
and $\psi'$ is smooth and $\mathrm{Sym}^{2n-2}(\wtilde{j})$--convex.  Moreover, since $\psi$ is bounded below and proper, $\psi'$ is also, since the two real valued functions differ by at most $\epsilon$.  Therefore $\psi'$ is an exhausting function on $M$, and $M$ is convex at infinity.

We have yet to produce an appropriate symplectic form on $M$, which we will do using work of Perutz~\cite{MR2509747}. We begin with a definition.  


\begin{definition}[Perutz~{\cite[Definition~7.3]{MR2509747}}]
Let $X$ be a complex manifold with complex structure $J$. A \emph{K\"ahler
cocycle} on X is a collection $(U_i, \phi_i)_{i \in I}$, where $(U_i)_{i
\in I}$ is an open cover of $X$ and $\phi_i \co U_i \rightarrow \mathbb R$
is an upper semicontinuous function such that
\begin{itemize}
\item $\phi_i$ is strictly plurisubharmonic, and
\item $\phi_i - \phi_j$ is pluriharmonic.
\end{itemize}
\end{definition} 

If a K\"ahler cocycle $(U_i, \phi_i)_{i \in I}$ is smooth then we can
associate to it the symplectic form $\omega$ which is $-dd^{\mathbb
C}\phi_i$ on each $U_i$. Notice, for example, that a K\"ahler cocycle can
consist of a single plurisubharmonic function on all of $X$, as in the
case of the smooth K\"ahler cocycle $((\Sigma(S) \backslash
\{\mathbf{w},\mathbf{z}\})^{2n-2}, \wtilde{\phi})$ on $(\Sigma(S)
\backslash \{\mathbf{w},\mathbf{z}\})^{2n-2}$ and the singular K\"ahler cocycle $(M, \psi)$ on $M$.

Perutz proves the following technical result.

\begin{lemma}[Perutz~{\cite[Lemma~7.4]{MR2509747}}] Let $(U_i, \psi_i)$ be a continuous K\"ahler cocycle on a complex manifold $X$. Suppose that $X = X_1 \cup X_2$ such that $X_1$ and $X_2$ are open and the functions $\psi_i|_{U_i \cap X_{\ell}}$ are smooth. Then there exists a continuous function
\[
\chi \co X \rightarrow \mathbb R,\quad \Supp(\chi) \subset X_2
\]
and a locally finite refinement
\[
V_j \subset U_{i(j)} 
\]
such that the family $(V_i, \psi_{i(j)}|_{V_j} + \chi|_{V_j})$ is a smooth K\"ahler cocycle.
\end{lemma}

Notice that if $(U_i, \psi_i)$ happened to be the K\"ahler cocycle associated to a single $J$--convex function $\psi$ on $X$, then $(V_i, \psi_{i(j)}|_{V_j} + \chi|_{V_j})$ is the K\"ahler cocycle associated to the smooth plurisubharmonic function $\psi + \chi$.

In our particular case we take $X$ to be $\mathrm{Sym}^{2n-2}(\Sigma(S)
\backslash \{\mathbf{w},\mathbf{z}\})$, $X_1$ to be the complement of the main
diagonal in this symmetric product, and $X_2$ to be a small neighborhood
of the main diagonal with no intersection with $\mathbb
T_{\wtilde{\boldsymbol \alpha}}$ and $\mathbb T_{\wtilde{\boldsymbol
\beta}}$. Then the function $\psi \co \mathrm{Sym}^{2n-2}(\Sigma(S^2)
\backslash \{\mathbf{w},\mathbf{z}\}) \rightarrow \mathbb R$ admits a smoothing to a
$\mathrm{Sym}^{2n-2}(\tilde{j})$--convex function $\psi + \chi \co
\mathrm{Sym}^{2n-2}(\Sigma(S^2) \backslash \{\mathbf{w},\mathbf{z}\}) \rightarrow
\mathbb R$ which is equal to $\psi$ away from a neighborhood of the large
diagonal. The symplectic form $\omega_{\psi+ \chi}$ is exact and
compatible with the complex structure on $M$. Finally, on a neighborhood
of $\mathbb T_{\wtilde{\boldsymbol \alpha}}$ the map $\chi$ is identically
$0$ and $\psi = (2n-2)!\wtilde{\phi}|_{\wtilde{\alpha}_1 \times
\tau(\wtilde{\alpha}_1) \times \cdots \times \wtilde{\alpha}_{n-1} \times
\tau(\wtilde{\alpha}_{n-1})}$. Therefore $\omega_{\psi+ \chi}|_{\mathbb
T_{\wtilde{\boldsymbol \alpha}}} = 0$ and $d^{\mathbb C}(\psi +
\chi)|_{\mathbb T_{\wtilde{\boldsymbol \alpha}}} = (2n-2)!d^{\mathbb
C}(\wtilde{\phi})|_{\wtilde{\alpha}_1 \times \cdots \times
\wtilde{\alpha}_{n-1} \times \tau(\wtilde{\alpha}_{n-1})}$ is exact. Ergo
$\mathbb T_{\wtilde{\boldsymbol \alpha}}$ is an exact Lagrangian in the
exact symplectic manifold $M$, and similarly $\mathbb T_{\wtilde{\boldsymbol \beta}}$ is as well.

The reader may at this point be alarmed that we have failed thus far to
check that $\tau$ is a symplectic involution. We'll make one final
alteration to the symplectic form on $M$ such that this is the case.
Because our original $\tilde{j}$--convex function $\wtilde{\phi}_1$ on
$\Sigma(S) \backslash \{\mathbf{w},\mathbf{z}\}$ was the pullback of a $j$--convex
function $\phi_1$ on $S^2 \backslash \{\mathbf{w},\mathbf{z}\}$, we see that
$\omega_{\wtilde{\phi}_1}$ is certainly $\tau$--invariant. Following our
construction of the continuous singular plurisubharmonic function $\psi$
on $M$, we see $\psi$ is invariant with respect to the induced involution
$\tau$ on $M$. Since $\omega_{\psi + \chi} = \omega_{\psi}$ away from
$X_2$ a neighborhood of the large diagonal in $M$, $\omega_{\psi + \chi}$
is $\tau$--invariant away from $X_2$. We replace $\omega_{\psi + \chi}$
with $\omega = \frac{1}{2}(\omega_{\psi + \chi} + \tau^*\omega_{\psi +
\chi})$; this exact form is $\mathrm{Sym}^{2n-2}(\wtilde{j})$--compatible
and nondegenerate since $\tau$ is holomorphic. Moreover, $\omega =
\omega_{\psi+\chi}$ away from $X_2$, implying that the two tori $\mathbb
T_{\wtilde{\boldsymbol \alpha}}$ and $\mathbb T_{\wtilde{\boldsymbol \beta}}$ remain exact Lagrangian submanifolds with respect to $\omega$. The form $\omega$ is our final choice of symplectic form on $M$.

\section{Homotopy type and cohomology of $M$ and $M^{\inv}$} \label{HomotopyCohomologySection}

In order to show that our setting can be made to satisfy the hypotheses of \fullref{Smith Inequality}, we will need to begin with a good grip on the homotopy type of $M^{\inv}$. This in turn will allow us to say something concrete concerning the cohomology ring of $M^{\inv} \times [0,1]$ and its relationship to the cohomology ring of $\smash{L_0^{\inv}} \times  \{0\} \cup \smash{L_1^{\inv}} \times \{1\}$.

We begin with homotopy type. Notice that both $S^2\backslash\{\mathbf{w},
\mathbf{z}\}$ and $\Sigma(S^2)\backslash\{\mathbf{w},\mathbf{z}\}$ are
$(2n-2)$--punctured surfaces of genus $0$ and $n-1$, respectively, and
therefore homotopy equivalent to the wedge product of $2n-1$ and $4n-3$
circles, respectively. As the operation of taking the symmetric product
preserves homotopy equivalence (see, for example, Hatcher~\cite[Section 4.K]{MR1867354}), it suffices to determine the homotopy type of a finite symmetric product of a wedge of circles $(S^1)^{\vee m}$. It will turn out that this space deformation retracts onto a subspace of a torus $(S^1)^m$. Let us start by establishing some notation. Give $S^1$ the coordinates of the unit circle in the complex plane, and let 1 be its basepoint and unique zero-cell. If $I$ is a subset of $\{1, \ldots ,m\}$, let $(S^1)^I$ be the subspace of $(S^1)^{m}$ defined as follows.
\begin{align*}
(S^1)^I = \{(x_1,\ldots ,x_m) \in (S^1)^{r}: x_i = 1 \text{ if } i
\not\in I\}
\end{align*}

\begin{lemma} \label{subtori}
For $r \geq m$, the symmetric product $\mathrm{Sym}^r((S^1)^{\vee m})$ has the homotopy type of an $m$--torus $(S^1)^m$. For $r < m$, the symmetric product $\mathrm{Sym}^r((S^1)^{\vee m})$ has the homotopy type of the union of the $m \choose r$ canonical subtori of $(S^1)^m$,
\begin{align*}
\bigcup_{|I| = r} (S^1)^I \subset (S^1)^m.
\end{align*}
If the torus $(S^1)^m$ is given the usual product CW structure this space is the $r$--skeleton of the CW complex.
\end{lemma}

\begin{proof}
We essentially follow the argument given by Ong~\cite{MR1993792}. Let $\sigma_1, \ldots ,\sigma_r$ be the first $r$ elementary symmetric functions of $r$ variables $x_1, \ldots ,x_r$, such that 
\begin{align*}
\sigma_j(x_1,\ldots ,x_r) = \sum_{1\leq i_1 < \cdots <i_j \leq r} x_{i_1} \ldots x_{i_j}.
\end{align*}
There is a well-known biholomorphism between $\mathrm{Sym}^{r}(\mathbb C)$ and $\mathbb C^r$ by evaluating each of the functions $\sigma_j$ for $1\leq j\leq r$ on an unordered set of complex numbers $x_1,\ldots ,x_r$. This has the effect of mapping to the ordered coefficients of the monic polynomial with roots $x_1,\ldots ,x_r$, multiplied by an alternating sign.
\begin{equation}
\begin{aligned} 
\phi \co \mathrm{Sym}^r(\mathbb C) &\rightarrow \mathbb C^r \\
(x_1 \ldots x_r) &\mapsto (\sigma_1(x_1, \ldots ,x_r), \ldots , \sigma_r(x_1, \ldots ,x_r)).
\end{aligned}
\label{biholomorphism}
\end{equation}
Under the map $\phi$ the submanifold $\mathrm{Sym}^r(\mathbb C^*)$ is carried to $\mathbb C^{r-1} \times \mathbb C^*$. We can use this map to construct the following homotopy equivalence between $S^1$ and $\mathrm{Sym}^r(S^1)$ for any $r$.
$$\smash{\xymatrix{
\mathrm{Sym}^r(S^1) \ar@{^(->}[r] & \mathrm{Sym}^r(\mathbb C^*) \ar[rr]^-{\phi|_{\mathrm{Sym}^r(\mathbb C^*)}}  && \mathbb C^{r-1} \times \mathbb C^*  \ar[r] & \mathbb C^* \ar[r] & S^1 
}}$$
Here the first inclusion map is a homotopy equivalence and the final two
maps are deformation retractions. The total map carries $\smash{(e^{i
\theta_1} \ldots e^{i\theta_r})} \in \mathrm{Sym}^r(S^1)$ to
$\smash{\sigma_r((e^{i \theta_1}, \ldots
,e^{i\theta_r}))}$, which is exactly the product of the entries,
$\smash{e^{i(\theta_1+ \cdots +\theta_r)}}$. Ergo the total map from $\mathrm{Sym}^r(S^1)$ to $S^1$ which multiplies the entries of an unordered $r$--tuple of points on the circle is a homotopy equivalence. Indeed, since this homotopy equivalence can be regarded as a retract from $\mathrm{Sym}^{r}(S^1)$ to its subspace $A_r = \{(e^{i\theta_1}e^{i\theta_2} \ldots e^{i\theta_r}) \co \theta_2 = \cdots = \theta_r = 0\} = S^1 \times \{1\}^{r-1}$, there must be a deformation retraction $F^r$ from $\mathrm{Sym}^r(S^1)$ to this subspace.

It will be useful to be slightly more careful concerning our choice of
deformation retraction. Suppose $\mathrm{Sym}^{r-1}(S^1)$ is regarded as a
subspace $\mathrm{Sym}^{r-1} \times \{1\}$ of $\mathrm{Sym}^r(S^1)$ via
the embedding $(e^{i\theta_1} \ldots e^{i\theta_{r-1}}) \mapsto
(e^{i\theta_1} \ldots e^{i\theta_{r-1}} e^{i0})$. Then both
$\mathrm{Sym}^{r}(S^1)$ and $\mathrm{Sym}^{r-1}(S^1) \times \{1\}$
deformation retract onto the subspace $A_r$ in $\mathrm{Sym}^r(S^1)$,
implying that the relative homotopy groups
$\pi_i(\mathrm{Sym}^{r}(S^1),\mathrm{Sym}^{r-1}(S^1) \times \{1\})$ are
trivial. Hence since all the spaces involved carry CW structures induced
by the CW structure on $S^1$ and the inclusion $\mathrm{Sym}^{r-1}(S^1)
\hookrightarrow \mathrm{Sym}^r(S^1)$ is cellular, there is a deformation
retraction $F^{r,0}$ from $\mathrm{Sym}^r(S^1)$ to
$\mathrm{Sym}^{r-1}(S^1) \times \{1\}$, which can be taken to run on a
time interval $\smash{\big[0, \frac{1}{r-1}\big]}$. By similar logic for $r\geq k
\geq 2$, there is a deformation retraction $F^{r,r-k}$ from
$\mathrm{Sym}^{k}(S^1) \times \{1\}^{r-k}$ to $\mathrm{Sym}^{k-1}(S^1)
\times \{1\}^{r-k+1}$ whose time input can be taken to be
$\smash{\big[\frac{r-k}{r-1}, \frac{r-k+1}{r-1}\big]}$. If we take $F^r$ to be the
map $\mathrm{Sym}^r(S^1) \times [0,1] \rightarrow \mathrm{Sym}^r(S^1)$
which is  $F^{r,r-k}$ on $\smash{\big[\frac{r-k}{r-1},
\frac{r-k+1}{r-1}\big]}$, then $F^r$ is a deformation retraction from $\mathrm{Sym}^r(S^1)$ to $A_r = \mathrm{Sym}^1(S^1) \times \{1\}^{r-1}$ which preserves each of the subspaces $\mathrm{Sym}^{k} \times \{1\}^{r-k}$.

We can now deal with the symmetric product of a wedge of circles. To keep track of which points originate in which circle, label the circles $S^1_1, \ldots ,S^1_m$ and refer to points $e^{i\theta}_j$ on $S^1_j$. Take the wedge point to be the basepoint $1=e^{i 0}$ on each circle.

Regard the space $\mathrm{Sym}^r\big(\bigvee_{i=1}^{m}S^1_i\big)$ as the
subspace $\mathrm{Sym}^{r}\big(\bigvee_{i=1}^{m}S^1_i\big) \times
\{1\}^{r(m-1)}$ of $\mathrm{Sym}^{rm}\big(\bigvee_{i=1}^{m}S^1_i\big)$.
Any point in $\mathrm{Sym}^r\big(\bigvee_{i=1}^m S^1_i\big) \times
\{1\}^{r(m-1)}$ may be split uniquely into a point in
$\mathrm{Sym}^r(S^1_1) \times \mathrm{Sym}^r(S^1_2)\times \cdots \times
\mathrm{Sym}^r(S^1_m)$, as there are at most $r$ terms from any circle
$S^1_i$ in an $mr$--tuple in $\mathrm{Sym}^r\big(\bigvee_{i=1}^m S^1_i\big) \times \{1\}^{r(m-1)}$. Consider applying $F^r \co \mathrm{Sym}^r(S^1_1) \times [0,1]$ to this submanifold; that is, consider the following map.
\begin{align*}
F^r_1 \co \mathrm{Sym}^r(S^1_1){\times}\cdots{\times}\mathrm{Sym}^r(S^1_m)
{\times}[0,``1]  &\rightarrow
\mathrm{Sym}^r(S^1_1){\times}\cdots{\times}\mathrm{Sym}^r(S^1_m) \\
\big(\big(e_1^{i\theta_1}{\ldots}e_1^{i\theta_r}\big),\ldots ,\big(e_m^{i
\theta_1}{\ldots}
e_m^{i \theta_r}\big)\big){\times}[0,``1]  &\mapsto
\big(F^r\big(e_1^{i\theta_1}\ldots e_1^{i\theta_r}, t\big), \ldots , \big(e_m^{i
\theta_1}{\ldots}e_m^{i \theta_r}\big)\big).
\end{align*}
Because $F^r$ never increases the number of nonbasepoint terms in
$\big(e_1^{i\theta_1}\ldots e_1^{i\theta_r}\big)$, the map $F^r_1$
preserves the subspace $\mathrm{Sym}^{r}\big(\bigvee_{i=1}^{m}S^1_i\big) \times \{1\}^{r(m-1)}$ in
$$\mathrm{Sym}^r(S^1_1) \times \mathrm{Sym}^r(S^1_2)\times \cdots \times
\mathrm{Sym}^r(S^1_m).$$
Ergo $F^r_1$ is a deformation retraction from
$\mathrm{Sym}^{r}\big(\bigvee_{i=1}^{m}S^1_i\big)$ to the subspace of
$r$--tuples in $\mathrm{Sym}^{r}\big(\bigvee_{i=1}^{m}S^1_i\big)$ of which
at most one entry is a nonbasepoint point on $S^1_1$. Running this
procedure on each factor of $\mathrm{Sym}^r(S^1_1) \times
\mathrm{Sym}^r(S^1_2)\times \cdots \times \mathrm{Sym}^r(S^1_m)$ produces
a deformation retraction from
$\mathrm{Sym}^{r}\big(\bigvee_{i=1}^{m}S^1_i\big)$ to the subspace of $r$--tuples in this space containing only one nontrivial point of each $S^1_j$ for $1\leq j \leq r$. So we have the homotopy equivalence
$$\mathrm{Sym}^{r}\bigg(\bigvee_{i=1}^{m}S^1_i\bigg) \simeq
\big\{\big(e_1^{i\theta_1},\ldots ,e_m^{i\theta_m}\big)
: \text{ at most } r~\theta_i \text{ are nonzero}\big\}
\subset S^1_1\times \cdots \times S^1_m.$$
If $r>m$, this space is all of $S_1^1 \times \cdots \times S^1_m =
(S^1)^m$. If $r<m$, this space is exactly $\bigcup_{I=|r|}(S^1)^I \subset
(S^1)^m$.
\end{proof}

We are now ready to discuss cohomology, and in particular to prove an important relationship between the cohomology of $M^{\inv} \times [0,1]$ and its subspace $(\smash{L_0^{\inv}} \times \{0\}) \cup (\smash{L_1^{\inv}} \times \{1\})$. Consider the inclusion map
$$i_1 \co (\mathbb T_{\boldsymbol \beta} \times \{0\}) \cup (\mathbb
T_{\boldsymbol \alpha} \times \{1\}) \rightarrow \mathrm{Sym}^{n-1}(S^2
\backslash \{\mathbf{w},\mathbf{z}\}) \times [0,1].$$

\begin{proposition} \label{cohomologysurjection}
The cohomology pullback
\begin{align*} \label{cohomologysurjectionmap}
i_1^* \co H^*(\mathrm{Sym}^{n-1}(S^2 \backslash \{\mathbf{w},\mathbf{z}\}) \times
[0,1]) \rightarrow H^*((\mathbb T_{\boldsymbol \beta} \times \{0\}) \cup
(\mathbb T_{\boldsymbol \alpha} \times \{1\}))
\end{align*} 
induced by inclusion map $i_1$ is a surjection for $*\geq 1$.
\end{proposition}

\begin{proof}
We will see in this proof that the homology (and hence the
cohomology) groups of $M^{\inv} \times [0,1] =
\mathrm{Sym}^{n-1}(S^2 \backslash \{\mathbf{w},\mathbf{z}\}) \times [0,1]$ and of
$(\mathbb T_{\boldsymbol \beta} \times \{0\}) \cup (\mathbb T_{\boldsymbol
\alpha} \times \{1\})$ are all free abelian. Under these circumstances it
suffices to work with a set of generators for $H_k(M^{\inv}) =
H_k(M^{\inv} \times [0,1])$ as an abelian group and a set of
generators for $H_k((\mathbb T_{\boldsymbol \beta} \times \{0\}) \cup
(\mathbb T_{\boldsymbol \alpha} \times \{1\}))$ as an abelian group. 

The following commutative square is natural, and both of the horizontal arrows are isomorphisms, since there is no torsion in the cohomology of any of the spaces involved.
\begin{align*}
\xymatrix{
H^k(M^{\inv} \times [0,1]) \ar[d]^-{i_1^*} \ar[r] & \mathrm{Hom}(H_k(M^{\inv} \times [0,1]), \mathbb Z) \ar[d]^-{(i_1)_*^T} \\
H^k((\mathbb T_{\boldsymbol \beta} \times \{0\}) \cup (\mathbb T_{\boldsymbol \alpha} \times \{1\})) \ar[r] &
\mathrm{Hom}(H_k((\mathbb T_{\boldsymbol \beta} \times \{0\}) \cup (\mathbb
T_{\boldsymbol \alpha} \times \{1\})), \mathbb Z)
}
\end{align*}
If we can show that the homology push-forward $(i_1)_*$ is an injection
mapping a generator $\zeta$ of the $k$th homology of the subspace to a
generator $\kappa$ of the $k$th homology of $M^{\inv}$, then the
cohomology pullback $\smash{i_1^*}$ maps the dual $\what{\kappa}$ of $\kappa$
in $H^k(M^{\inv} \times [0,1])$ to the dual $\what{\zeta}$ of
$\zeta$ in $H^k((\mathbb T_{\boldsymbol \beta} \times \{0\}) \cup (\mathbb
T_{\boldsymbol \alpha} \times \{1\}))$.

Let the curves $\{\alpha_1,\ldots ,\alpha_{n-1}\}$, $\{\beta_1,\ldots
,\beta_{n-1}\}$ on the punctured sphere $S^2 \backslash
\{\mathbf{w},\mathbf{z}\}$ be as in \fullref{HeegaardFloerSection}. Choose
a parametrization of each of the $\alpha$ and $\beta$ curves such that
$\alpha_j \co [0,1] \rightarrow S^2 \backslash \{\mathbf{w},\mathbf{z}\}$
is a one-cycle in the punctured sphere, and similarly for $\beta_j \co
[0,1] \rightarrow S^2 \backslash \{\mathbf{w},\mathbf{z}\}$. Moreover,
choose a set of one-cycles $\nu_i \co [0,1] \rightarrow S^2 \backslash
\{\mathbf{w},\mathbf{z}\}$ for $1 \leq i \leq 2n$ such that the image of
$\nu_{2j-1}$ is a small circle oriented counterclockwise around $w_j$ and
the image of $\nu_{2j}$ is a small circle around $z_j$. Then $[\beta_j] =
[\nu_{2j-1}] + [\nu_{2j}]$ and $[\alpha_j] = [\nu_{2j}] + [\nu_{2j+1}]$ in
$H_1(S^2 \backslash \{\mathbf{w},\mathbf{z}\})$.

\begin{figure}[ht!]
\begin{center}
\includegraphics{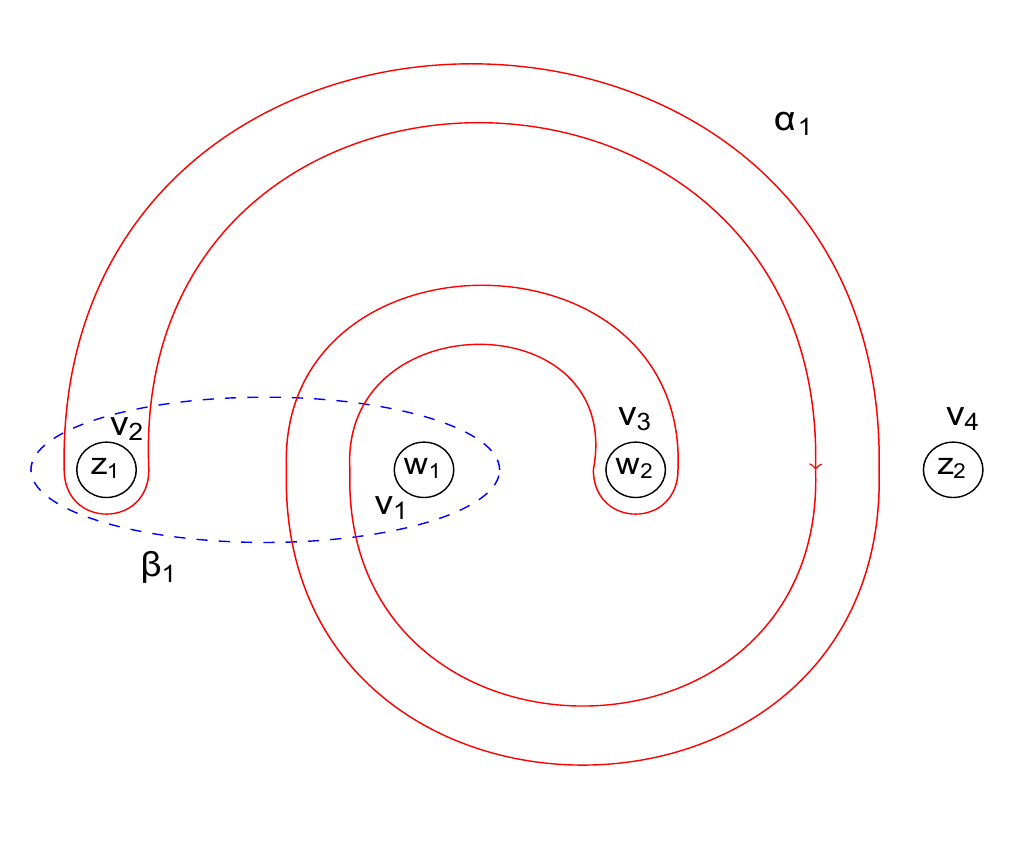}
\caption{The Heegaard diagram of Figure 1 with the knot deleted.  The
curves $\nu_1$, $\nu_2$, and $\nu_3$ generate the first homology of $S^2
\backslash \{\mathbf{w},\mathbf{z}\}$.}
\end{center}
\end{figure}

Choose a set of one-cycles $\{\nu_1',\ldots ,\nu_{2n-1}'\}$ in $S^2
\backslash \{\mathbf{w},\mathbf{z}\}$ such that the image of each $\nu_i'$
is an oriented circle homologous to $\nu_i$, and moreover the union of the
images of the $\nu_i$ is a wedge of circles. We abuse notation slightly by
referring to this wedge as $\bigvee_{i=1}^{2n-1} \nu_i'$. Then $S^2
\backslash \{\mathbf{w},\mathbf{z}\}$ admits a deformation retraction onto
the wedge of circles $\bigvee_{i=1}^{2n-1} \nu_i'$. This deformation
retraction of the punctured sphere onto a wedge of circles induces a
deformation retraction of the space $M^{\inv} = \mathrm{Sym}^{n-1}(S^2
\backslash \{ \mathbf{w},\mathbf{z}\})$ onto
$\mathrm{Sym}^{n-1}(\bigvee_{i=1}^{2n-1} \nu_i')$. However, since
$\bigvee_{i=1}^{2n-1} \nu_i'$ is a wedge of circles, from
\fullref{subtori} we observe that
$\mathrm{Sym}^{n-1}\big(\bigvee_{i=1}^{2n-1} \nu_i'\big)$ admits a
deformation retraction onto the $(n{-}1)$st skeleton of the torus
$\smash{\prod_{i=1}^{2n-1}} \nu_i'$. In particular, the first homology of this space is generated by the homology classes of the following one-cycles.%
$$\xymatrix@R=.5cm@C-0.1cm{
\wbar{\nu}_i' \co [0,1] \ar[r] & \prod_{i=1}^{2n-1} \nu_i'
\ar@{^(->}[r] & \mathrm{Sym}^{n-1}\big(\bigvee_{i=1}^{2n-1} \nu_i'\big) \ar@{^(->}[r] & M^{\inv} \\
t \ar@{|->}[r]& (\nu_i'(t),x_0, \ldots ,x_0) \ar@{|->}[r]& (\nu_i'(t)x_0 \ldots x_0) \ar@{|->}[r]& (\nu_i'(t)x_0 \ldots x_0)
}$$
Hence $H_1(M^{\inv}) = \mathbb Z \langle [\wbar{\nu}_1'],\ldots
,[\wbar{\nu}_{2n-1}'] \rangle$ and by the Kunneth formula $H_k(M^{\inv})$
has a basis consisting of tensor products of $k$ distinct one-cycles
$[\wbar{\nu}_i']$ for $1 \leq k \leq n-1$. Through a slight abuse of notation, we say $H_k(M^{\inv}) = \textstyle{\bigwedge^{k} H_1(M^{\inv})}$, although it does not have the product structure of the exterior algebra. This will not, however, be the most convenient presentation of this group. Let us introduce the following one-cycles in $M^{\inv}$, for $1 \leq i \leq n-1$.
\begin{align*}
\wbar{\alpha}_i \co [0,1] &\rightarrow M^{\inv} &
\wbar{\beta}_i \co [0,1] & \rightarrow M^{\inv} \\
t &\mapsto (\alpha_i(t)x_0\ldots x_0) &
t & \mapsto (\beta_i(t)x_0 \ldots x_0)
\end{align*}
Recall that in $S^2 \backslash \{\mathbf{w},\mathbf{z}\}$, $\beta_i$ is homologous to $\nu_{2i-1} + \nu_{2i}$, and hence to $\nu_{2i-1}' + \nu_{2i}'$. Similarly, $\alpha_i$ is homologous to $\nu_{2i} + \nu_{2i+1}$, and hence to $\nu_{2i}' + \nu_{2i+1}'$. Therefore in $H_1(M^{\inv})$ we have the following equalities.
\begin{align*}
[\wbar{\beta}_i] &= [\wbar{\nu}_{2i-1}'] + [\wbar{\nu}_{2i}'] \\
[\wbar{\alpha}_i] &= [\wbar{\nu}_{2i}'] + [\wbar{\nu}_{2i+1}']
\end{align*}
Therefore we can rewrite the homology of $M^{\inv}$ as follows.
\begin{align*}
H_1(M^{\inv}) &= \mathbb Z \langle [\wbar{\nu}_1'],\ldots ,[\wbar{\nu}_{2n-1}'] \rangle\\
&=\mathbb Z \langle [\wbar{\alpha}_1],\ldots ,[\wbar{\alpha}_{n-1}],
[\wbar{\beta}_1],\ldots ,[\wbar{\beta}_{n-1}], [\wbar{\nu}_{2n-1}'] \rangle \\ 
H_k(M^{\inv}) &= \textstyle{\bigwedge^k  \mathbb Z} \langle
[\wbar{\nu}_1'],\ldots ,[\wbar{\nu}_{2n-1}'] \rangle\\
&=\textstyle{\bigwedge^k} \mathbb Z \langle [\wbar{\alpha}_1],\ldots
,[\wbar{\alpha}_{n-1}], [\wbar{\beta}_1],\ldots ,[\wbar{\beta}_{n-1}],
[\wbar{\nu}_{2n-1}'] \rangle
\end{align*}
Here as before $1 \leq k \leq n-1$. Using the isomorphism $H_k(M^{\inv}) \cong H_k(M^{\inv} \times [0,1])$ we will also regard this as a description of $H_k(M^{\inv} \times [0,1])$.

Next we consider $\mathbb T_{\boldsymbol \alpha} = \alpha_1 \times \cdots
\times \alpha_{n-1} \subset M^{\inv}$. Let
$\smash{\wtilde{\wbar{\alpha}}_i}$ denote the one-cycle in $\mathbb
T_{\boldsymbol \alpha}$ corresponding to the curve $\alpha_i$. To wit,
choose a basepoint $x_i$ on each $\alpha_i$ and let
$\smash{\wtilde{\wbar{\alpha}}_i}$ be defined as follows.
\begin{align*}
\wtilde{\wbar{\alpha}}_i \co [0,1] &\rightarrow \mathbb T_{\boldsymbol \alpha} \\
t &\mapsto (x_1,\ldots ,x_{i-1},\alpha_i(t),x_{i+1},\ldots ,x_{n-1})
\end{align*}
Then the first homology of $\mathbb T_{\boldsymbol \alpha}$ has generators
$\big[\wtilde{\wbar{\alpha}}_i\big]$ for $1 \leq i \leq n-1$ and we see that, for $1 \leq k \leq n-1$,
\begin{align*}
H_1(\mathbb T_{\boldsymbol \alpha}) &= \mathbb Z \big\langle
\big[\wtilde{\wbar{\alpha}}_1\big],\ldots
,\big[\wtilde{\wbar{\alpha}}_{n-1}\big] \big\rangle \\
H_k(\mathbb T_{\boldsymbol \alpha})&=\textstyle{\bigwedge^k} \mathbb Z
\big\langle \big[\wtilde{\wbar{\alpha}}_1\big],\ldots
,\big[\wtilde{\wbar{\alpha}}_{n-1}\big] \big\rangle 
\end{align*}
The story for $\mathbb T_{\boldsymbol \beta}$ is completely analogous; if
$y_i$ is a basepoint on $\beta_i$, then let
$\smash{\wtilde{\wbar{\beta}}_i}$ be the one-cycles
\begin{align*}
\wtilde{\wbar{\beta}}_i \co [0,1] &\rightarrow \mathbb T_{\boldsymbol \beta} \\
t &\mapsto (x_1,\ldots ,x_{i-1},\beta_i(t),x_{i+1},\ldots ,x_{n-1}).
\end{align*}
Then
\begin{align*}
H_1(\mathbb T_{\boldsymbol \beta}) &= \mathbb Z \big\langle
\big[\wtilde{\wbar{\beta}}_1\big],\ldots
,\big[\wtilde{\wbar{\beta}}_{n-1}\big]
\big\rangle \\
H_k(\mathbb T_{\boldsymbol \beta}) &= \textstyle{\bigwedge^k} \mathbb Z
\big\langle \big[\wtilde{\wbar{\beta}}_1\big],\ldots
,\big[\wtilde{\wbar{\beta}}_{n-1}\big] \big\rangle.
\end{align*}
Notice that we have now seen that $H_*(M^{\inv} \times
[0,1])$, $H_*(\mathbb T_{\boldsymbol \beta} \times \{0\}) = H_*(\mathbb
T_{\boldsymbol \beta})$, and $H_*(\mathbb T_{\boldsymbol \alpha} \times \{1\})
= H_*(\mathbb T_{\boldsymbol \alpha})$ are all free abelian as promised at the start of this proof.

We are now ready to consider the relationship of $\mathbb T_{\boldsymbol
\alpha}$ and $\mathbb T_{\boldsymbol \beta}$ to $M^{\inv}$ and prove that the homology push forward
\begin{align*}
(i_1)_* \co H_*((\mathbb T_{\boldsymbol \alpha} \times \{0\}) \cup (\mathbb
T_{\boldsymbol \beta} \times \{1\})) \rightarrow H_*(M^{\inv})
\end{align*}
is an injection and maps a generator on homology to a generator on homology for $* \geq 1$.

From our description of the one-cycles whose homology classes generate the
first homology of $M^{\inv} \times [0,1]$ and $(\mathbb
T_{\boldsymbol \alpha} \times \{0\}) \cup (\mathbb T_{\boldsymbol \beta}
\times \{1\})$, we see that the inclusion of
$\smash{\wtilde{\wbar{\alpha}}_i}$ into $M^{\inv} \times [0,1]$ is
$\wbar{\alpha}_i$ and likewise the inclusion of
$\smash{\wtilde{\wbar{\beta}}_i}$
into $M^{\inv}$ is $\wbar{\beta}_i$. Ergo for $1 \leq k \leq n-1$, $i_1$ induces an injective map on homology of the following form.
\begin{align*}
(i_1)_*\co H_{k}((\mathbb T_{\boldsymbol \beta} \times \{0\}) \cup (\mathbb
T_{\boldsymbol \alpha} \times \{1\})) &\rightarrow H_k(M^{\inv}) \\
\big(\textstyle{\bigwedge_{t=1}^{k_1}} [\wtilde{\wbar{\beta}}_{j_t}]\big)
\oplus \big(\textstyle{\bigwedge_{t=1}^{k_2}}
[\wtilde{\wbar{\alpha}}_{j_t'}]\big)
&\mapsto \textstyle{\bigwedge_{t=1}^{k_1}} [\wbar{\beta}_{j_t}] +
\textstyle{\bigwedge_{t=1}^{k_2}}[\wbar{\alpha}_{j_t'}] 
\end{align*}
Here $(j_1,\ldots ,j_{k_1})$ and $(j_1',\ldots ,j_{k_2}')$ are any two
collections of distinct integers between $1$ and $n{-}1$ inclusive, with
$k_1{+}k_2{=}k$. Since the generators of $H_k((\mathbb T_{\boldsymbol \alpha}
{\times}\{0\}) \cup (\mathbb T_{\boldsymbol \beta}{\times}\{1\}))$ as a
free abelian group are $\smash{\big\{\textstyle{\bigwedge_{t=1}^{k}}
\big[\smash{\wtilde{\wbar{\alpha}}_{j_t}}\big]\big\}} \cup \smash{\big\{\textstyle{\bigwedge_{t=1}^k}
\big[\wtilde{\wbar{\beta}}_{j'_t}\big]\big\}}$, whereas the generators of
$H^k(M^{\inv})$ include $\big\{\textstyle{\bigwedge_{t=1}^k}
[\wbar{\alpha}_{j_t}]\big\}$ and
$\big\{\textstyle{\bigwedge_{t=1}^k}[\wbar{\beta}_{j'_t}]\big\}$, the homology map $(i_1)_*$ on $H_k$ maps a generator to a generator injectively. Therefore the corresponding map on cohomology $i_1^*$ is a surjection.
\end{proof}

\begin{remark}
The map in \fullref{cohomologysurjection} is \textit{not} a surjection if only half the basepoints of the Heegaard diagram are removed; that is, for the inclusion map 
\begin{align*}
i_2 \co \mathbb (T_{\boldsymbol \beta} \times \{0\}) \cup (\mathbb
T_{\boldsymbol \alpha} \times \{1\}) \rightarrow \mathrm{Sym}^{n-1}(S^2
\backslash \{\mathbf{w}\}) \times [0,1]
\end{align*}
the cohomology pullback
\begin{align*}
i_2^*\co H^*(\mathrm{Sym}^{n-1}(S^2 \backslash \{\mathbf{w}\})) &\rightarrow
H^*((\mathbb T_{\boldsymbol \beta} \times \{0\}) \cup (\mathbb T_{\boldsymbol \alpha} \times \{1\})) \\ &= H^*((S^1)^{n-1})\oplus H^*((S^1)^{n-1})
\end{align*}
has image the diagonal of $H^*((S^1)^{n-1})\oplus H^*((S^1)^{n-1})$. This
is the primary reason that the argument of this paper does not extend to a
statement regarding $\widehat{\mathit{HF}}(\Sigma(K))$ and
$\widehat{\mathit{HF}}(S^3)$, although in a future paper we will be able to overcome this difficulty.
\end{remark}

\section{Important constructions from $K$--theory}

Central to our argument that the space $(M^{\inv},
\smash{L_0^{\inv}}, \smash{L_1^{\inv}})$ admits a stable normal
trivialization will be several useful results from complex $K$--theory,
most trivial but one rather less so. A detailed treatment of the subject,
along with proofs of all results up to
\fullref{Cherncharacterisomorphism}, may be found in Hatcher~\cite{HatcherKTheory}.

Let $V$ be a complex vector bundle over a base space $X$, which for purposes of this paper we take to be a compact Hausdorff topological space. We let $\mathrm{Vect}^{\mathbb C}_n(X)$ denote the set of isomorphism classes of $n$--dimensional vector bundles over $X$. This is a monoid under the direct sum of vector bundles. For any $n$, as before we refer to the trivial complex vector bundle of degree $n$ over $X$ simply as $\mathbb C^n$. Then there is an equivalence relation $\approx_S$ on $\mathrm{Vect}^{\mathbb C}_n(X)$ defined as follows: given $V,W$ two $n$--dimensional vector bundles over $X$, we say that $V \approx_S W$ if and only if there is some $m$ such that $V \oplus \mathbb C^m \cong W \oplus \mathbb C^m$. The two vector bundles $V$ and $W$ are said to be \textit{stably isomorphic}. This relation respects the direct sums and tensor products of vector bundles, such that the set of equivalence classes of bundles under $\approx_{S}$ inherits two abelian laws of composition $[V]_S + [W]_S = [V \oplus W]_S$ and $[V]_S \times [W]_S = [V \otimes W]_S$. This set of equivalence classes may be given the structure of a ring by formally adjoining the inverse of each element under the direct sum. More precisely, we set
\begin{align*}
K^0(X) = \big\{[V]_S - [W]_S: [V]_S, [W]_S \text{ are equivalence classes
with respect to } \approx_S\big\}.
\end{align*}
Commonly $[V]_S - [\mathbb C^0]_S$ will be written simply as $[V]_S$ and its additive inverse $[\mathbb C^0]_S - [V]_S$ simply as $-[V]_S$. 

\begin{lemma}
$K^0(X)$ is a ring with respect to the operations $[V]_S + [W]_S = [V \oplus W]_S$ and $[V]_S \times [W]_S = [V \otimes W]_S$.
\end{lemma}

There is also a reduced form of this ring $\wtilde{K}(X)$ constructed as follows. Let $\sim$ be a second equivalence relation on $\bigcup_{n \in \mathbb N} \mathrm{Vect}^{\mathbb C}_n$ such that $V \sim W$ if there is some $m_1,m_2$ such that $V \oplus \mathbb C^{m_1} \cong W \oplus \mathbb C^{m_2}$. In this case one can show that the set of equivalence classes with respect to $\sim$ contains additive inverses without adjoining any additional elements. Let the equivalence class of a vector bundle $V$ with respect to $\sim$ be $[V]$.

\begin{lemma}
$\wtilde{K}^0(X)$ is a ring with respect to the operations $[V] + [W] = [V \oplus W]$ and $[V] \times [W] = [V \otimes W]$.
\end{lemma}

Then $K^0(X) \cong \wtilde{K}^0(X) \oplus \mathbb Z$. In both cases a
vector bundle in the same equivalence class as $\mathbb C^m$ is said to be
\textit{stably trivial}. Notice as a most basic case that for $\{x_0\}$ a
one point space $K^0(x_0) = \mathbb Z$ and $\wtilde{K}^0(x_0) = 0$.

Given a continuous map $f \co X \rightarrow Y$, the corresponding map $f^*
\co \mathrm{Vect}^{\mathbb C}_n(Y) \rightarrow  \mathrm{Vect}^{\mathbb
C}_n(X)$ which maps a vector bundle $V$ over $Y$ to its pullback $f^*(V)$
descends to maps $f^* \co K^0(Y) \rightarrow K^0(Y)$ and $f^* \co
\wtilde{K}^0(Y) \rightarrow \wtilde{K}^0(X)$. Recalling that if maps $f,g
\co X \rightarrow Y$ are homotopic then $f^*(V)$ and $g^*(V)$ the two
pullbacks of a bundle $V$ over $Y$ are isomorphic, we see that homotopic
maps $f$ and $g$ induce the same maps $f^* = g^*$ on $K^0(Y)$ and $\wtilde{K}^0(Y)$. In particular, we have the following lemma.

\begin{lemma}
If $f \co X_1 \rightarrow X_2$ is a homotopy equivalence with homotopy inverse $g \co X_2 \rightarrow X_1$, the induced map $f^* \co K^0(X_2) \rightarrow K^0(X_1)$ is an isomorphism with inverse $g^*$. The same is true of the reduced theory.
\end{lemma}

We will use this lemma to deal with the minor problem that $M^{\inv}$ is
not actually compact. The three spaces whose $K$--theory will be of
interest to us are $M^{\inv} \times [0,1]$, its subspace $X =
\smash{L_0^{\inv}} \times \{0\} \cup \smash{L_1^{\inv}} \times \{1\}$, and
the quotient space $M^{\inv}/X$. The space $X$ is a disjoint union of two
tori, hence compact. We have seen previously that $M^{\inv}$ deformation
retracts onto a compact subspace homeomorphic to the $n-1$ skeleton of
$(S^1)^{2n-1}$. For purposes of dealing with $K^0(M^{\inv})$ let us choose
a slightly different deformation retraction. Let $Y = \big(\bigcup
\alpha_i\big) \cup \big(\bigcup \beta_i\big)$. Then let $F \co (S^2 \backslash
\{\mathbf{w},\mathbf{z}\}) \times [0,1] \rightarrow S^2 \backslash
\{\mathbf{w},\mathbf{z}\}$ be a deformation retraction from the punctured
sphere to the union of $Y$ and all components of $S^2 \backslash Y$ not
containing any point of $\{\mathbf{w},\mathbf{z}\}$. This is a compact
subspace of the punctured sphere; let it be $Z$. Then $F$ induces a
deformation retraction $\mathrm{Sym}^{n-1}(F)$ of $M^{\inv} =
\mathrm{Sym}^{n-1}(S^2 \backslash \{\mathbf{w},\mathbf{z}\})$ onto the
compact set $\mathrm{Sym}^{n-1}(Z)$, so we can legitimately refer to the
$K$--theory of $M^{\inv}$ by identifying it with
$K^0(\mathrm{Sym}^{n-1}(Z))$. Notice further that $F$ is the identity on
$Y \times I$, so the induced deformation retraction from $M^{\inv}$
preserves $X \subset \mathrm{Sym}^{n-1}(Y)$. Therefore we may produce a
deformation retract $\smash{\overline{\mathrm{Sym}^{n-1}(F)}}$ from $M^{\inv}/X$ to $\mathrm{Sym}^{n-1}(Z)/X$, and $\mathrm{Sym}^{n-1}(Z)/X$ is compact Hausdorff. So we identify the $K$--groups of $M^{\inv}/X$ with those of $\mathrm{Sym}^{n-1}(Z)/X$, and from now on make no further reference to this technical subtlety.

The relationship between homotopy classes of maps and pullbacks of vector bundles in fact gives us the following deeper proposition, which follows from the theorem that any $n$--dimensional complex vector bundle over $X$ is a pullback of the canonical $n$--dimensional bundle over the Grassmanian $G_n(\mathbb C^{\infty})$ along some homotopy class of maps $X \rightarrow G_n(\mathbb C^{\infty})$.

\begin{proposition}
$\wtilde{K}^0(X) \cong [X, BU]$.
\end{proposition}

Here $BU$ is the classifying space of the infinite unitary group.

We extend our definition of $\wtilde{K}^0(X)$ as follows. Let
$\wtilde{K}^{-i}(X) = \wtilde{K}^{0}(\Sigma^i(X))$, where $\Sigma(X)$ is (for this section only) the reduced suspension of $X$. Then we have the following.


\begin{proposition}[The Bott Periodicity Theorem]
$\wtilde{K}^0(X) \cong \wtilde{K}^0(\Sigma^2(X))$.
\end{proposition}

We use this two-periodicity to define $\wtilde{K}^i(X)$ for positive $i$.
Now notice that $\wtilde{K}^0(X)=K^0(X^+)$, where $X^+$ denotes the
one-point compactification of $X$ (when $X$ is compact, this will mean $X$
together with a point at infinity). In particular, $K^1(X) = \wtilde{K}^0(\Sigma(X))$.  Then the groups $K^i(X)$ are also doubly-periodic.

With $K^i(X)$ now defined for all $i \in \mathbb Z$, we are ready for the
following proposition, which deals with $K^*(X) = \bigcup_i K^i(X)$ merely as a collection of abelian groups and (momentarily) ignores their ring structure.

\begin{proposition} 
The groups $K^i(X)$ form a generalized cohomology theory on compact
topological spaces. Likewise, the groups $\wtilde{K}^i(X)$ form a reduced cohomology theory.
\end{proposition}

A relevant note is that if $\{x_0\}$ is a one point space, we have
\begin{align*}
K^i(\{x_0\}) = 
\begin{cases}
\mathbb Z \text{ if } i \text{ even } \\
0 \text{ if } i \text{ odd }.
\end{cases}
\end{align*}
One consequence of this structure on the $K$--theory of a space is that
there is a notion of the relative $K$--theory of a compact space $X$ and a
closed subspace $Y$. We let $K^0(X,Y)$ be the ring of isomorphism classes
of vector bundles over $X$ which restrict to a trivial bundle over $Y$.
Equivalently, $K^0(X,Y) \cong \wtilde{K}^0(X/Y)$; notice that if $X$ is
compact Hausdorff and $Y$ is a closed subspace, $X/Y$ is compact
Hausdorff, so this statement makes sense. With respect to this
construction, the long exact sequence of a pair $(X, x_0)$ consisting of a
space $X$ together with a basepoint $x_0$ becomes a long exact sequence
relating the groups $K^i(X)$ and $\wtilde{K}^i(X)$. 
\[
\cdots \rightarrow \wtilde{K}^i(X) \rightarrow K^i(X) \rightarrow K^i(x_0) \rightarrow
\wtilde{K}^{i-1}(X) \rightarrow \cdots
\]
For our purposes, the crucial results from $K$--theory will be those
relating $K^*(X)$ to the rational cohomology $H(X; \mathbb Q)$ of $X$. The
exact sequence above, and the corresponding exact sequence on the reduced
and unreduced rational cohomology of $X$, will allow us to move between
maps on $K^*(X)$ and $\wtilde{K}^*(X)$ with relative ease.

Recall that the Chern classes of a vector bundle $V$ over a space $X$ are a set of natural characteristic classes $c_i(V) \in H^{2i}(X)$. The total Chern class of an $n$--dimensional vector bundle $V$ is $c(V) = c_0(V) + c_1(V) + \cdots + c_n(X)$. If $X$ has the structure of a complex manifold, as do most spaces of interest to us, we will often use $c(X)$ to refer to the total Chern class of the tangent bundle $TX$ of $X$. The Chern classes may be used to produce a ring homomorphism $K^0(X) \cup K^1(X) \rightarrow H^*(X; \mathbb Q)$.

Let $\sigma_j(x_1,\ldots ,x_k)$ be the elementary symmetric functions on $k$ elements as in \fullref{HomotopyCohomologySection}. Then we assert the existence of a new set of polynomials $s(x_1,\ldots ,x_k)$ with the following properties.

\begin{lemma}
There exists a unique set of polynomials $s_j(x_1,\ldots ,x_k)$ with the property that $s_j(\sigma_1(x_1,\ldots ,x_k),\ldots ,\sigma_k(x_1,\ldots ,x_k)) = x_1^j + \cdots + x_k^j$.
\end{lemma}

The polynomials $s_j$ are defined recursively in terms of the elementary symmetric functions and the $s_i$ of lower degree by the following relation.
\begin{align*}
s_j = \sigma_1 s_{j-1} - \sigma_2 s_{j-2} + \cdots + (-1)^{j-2}\sigma_{j-2}s_{2} + (-1)^{j-1}\sigma_{j-1}s_1
\end{align*}
We can now define the Chern character of a vector bundle.
\begin{align*}
\ch\co K^0(X) &\rightarrow H^{\mathrm{even}}(X; \mathbb Q) \\
[V]_S &\mapsto n + \sum_{j>0} s_k(c_1(V),\ldots ,c_k(V))/k!
\end{align*}
While this is the quickest way to produce this definition, it is not clear
that it is necessarily the most intuitive. To clarify: this definition is
explicitly chosen so that if $L$ is a line bundle over $X$ and the total
Chern class of $L$ is $1 + c_1(L)$, then $\ch(L) = 1 +c_1(L) +
\frac{c_1(L)^2}{2!} + \frac{c_1(L)^3}{3!}+\cdots$, the ``exponential'' of
the total Chern class of $L$. As the map $\ch$ was intended to be a ring
isomorphism on $K^0(X)$, we next require that for a product of line
bundles $L_1 \otimes L_2 \otimes \cdots \otimes L_k$ over $X$ we have
$\ch(L_1 \otimes \cdots \otimes L_k) = \ch(L_1)\cdots\ch(L_k)$. A computation leads to the formula above.

The Chern character descends to a reduced map $\wtilde{\ch}\co
\wtilde{K}^0(X) \rightarrow \wtilde{H}^{\mathrm{even}}(X; \mathbb Q)$. We may also consider the Chern character of vector bundles on $\Sigma(X)$, which yields the following map into the odd cohomology of $X$.
\begin{align*}
\ch\co K^1(X)=\wtilde{K}^0(\Sigma(X)) \rightarrow \wtilde{H}^{\mathrm{even}}(\Sigma(X);\mathbb Q)\cong H^{\mathrm{odd}}(X;\mathbb Q)
\end{align*}
Similarly, we define $\widetilde{\ch}\co\wtilde{K}^1(X)\rightarrow
\wtilde{H}^{\mathrm{odd}}(X;\mathbb Q)$. We then have the following extremely useful result.

\begin{proposition} \label{Cherncharacterisomorphism}
The Chern character induces a rational isomorphism
\begin{align*}
\ch \co K^0(X)\cup K^1(X) \otimes \mathbb Q \rightarrow H^*(X; \mathbb Q).
\end{align*}
In particular, the rank of $K^0(X)\cup K^1(X)$ is equal to the rank of
$H^*(X)$. The analogous statement holds for $\widetilde{\ch}$ and the two reduced theories.
\end{proposition}

The proof of this proposition takes Bott periodicity as a base case and is
laid out in Hatcher~\cite[Theorem 4.5]{HatcherKTheory}.

We will last need a $K$--theoretic statement about relationship between the torsion of $H^*(X)$ and $K^*(X)$. As it is a rather high-powered result, we'll say a few words about the proof, which relies on the Atiyah--Hirzebruch spectral sequence, a device for computing an arbitrary cohomology theory first described by its eponymous introducers in~\cite{MR0139181}.

\begin{proposition} \label{AtiyahHirzebruchSpecSeq}
Let $h^n(X)$ be any cohomology theory. Then there is a spectral sequence which converges to $h^n(X)$ whose $E^2$ page is given by
\begin{align*}
E^2_{i,j} = H^i(X; h^j(\{\mathrm{pt}\}).
\end{align*}
\end{proposition}

In the particular case of the cohomology theory $K^*(X)$ this spectral sequence collapses rationally, leading to a useful statement concerning the torsion groups of $K^*(X)$.

\begin{proposition}[Atiyah and Hirzebruch~{\cite[Section~2.5]{MR0139181}}]
\label{AtiyahHirzebruch}
Let $X$ be a compact Hausdorff space. If $H^i(X; \mathbb Z)$ is torsion-free for all $i$ and finitely generated, then $K^0(X)\cup K^1(X)$ is torsion-free of the same rank.
\end{proposition}

\begin{proof}
Recall that if $\{x_0\}$ is a one-point space then $K^q(x_0)$ is $\mathbb Z$ for $q$ even and $0$ for $q$ odd. Since unreduced $K$--theory is a cohomology theory, \fullref{AtiyahHirzebruchSpecSeq} gives first and second quadrant spectral sequence converging to $K^*(X)$ whose $E^2$ page has entries 
\begin{align*}
E^2_{p,q} = H^p(X; K^q(\{x_0\})) = 
\begin{cases}
H^p(X) \text{ if } q \text{ is even } \\
0 \text{ if } q  \text{ is odd } .\\
\end{cases}
\end{align*}
The rank of $K^i(X)$ is the sum $\sum_{p+q =i} \rk(E^{\infty}_{p,q})$.
Moreover, the ranks of the corresponding entries on the $E_2$ page of the
spectral sequence sum to $\sum_{p+q = i} \rk(E^{2}_{p,q}) = \sum_{p \equiv
n \mathrm{ mod } 2} H^p(X)$. By \fullref{Cherncharacterisomorphism}, the
rank of $K^i(X)$ and the total rank of the integer cohomology groups of
degree the same parity as $i$ are equal, so
$$\sum_{p+q =i}
\rk(E^{\infty}_{p,q}) = \sum_{p+q =i} \rk(E^2_{p,q}).$$
As $\rk(E^{\infty}_{p,q}) \leq \rk(E^2_{p,q})$ for all $p,q$, we see
that in fact this is an equality for all pairs $(p,q)$.

Now notice that every entry on the $E_2$ page of this spectral sequence is free abelian. Suppose $E^{\infty}_{p,q}$ has a nontrivial torsion subgroup for some fixed $p,q$ such that $q$ is even. Then there must be some first $E^r_{p,q}$ such that $r>2$ with this property. Because $E^r_{p,q}$ is a quotient of a subgroup of the free abelian group $E^{r-1}_{p,q}$, if it contains a torsion subgroup it must have lower rank than $E^r_{p,q}$, implying that $E^{\infty}_{p,q}$ has strictly lower rank than $E^2_{p,q}$. We have seen that this never happens. Therefore the $E^{\infty}$ page of our spectral sequence is torsion-free, implying that $K^*(X)$, which is filtered by the entries on the $E^{\infty}$ page, is also free abelian.

This argument, which is similar to the original proof of Atiyah and
Hirzebruch~\cite[Section 2.5]{MR0139181}, was suggested by Dan Ramras.
\end{proof}

Our proof that $(M^{\inv}, \smash{L_0^{\inv}}, \smash{L_1^{\inv}})$
carries a stable normal trivialization will contain as a crucial step a
proof that $H^*(M^{\inv}{\times}[0,1], (\smash{L_0^{\inv}} \times \{0\})
\cup (\smash{L_1^{\inv}} \times \{1\}))$ is torsion-free, and thus that the $K$--theory of this space must likewise be free abelian.

\section{Stable trivialization of the normal bundle} \label{StableTrivSection}

We are now finally ready to discuss the proof of \fullref{stable}. We begin by restating the theorem in a form that will be slightly easier to prove. Let $J$ denote the complex structure on $M^{\inv}$. As in \fullref{FloerCohomologyDefn}, we denote the trivial bundle $X \times \mathbb C^n \rightarrow X$ by $\mathbb C^n$ whenever the base space $X$ is clear from context, and similarly for $\mathbb R^n$.

\begin{proposition}[Seidel and Smith~{\cite[Section 3d]{MR2739000}}]
\label{stable2}
The existence of a stable normal trivialization of $\Upsilon(M^{\inv})$ is implied by the existence of a nullhomotopy of the map
\[
f\co(M^{\inv} \times [0,1], (\smash{L_0^{\inv}} \times \{0\}) \cup (\smash{L_1^{\inv}} \times \{1\})) \rightarrow (BU, BO)
\]
which classifies the complex bundle $\Upsilon(M^{\inv})$ and the totally real subbundles $N(\smash{L_0^{\inv}}) \times \{0\}$ over $L_0 \times \{0\}$ and $J(N(\smash{L_1^{\inv}})) \times \{1\}$ over $L_1 \times \{1\}$. That is, it suffices to find a complex stable trivialization of $\Upsilon(M^{\inv})$ which restricts to a real stable trivialization of $N(\smash{L_0^{\inv}}) \times \{0\}$ and of $J(N(\smash{L_1^{\inv}}))\times \{1\}$.

\end{proposition} 

\begin{proof}
This is largely the statement that we can dispense with the symplectic
structure involved in a stable normal trivialization and argue purely in
terms of complex vector bundles. Recall that any symplectic vector bundle
can be equipped with a compatible complex structure which is unique up to
homotopy, and any complex vector bundle similarly admits a
symplectification. Moreover, two symplectic vector bundles are isomorphic
if and only if their underlying complex vector bundles are isomorphic, and
isomorphisms of symplectic vector bundles map Lagrangian subbundles to
Lagrangian subbundles (see McDuff and
Salamon~\cite[Theorem~2.62]{MR1373431}). Let $\omega_M$ be the natural
symplectic structure on $N(M^{\inv})$ coming from the symplectic
structure on $TM$.

Suppose that $g \co N(M^{\inv}) \rightarrow BU$ is a classifying
map of $N(M^{\inv})$ thought of as a unitary vector bundle, so
that the image of $g$ lies inside $BU_{k_{\anti}}$. Let
$$\zeta_{k_{\anti}}\co EU_{k_{\anti}} \rightarrow
BU_{k_{\anti}}$$
be the complex $k_{\anti}$--dimensional
universal bundle, and similarly let
$$\eta_{k_{\anti}}\co
EO_{k_{\anti}} \rightarrow BO_{k_{\anti}}$$
be the real
$k_{\anti}$--dimensional universal bundle. Equip
$EU_{k_{\anti}}$ with a symplectic structure $\omega_{\zeta}$ such
that $\eta_{k_{\anti}} \subset \zeta_{k_{\anti}}$ is a
Lagrangian subbundle. Then the bundles $(N(M^{\inv}), \omega_M)$
and $(N(M^{\inv}), g^*(\omega_{\zeta}))$ are isomorphic (indeed,
equal) as complex vector bundles, so there is a symplectic vector bundle
isomorphism
$$\chi\co(N(M^{\inv}), \omega_M) \rightarrow
(N(M^{\inv}), g^*(\omega_{\zeta})).$$
This extends to a symplectic
vector bundle isomorphism
$$\wtilde{\chi}\co (\Upsilon(M^{\inv}),
\wtilde{\omega}_M) \rightarrow (\Upsilon(M^{\inv}),
\wtilde{g}^*(\omega_{\zeta}) = f^*(\omega_M)),$$
where in both cases the
symplectic forms are the pullbacks of the original symplectic forms on
$N(M^{\inv})$ to $\Upsilon(M^{\inv})$, and therefore
constant with respect to the interval $[0,1]$, as is the map
$\wtilde{\chi}$. From now on, we assume that we have first applied an
isomorphism of this form to $\Upsilon(M^{\inv})$ so that the map
$g \co \Upsilon(M^{\inv}) \rightarrow BU$ is in fact a symplectic
classifying map. We can if necessary precompose the resulting stable
normal trivialization with $\wtilde{\chi}$.

Consider a nullhomotopy $H$ of $f$.
\begin{align*}
H \co (M^{\inv} \times [0,1], \smash{L_0^{\inv}} \times \{0\} \cup \smash{L_1^{\inv}} \times \{1\}) \times [0,1] &\rightarrow (BU,BO) \\
(x,t,s) &\mapsto h_s(x,t)
\end{align*}
Here the map $h_0$ is equal to $f$ and the map $h_1$ is constant.

Since $M^{\inv}$ is homotopy equivalent to a compact subspace of itself, we may assume there is some $K>0$ such that if $s=k_{\anti} +K$, the image of $H$ lies inside $(BU_s, BO_s)$. Let $\zeta_s \co EU_s \rightarrow BU_s$ be the complex $s$--dimensional universal bundle with subbundle $\eta_s \co EO_s \rightarrow BO_s$ the real $s$--dimensional universal bundle. Then the pullbacks of $\zeta_s$ and $\eta_s$ along $h_1$ and $h_0$ are certain bundles of great interest to our investigation.
\begin{align*}
h_0^*(\zeta_s) &= (\Upsilon(M^{\inv}) \oplus \mathbb C^K, h_0^*\omega_{\zeta}) \\
h_1^*(\zeta_s) &= (\mathbb C^s, h_1^*\omega_{\zeta}) \\
(h_0|_{\smash{L_0^{\inv}}\times \{0\}})^*(\eta_s) &= (N(\smash{L_0^{\inv}}) \times \{0\}) \oplus \mathbb R^K \\
(h_0|_{\smash{L_1^{\inv}} \times \{1\}})^*(\eta_s) &= (J(N(\smash{L_1^{\inv}})) \times \{1\}) \oplus \mathbb R^K \\
(h_1|_{L_i^{\inv} \times \{i\}})^*(\eta_s) &= \mathbb R^s \text{ for } i=0,1
\end{align*}
For clarity's sake it should be borne in mind that $\mathbb R^K$ and $\mathbb R^s$ refer to the canonical real subspaces in $\mathbb C^K$ and $\mathbb C^s$.

Since $H$ is a nullhomotopy, it induces a stable trivialization $\psi$ of $\Upsilon(M^{\inv})$. Write an arbitrary vector in $\Upsilon(M^{\inv})$ as $(x,t,v)$ where $(x,t) \in M^{\inv} \times [0,1]$ and $v$ is an element of the fiber over $(x,t)$.
\begin{align*}
\psi \co N(M^{\inv}) \oplus \mathbb C^K = h_0^*(\zeta_s)
&\stackrel{\hbox{\footnotesize$\sim$}}{\longrightarrow} h_1^*(\zeta_s) = \mathbb C^s \\
(x,t,v) & \mapsto \psi(x,t,v)
\end{align*}
The restrictions of $\psi$ to $(N(\smash{L_0^{\inv}}) \times \{0\}) \oplus \mathbb  R^K$ and to $(J(N(\smash{L_1^{\inv}}) \times\{1\})) \oplus \mathbb R^K$ are real stable trivializations of these two bundles.
\begin{align*}
\psi|_{(N(\smash{L_0^{\inv}} \times \{0\}) \oplus \mathbb R^K)} \co (N(\smash{L_0^{\inv}}) \times \{0\}) \oplus \mathbb R^K \rightarrow \mathbb R^s \\
\psi|_{(J(N(\smash{L_1^{\inv}}) \times \{1\}) \oplus \mathbb R^K)} \co (J(N(\smash{L_0^{\inv}})) \times \{1\}) \oplus \mathbb R^K \rightarrow \mathbb R^s
\end{align*}
Since $\Upsilon(M^{\inv})$ is the pullback of $N(M^{\inv})$ to $M^{\inv} \times [0,1]$, the map $h_0 = f$ is constant with respect to the interval $[0,1]$. That is, $h_0(x,t_1) = h_0(x,t_2)$ for all $x \in M^{\inv}$ and $t_1,t_2 \in [0,1]$. Then each $\psi_t = \psi|_{M^{\inv} \times \{t\}}$ is a stable trivialization of $N(M^{\inv}) \times \{t\} = N(M^{\inv})$.  More concretely, we have symplectic trivializations
\begin{align*}
\psi_t \co N(M^{\inv}) &\rightarrow \mathbb C^k \\
(x,v) &\mapsto \psi(x,t,v).
\end{align*}
We will use the family of isomorphisms $\psi_t$ to produce a stable normal trivialization of $\Upsilon(M^{\inv})$. Consider a map $\phi$ defined by applying $\psi_0$ to each $M^{\inv} \times \{t\} \subset M^{\inv} \times [0,1]$.
\begin{align*}
\phi \co \Upsilon(M^{\inv}) &\rightarrow \mathbb C^k \\
(x,t,v) &\mapsto \psi_0((x,v)) = \psi(x,0,v).
\end{align*}
This is a stable trivialization of $\Upsilon(M^{\inv})$. Because the symplectic structure on $\Upsilon(M^{\inv})$ is constant with respect to the interval $[0,1]$, it is in fact a symplectic isomorphism of vector bundles. We next need to produce two Lagrangian subbundles $\Lambda_0$ and $\Lambda_1$ satisfying the conditions outlined in \fullref{stablenormaltriv}. Consider the following candidates.
\begin{align*}
\Lambda_0|_{\smash{L_0^{\inv}} \times \{t\}} &= (N(\smash{L_0^{\inv}}) \times \{t\}) \oplus \mathbb R^K \\
\Lambda_1|_{\smash{L_1^{\inv}} \times \{t\}} &= \psi_0^{-1} \circ \psi_t(N(\smash{L_1^{\inv}}) \times \{t\} \oplus i \mathbb R^K).
\end{align*}
Since the maps $\psi_t$ form a homotopy, $\Lambda_1$ is a smooth subbundle as desired. Both subbundles are Lagrangian since their restriction to each $L_i^{\inv} \times \{t\}$ is Lagrangian.  The next thing to check is the restriction of $\Lambda_i$ to $L^{\inv}_i \times \{0\}$ for $i=0,1$.
\begin{align*}
\Lambda_0|_{\smash{L_0^{\inv}} \times \{0\}} &= (N(\smash{L_0^{\inv}}) \times \{0\}) \oplus \mathbb R^K \\
\Lambda_1|_{\smash{L_1^{\inv}} \times \{0\}} &= \psi_0^{-1} \circ \psi_0((N(\smash{L_1^{\inv}}) \times \{0\}) \oplus i\mathbb R^K)
= (N(\smash{L_1^{\inv}}) \times \{0\}) \oplus i \mathbb R^K.
\end{align*}
This is exactly as desired. The other condition to check is that $\phi(\Lambda_0|_{L^{\inv}_0 \times \{1\}})$ is $\mathbb R^s \subset \mathbb C^s$ and $\phi(\Lambda_1|_{L^{\inv}_1 \times \{1\}})$ is $i\mathbb R^s \subset \mathbb C^s$.
\begin{align*}
\phi(\Lambda_0|_{\smash{L_0^{\inv}} \times \{1\}}) &= \psi_0((N(\smash{L_0^{\inv}}) \times \{0\}) \oplus \mathbb R^K) \\
&= \psi((N(\smash{L_0^{\inv}}) \times \{0\}) \oplus \mathbb R^K) \\
&= \mathbb R^s \\
\phi(\Lambda_1|_{\smash{L_1^{\inv}} \times \{1\}}) &= \psi_0( \psi_0^{-1} \circ \psi_1((N(\smash{L_1^{\inv}}) \times \{0\}) \oplus i\mathbb R^K) \\
&= \psi_1(J(J((N(\smash{L_1^{\inv}}) \times \{0\}) \oplus i \mathbb R^K)))	\\
&=J(\psi_1(J(N(\smash{L_1^{\inv}}	) \times \{0\}) \oplus \mathbb R^k))	\\								
&= i(\mathbb R^s)							
\end{align*}
Ergo the stable trivialization $\phi$ of $\Upsilon(M^{\inv})$ together with the subbundles $\Lambda_0$ and $\Lambda_1$ constitutes a stable normal trivialization of $\Upsilon(M^{\inv})$.\end{proof}

The proof of \fullref{stable} will rely on analysis of the bundle in question and the $K$--theory of $\mathrm{Sym}^{n-1}(S^2 \backslash \{z,w\})$. We start by showing that $\Upsilon(M^{\inv})$ is trivial as a complex vector bundle. 

\begin{lemma} \label{chern}
The complex vector bundle $\Upsilon(M^{\inv})$ is trivial.
\end{lemma}

\begin{proof}
It suffices to show $N(M^{\inv})$ is trivial.  Let $h$ be a map from
$\mathrm{Sym}^{n-1}(S^2)$ to $\mathrm{Sym}^{2n-2}(\Sigma(S^2))$ which
restricts to the embedding \eqref{embedding} on $\mathrm{Sym}^{n-1}(S^2
\backslash \{\mathbf{w},\mathbf{z}\})$ and is defined as follows.
\begin{align*}
h \co \mathrm{Sym}^{n-1}(S^2) &\rightarrow \mathrm{Sym}^{2n-2}(\Sigma(S^2)) \\
(x_1\ldots x_{n-1}) &\mapsto (\wtilde{x}_1\tau(\wtilde{x}_1)\ldots
\wtilde{x}_{n-1}\tau(\wtilde{x}_{n-1}))
\end{align*}
A proof on charts that $h$ is continuous and holomorphic appears in
\fullref{appendix1}.

Now we have the following useful commutative diagram of spaces, where $i_1$, $i_2$ are inclusion maps.
\begin{align}\label{fourspacesetup}
\xymatrix{ 
\mathrm{Sym}^{n-1}(S^2) \ar[rr]^-{h} && \mathrm{Sym}^{2n-2}(\Sigma(S^2)) \\
\mathrm{Sym}^{n-1}(S^2 \backslash \{\mathbf{z}, \mathbf{w}\}) \ar@{^(->}[u]^-{i_1} \ar@{^(->}[rr]^-{i} && \mathrm{Sym}^{2n-2}(\Sigma(S^2) \backslash \{\mathbf{z},\mathbf{w}\}) \ar@{^(->}[u]^-{i_2}
}
\end{align}
Since the map $i_2$ is an inclusion of the space
$\mathrm{Sym}^{2n-2}(\Sigma(S) \backslash \{ \mathbf{z}, \mathbf{w} \})$
of complex dimension $2n-2$ into the similarly $(2n-2)$--dimensional space
$\mathrm{Sym}^{2n-2}(\Sigma(S))$, the tangent bundle to $M =
\mathrm{Sym}^{2n-2}(\Sigma(S) \backslash \{ \mathbf{z}, \mathbf{w} \})$ is the pullback of the tangent bundle to $\mathrm{Sym}^{2n-2}(\Sigma(S))$ along $i_2$. Hence we have the following equalities of vector bundles.
\begin{align*}
N(M^{\inv}) \oplus T(M^{\inv}) & = i^*(TM) \\
										   &=
i^*(T(\mathrm{Sym}^{2n-2}(\Sigma(S) \backslash \{ \mathbf{z}, \mathbf{w} \})) \\
										   &= i^* \circ i_2^*(T(\mathrm{Sym}^{2n-2}(\Sigma(S)))) \\
										   &= i_1^* \circ h^*(T(\mathrm{Sym}^{2n-2}(\Sigma(S))))
\end{align*}
It will be most illuminating to consider the map $i_1$. Let $i_3 \co S^2
\backslash \{\mathbf{z}, \mathbf{w}\} \hookrightarrow S^2$, and $\pi_1$, $\pi_2$ be quotient maps, so that the following diagram commutes.
\begin{align*}
\xymatrix{
(S^2 \backslash \{ \mathbf{z}, \mathbf{w} \})^{n-1} \ar@{^(->}[rr]^-{i_3^{\times (n-1)}} \ar[d]^{\pi_1}&& (S^2)^{n-1} \ar[d]^{\pi_2} \\
\mathrm{Sym}^{n-1}(S^2 \backslash \{ \mathbf{z}, \mathbf{w} \}) \ar@{^(->}[rr]^-{i_1} && \mathrm{Sym}^{n-1}(S^2)
}
\end{align*}
The map $i_3$ is nullhomotopic. Since the operation of taking the symmetric product preserves homotopy equivalence, the original map $i_1$ is likewise nullhomotopic. In particular, the pullback of any bundle along $i_1$ is trivial, so we conclude that 
\begin{align*}
0 &= i_1^* \circ h^* (T(\mathrm{Sym}^{2n-2}(\Sigma(S)))) \\
  &= N(M^{\inv}) \oplus T(M^{\inv})
\end{align*}
Finally, $TM^{\inv}$ is the pullback of $T(\mathrm{Sym}^{n-1}(S^2))$ along $i_1$, and therefore trivial. Ergo $N(M^{\inv})$ is trivial as a complex vector bundle. This implies that $\Upsilon(M^{\inv})$, the pullback of $N(M^{\inv})$ to $M^{\inv} \times [0,1]$, is also trivial.\end{proof}

\begin{remark}
The triviality of $N(M^{\inv})$ is the reason we have been
compelled to choose a Heegaard surface for $(S^3,K)$ on the sphere $S^2$
and strew basepoints about like confetti. In the general case, given a
surface $S$ of genus $g$ with a choice of $2n$ basepoints $\mathbf{w} =
(w_1,\ldots ,w_n)$ and  $\mathbf{z} = (z_1,\ldots ,z_n)$, its double branched cover over these basepoints is a surface $\Sigma(S)$ of genus $n+2g-1$. By the same mechanism as before, there is an embedding
\begin{align*}
i \co \mathrm{Sym}^{n+g-1}(S \backslash \{\mathbf{w},\mathbf{z}\})
\hookrightarrow \mathrm{Sym}^{2n+2g-2}(\Sigma(S) \backslash \{\mathbf{w},\mathbf{z}\})
\end{align*}
In~\cite{MR0151460} Macdonald has calculated the Chern classes of the
symmetric product of a Riemann surface. A computation using his results
proves that the normal bundle to $i(\mathrm{Sym}^{n+g-1}(S \backslash
\{\mathbf{w},\mathbf{z}\}))$ in $\mathrm{Sym}^{2n+2g-2}(\Sigma(S)
\backslash \{\mathbf{w},\mathbf{z}\})$ is in general stably equivalent to
the tangent bundle of $\mathrm{Sym}^{n+g-1}(S \backslash \{\mathbf{w},\mathbf{z}\})$. In the case that $S$ is actually a sphere, both are trivial.
\end{remark}

We will also require  $N(L_i^{\inv})$ for $i=0,1$ to be a trivial real bundle.

\begin{proposition} \label{trivialsubbundles}
 $N(L_i^{\inv})$ is trivial for $i=0,1$.
\end{proposition}

\begin{proof}
Recall that $L_1 = \mathbb T_{\wtilde{\boldsymbol \alpha}} \subset
\mathrm{Sym}^{2n-2}(S \backslash \{\mathbf{z},\mathbf{w}\})$ is the totally real
torus
$$\wtilde{\alpha}_1 \times \tau(\wtilde{\alpha}_1) \times \cdots
\times \wtilde{\alpha}_{n-1} \times \tau(\wtilde{\alpha}_{n-1}).$$
Thus the tangent bundle
$$TL_1 = T(\wtilde{\alpha}_1) \times
T(\tau(\wtilde{\alpha}_1)) \times \cdots \times T(\wtilde{\alpha}_{n-1})
\times T(\tau(\wtilde{\alpha}_{n-1}))$$
is the trivial real tangent bundle
to $(S^1)^{2n-2}$.  The invariant set
$$\smash{L_1^{\inv}} = i(\mathbb
T_{\boldsymbol \alpha}) = i(\alpha_1 \times \cdots \times \alpha_{n-1})$$
is embedded in $L_1$ via
\begin{align*}
i|_{\mathbb T_{\boldsymbol \alpha}} \co \mathbb T_{\boldsymbol\alpha} &\hookrightarrow
\mathbb T_{\wtilde{\boldsymbol \alpha}}\\
(x_1,\ldots ,x_{n-1}) &\mapsto (\wtilde{x}_1, \tau(\wtilde{x}_1),\ldots
,\wtilde{x}_{n-1}, \tau(\wtilde{x}_{n-1})).
\end{align*}
Therefore the tangent and normal bundles to $\smash{L_1^{\inv}}$ have the following descriptions.
\begin{align*}
T(\smash{L_1^{\inv}}) &= i_*(T(\mathbb T_{\boldsymbol \alpha})) \\ &=
\{(\wtilde{v}_1, \tau_*(\wtilde{v}_1),\ldots ,\wtilde{v}_{n-1},
\tau_*(\wtilde{v}_{n-1})) : (v_1,v_2,\ldots ,v_{n-1}) \in \mathbb
T_{\boldsymbol \alpha} \} \subset TL_1 \\
N(\smash{L_1^{\inv}}) &= \{(\wtilde{v}_1, -\tau_*(\wtilde{v}_1),\ldots ,
\wtilde{v}_{n-1}, -\tau_*(\wtilde{v}_{n-1})): (v_1,\ldots ,v_{n-1}) \in
\mathbb T_{\boldsymbol \alpha} \} \subset TL_1.
\end{align*}

The point is that $T(\smash{L_1^{\inv}}) \simeq T((S^1)^{n-1})$ is trivial, and there is an isomorphism
\begin{align*}
T(\smash{L_1^{\inv}}) &\rightarrow N(\smash{L_1^{\inv}}) \\
(\wtilde{v}_1, \tau_*(\wtilde{v}_1),\ldots ,\wtilde{v}_{n-1},
\tau_*(\wtilde{v}_{n-1})) &\mapsto (\wtilde{v}_1,
-\tau_*(\wtilde{v}_1),\ldots ,\wtilde{v}_{n-1}, -\tau_*(\wtilde{v}_{n-1})). 
\end{align*}

Triviality of $N(\smash{L_0^{\inv}})$ is proven analogously. \end{proof}

With these facts in hand, we are finally ready to give a proof of \fullref{stable}.

\begin{proof}[Proof of \fullref{stable}]
We claim that $\Upsilon(M^{\inv})$ carries a stable normal trivialization. As per \fullref{stable2}, it suffices to produce a nullhomotopy of the map
\begin{align*}
f \co(M^{\inv}, (\smash{L_0^{\inv}} \times \{0\}) \cup (\smash{L_1^{\inv}} \times \{1\})) \rightarrow (BU,BO)
\end{align*}
which classifies the complex bundles $\Upsilon(M^{\inv})$ and 
Lagrangian subbundles $N(\smash{L_0^{\inv}}) \times \{0\}$ over
$\smash{L_0^{\inv}}$ and $J(N(\smash{L_1^{\inv}})) \times \{1\}$
over $\smash{L_1^{\inv}}$. In fact, we shall do slightly better. By
\fullref{trivialsubbundles}, both $N(\smash{L_0^{\inv}}) \times \{0\}$
and $J(N(\smash{L_1^{\inv}})) \times \{1\}$ are trivial bundles. Choose,
once and for all, a preferred real trivialization of each. We will show
that there is a complex trivialization of $\Upsilon(M^{\inv})$ whose
restriction to $\left(N(\smash{L_0^{\inv}}) \times \{0\}\right) \cup
J(N(\smash{L_1^{\inv}}) \times \{1\})$ is the fixed real trivialization
in question; this implies the existence of a nullhomotopy of the map $f$.

For ease of reference, let $X = (\smash{L_0^{\inv}} \times \{0\}) \cup
(\smash{L_1^{\inv}} \times \{1\}) = (\mathbb T_{\boldsymbol \beta} \times
\{0\}) \cup (\mathbb T_{\boldsymbol \alpha} \times \{1\})$.

Since $N(\smash{L_0^{\inv}} \times \{0\})$ is a totally real subbundle of
$\Upsilon`(`M^{\inv}`)|_{\smash{L_0^{\inv}}}$, the bundle
$\Upsilon`(`M^{\inv}`)|_{\smash{L_0^{\inv}}}$ is equal to $N(\smash{L_0^{\inv}}) \times \{0\} \oplus J(N(\smash{L_0^{\inv}})) \times \{0\})$. In particular, a choice of real trivialization of $N(\smash{L_0^{\inv}}) \times \{0\}$ induces a choice of complex trivialization of $\Upsilon(M^{\inv})|_{\smash{L_0^{\inv}}}$. Similarly, a choice of real trivialization of $J(N(\smash{L_1^{\inv}})\times \{1\})$ induces a choice of complex trivialization of $\Upsilon(M^{\inv})|_{\smash{L_1^{\inv}} \times \{1\}}$.

We now have a complex trivialization of $\Upsilon(M^{\inv})|_{X}$,
and we would like to show it extends to a trivialization of
$\Upsilon(M^{\inv})$. It is enough to demonstrate that the
relative vector bundle $[\Upsilon(M^{\inv})_{\mathrm{rel}}] \in
\wtilde{K}^0((M^{\inv}\times [0,1]), X) \cong \wtilde{K}^0((M^{\inv}\times[0,1])/X)$ is trivial.

To begin, we claim that the equivalence class
$[\Upsilon(M^{\inv})_{\mathrm{rel}}]$ is a torsion element of
$\wtilde{K}^0((M^{\inv}\times[0,1])/X)$. As the reduced Chern
character
$$\widetilde{\ch} \co
\wtilde{K}^0((M^{\inv}\times [0,1])/X)\otimes \mathbb Q
\rightarrow \wtilde{H}^*((M^{\inv}\times [0,1])/X; \mathbb Q)$$
is an isomorphism, it suffices to show that the Chern classes of $\Upsilon(M^{\inv})_{\mathrm{rel}}$ are trivial.

Recall from \fullref{cohomologysurjection} that the map $H^*(M^{\inv}\times[0,1]) \rightarrow H^*(X)$ induced by inclusion is a surjection for each $*\geq 1$. This remains true in the setting of reduced cohomology, which will be a slightly more useful setting for our purposes. Indeed, since $(M^{\inv} \times [0,1],X)$ can be taken to be a CW pair, there is a long exact sequence on the reduced cohomology of these three spaces.
\begin{multline*}
\cdots \longrightarrow \wtilde{H}^m((M^{\inv} \times[0,1])/X)
\stackrel{q^*}{\longrightarrow} \wtilde{H}^m(M^{\inv})
\stackrel{i^*}{\longrightarrow} \wtilde{H}^m(X) \\
\stackrel{\partial}{\longrightarrow} \wtilde{H}^{m+1}((M^{\inv}\times[0,1])/X)
\longrightarrow \cdots
\end{multline*}
Here $i$ is the inclusion map $i \co X \hookrightarrow M^{\inv}$ and $q$ is the quotient map $q \co M^{\inv} \rightarrow (M^{\inv} \times[0,1])/X$. Since $i^*$ is a surjection for $m\geq1$, for $m\geq 2$ this sequence breaks up into short exact sequences as follows.
\label{shortexact}
\begin{align*}
\xymatrix{
0 \ar[r] & \wtilde{H}^m((M^{\inv}\times[0,1]/X)
\ar@{^(->}[r]^-{q^*} & \wtilde{H}^m(M^{\inv}) \ar@{>>}[r]^-{i^*} &
\wtilde{H}^m(X) \ar[r] & 0
}
\end{align*}
In particular, for $m\geq 2$ the map $q^* \co
\wtilde{H}^m((M^{\inv}\times[0,1])/X) \hookrightarrow
\wtilde{H}^m(M^{\inv})$ is an injection. For $i>0$, let
$c_i(\Upsilon(M^{\inv})) = 0$ be the $i$th Chern class of
$\Upsilon(M^{\inv})$ in $H^{2i}(M^{\inv}) = \wtilde{H}^{2i}(M^{\inv})$. Then since $\Upsilon(M^{\inv}) = q^*(\Upsilon(M^{\inv})_{\mathrm{rel}})$, we see that $ 0 = c_i(\Upsilon(M^{\inv})) = q^*(c_i(\Upsilon(M^{\inv})_{\mathrm{rel}}))$. As the induced map $q^*$ is injective on $H^{2i}(M^{\inv}/X)$, the $i$th Chern class of $\Upsilon(M^{\inv})_{\mathrm{rel}}$ is trivial.

Next we show that the $K$--theory of $M^{\inv}/X$ is in fact
torsion free. By \fullref{AtiyahHirzebruch}, it suffices to show that the
cohomology of $M^{\inv}/X$ is torsion-free. Consider again the
short exact sequence above. When $m\geq 2$, the group
$\wtilde{H}^m(M^{\inv}/X)$ injects into the free abelian group
$\wtilde{H}^m(M^{\inv}) \cong \mathbb{Z}^{\tbinom{2n-1}{m}}$ and
therefore is torsion free. In order to analyze $\wtilde{H}^1(M^{\inv}/X)$, let us look closely at the early stages of the long exact sequence of the pair $(M^{\inv}, X)$.
\begin{align*}
\xymatrix{
\cdots \longrightarrow 0 = \wtilde{H}^0(M^{\inv}) \ar[r]^-{i^*} & \wtilde{H}^0(X)
\ar[r]^-{\partial}& \wtilde{H}^1(M^{\inv}/X) \ar[r]^{q^*} &
\wtilde{H}^1(M^{\inv}) \longrightarrow \cdots
}
\end{align*}
We observe that $\wtilde{H}^0(X) \cong \mathbb Z$ since $X =
(\smash{L_0^{\inv}} \times \{0\}) \cup (\smash{L_1^{\inv}} \times \{1\})$
has two connected components. Moreover, $\wtilde{H}^1(M^{\inv})$ is free abelian, so its subgroup $\mathrm{Im}(q^*)$ is as well. Thus we have the short exact sequence
\begin{align*}
\xymatrix{
0 \ar[r] & \mathbb Z \ar@{^(->}[r] & \wtilde{H}^1((M^{\inv}\times[0,1])/X)
\ar@{>>}[r] & \mathrm{Im}(q^*) \ar[r] & 0.
}
\end{align*}
As $\mathbb Z$ and $\mathrm{Im}(q^*)$ are free abelian,
$\wtilde{H}^1((M^{\inv}\times[0,1])/X)$ is as well. Thus
$\wtilde{H}^*((M^{\inv}\times[0,1])/X)$ is torsion free, implying
that $\wtilde{K}^0((M^{\inv}\times[0,1])/X)$ is as well.

Ergo the relative bundle $[\Upsilon(M^{\inv})_{\mathrm{rel}}] \in
\wtilde{K}^0((M^{\inv}\times[0,1])/X)$ is trivial, implying that $\Upsilon(M^{\inv})$ carries a stable normal trivialization.\end{proof}

This proves \fullref{stable}, which together with the discussion of the
symplectic geometry of $M$ in \fullref{SymplecticGeometrySection} shows
that $M = \mathrm{Sym}^{2n-2}(\Sigma(S) \backslash \{\mathbf{w},\mathbf{z}\})$ satisfies the hypotheses of 
\fullref{SeidelSmith}. This proves \fullref{existsspecseq}. For the reader's convenience, we summarize here the arguments that allow us to deduce Corollaries~\ref{main}, \ref{sharper}, and~\ref{Alexander corollary} from \fullref{existsspecseq}.

\begin{proof}[Proof of \fullref{main}] Since there is a spectral sequence whose $E^1$ page is
$$(\widehat{\HFK}(\Sigma(K), K) \otimes V^{\otimes (n-1)}) \otimes \mathbb
Z\llparen q\rrparen$$
and whose $E^{\infty}$ page is isomorphic to $(\widehat{\HFK}(S^3,K) \otimes V^{\otimes (n-1)}) \otimes Z\llparen q\rrparen$ as $\mathbb Z\llparen q\rrparen$--modules, the rank of $(\widehat{\HFK}(\Sigma(K), K) \otimes V^{\otimes (n-1)}) \otimes \mathbb Z\llparen q\rrparen$ as a $\mathbb Z\llparen q\rrparen$--module is greater than the rank of $(\widehat{\HFK}(S^3,K) \otimes V^{\otimes (n-1)}) \otimes Z\llparen q\rrparen$ as a $\mathbb Z\llparen q\rrparen$--module. Ergo
\begin{align*}
2^{n-1} \rk ( \widehat{\HFK}(\Sigma(K),K)) \geq 2^{n-1}\rk( \widehat {\HFK}(S^3,K))
\end{align*}
implying
\[
\rk ( \widehat{\HFK}(\Sigma(K),K)) \geq \rk( \widehat
{\HFK}(S^3,K)).\proved
\]
\end{proof}

\begin{proof}[Proof of \fullref{sharper}]
We have seen in \fullref{HeegaardFloerSection} that the spectral sequence
of \fullref{existsspecseq} arises from the double complex 
\begin{align*}
\xymatrix{
0 \ar[r] \ar[d] & \widehat{\mathit{CFK}}_{i+1}(\wtilde{\mathcal D})
\ar[d]^-{\partial} \ar[r]^-{1 + \tau^{\#}}  &
\widehat{\mathit{CFK}}_{i+1}(\wtilde{\mathcal D}) \ar[d]^-{\partial}
\ar[r]^-{1 + \tau^{\#}} & \widehat{\mathit{CFK}}_{i+1}(\wtilde{\mathcal
D}) \ar[d]^-{\partial} \ar[r] & \cdots \\
0 \ar[r] \ar[d] & \widehat{\mathit{CFK}}_i(\wtilde{\mathcal D})
\ar[d]^-{\partial} \ar[r]^-{1 + \tau^{\#}}  &
\widehat{\mathit{CFK}}_i(\wtilde{\mathcal D}) \ar[d]^-{\partial}
\ar[r]^-{1 + \tau^{\#}} & \widehat{\mathit{CFK}}_i(\wtilde{\mathcal D})
\ar[d]^-{\partial} \ar[r] & \cdots \\
0 \ar[r] & \widehat{\mathit{CFK}}_{i-1}(\wtilde{\mathcal D}) \ar[r]^-{1 +
\tau^{\#}}  & \widehat{\mathit{CFK}}_{i-1}(\wtilde{\mathcal D}) \ar[r]^-{1
+ \tau^{\#}} & \widehat{\mathit{CFK}}_{i-1}(\wtilde{\mathcal D})
\ar[r] & \cdots
}
\end{align*}
where $\mathcal D$ is a Heegaard diagram for $(S^3,K)$ on the
sphere $S^2$, and $\wtilde{\mathcal D}$ its lift to a Heegaard diagram for
$(\Sigma(K),K)$, and that this spectral sequence preserves the canonical
$\mathrm{spin}^{\mathrm{c}}$ structure on
$\widehat{\mathit{CFK}}(\wtilde{\mathcal D})$. (Other
$\mathrm{spin}^{\mathrm{c}}$ structures are exchanged in pairs by the
involution $\tau^*$ on $\widetilde{\HFK}(\wtilde{\mathcal D})$, and thus vanish after the second page of the spectral sequence.) Since the splitting of $\widetilde{\HFK}(\tilde{\mathcal D})$ along $\mathrm{spin}^{\mathrm{c}}$ canonically corresponds to the splitting of $\widehat{\HFK}(Y^3,K)$, we sharpen our statement to the following.
\[
\rk \big( \widehat{\HFK}(\Sigma(K), K, \mathfrak s_0)\big) \geq \rk
\big(\widehat{\HFK}(S^3, K)\big).\proved
\]
\end{proof}

\begin{proof}[Proof of \fullref{Alexander corollary}]
Consider the double complex of the preceding proof. Both differentials in the spectral sequence preserve the Alexander grading, and the isomorphism between the $E^{\infty}$ page of the spectral sequence and $\widetilde{\HFK}(\mathcal D)$ does not disrupt the splitting of the spectral sequence along the Alexander grading. Therefore we sharpen \fullref{main} to
\[
\rk \big( \widetilde{\HFK}(\wtilde{\mathcal D}, \mathfrak s_0, i)\big)
\geq \rk \big(\widetilde{\HFK}(\mathcal D, i)\big)
\]
In particular, let $g$ be the largest Alexander grading for which
$\widetilde{\HFK}(\wtilde{\mathcal D})$ is nonzero. Then since the vector
space $V$ has elements with gradings $(0,0)$ and $(-1,-1)$, we have
$\widetilde{\HFK}(\wtilde{\mathcal D}, s_0, g) = \widehat{\HFK}(\Sigma(K), K, g)$, and similarly $\widetilde{\HFK}(\mathcal D, g) = \widehat{\HFK}(S^3,K)$. Ergo the inequality in this Alexander grading takes the form
\[
\rk \big( \widehat{\HFK}(\Sigma(K), K, \mathfrak s_0, g)\big) \geq \rk
\big(\widehat{\HFK}(S^3, K, g)\big).\proved
\]
\end{proof}

\begin{remark}
In~\cite[Section 3d]{MR2739000}, Seidel and Smith observe that although their theorem deals with an obstruction on the level of the classifying map
\[
f \co([0,1] \times M^{\inv}, (\smash{L_0^{\inv}} \times \{0\}) \cup (\{1\} \times \smash{L_1^{\inv}})) \rightarrow (BU, BO)
\]
there is evidence to suggest that the fundamental obstruction is a map encoding slightly less structure. Let $\mathcal P^{\inv} = \{y\co[0,1] \rightarrow M^{\inv}\co y(0) \in \smash{L_0^{\inv}}, y(1) \in \smash{L_1^{\inv}} \}$ be the set of paths between the two Lagrangians in $M^{\inv}$. Then $f$ induces a map
\[
\mathcal P^{\inv} \rightarrow U/O
\]
This descends to a map on loop spaces:
\begin{align} \label{loopspaces}
\Omega \mathcal P^{\inv} \rightarrow \Omega(U/O) \cong \mathbb Z \times BO
\end{align}
This map records the Maslov index of a self-flow in $\Omega \mathcal
P^{\inv}$. The difference between the Maslov index and equivariant
Maslov index of a holomorphic curve $u$ in $M^{\inv}$ is
classified by the composition of \eqref{loopspaces} with a map
$M(x_-,x_+)^{\inv} \rightarrow  \Omega \mathcal P^{\inv}$;
which suggests that \eqref{loopspaces} is the fundamental obstruction to
the existence of a spectral sequence from $\mathit{HF}(L_0, L_1)$ to
$\mathit{HF}(\smash{L_0^{\inv}}, \smash{L_1^{\inv}})$. For more detail,
see Seidel and Smith~\cite[Section~3d]{MR2739000}.

In the Heegaard Floer context this has a nice interpretation. If $D = (S,
\boldsymbol \alpha, \boldsymbol \beta, \mathbf{w},\mathbf{z})$ is a Heegaard surface
for $(Y,K)$, and $M^{\inv}$ is $\mathrm{Sym}^{g+n-1}(S \backslash
\{\mathbf{w},\mathbf{z}\})$, then a loop in $\Omega \mathcal P^{\inv}$ is a
periodic domain on the punctured Heegaard surface $S \backslash
\{\mathbf{w},\mathbf{z}\}$. The map $\Omega \mathcal P^{\inv} \rightarrow \mathbb Z$
which is the composition of \eqref{loopspaces} with projection to $\mathbb
Z$ records the Maslov index of the periodic domain, and is nullhomotopic exactly when there are no periodic domains of nonzero index on $S$. We speculate that Seidel and Smith's observations might contain a method for extending \fullref{main} from pairs $(S^3,K)$ to pairs $(Y,K)$ for which a punctured Heegaard surface contains no periodic domain of nonzero index.
\end{remark}

\section*{Appendix 1: Inclusion maps of $M^{\inv}$ into $M$}
\label{appendix1}

Since the proof that the map
\begin{align*}
h \co \mathrm{Sym}^{n-1}(S^2) &\rightarrow \mathrm{Sym}^{2n-2}(\Sigma(S^2)) \\
(x_1 \ldots x_{n-1}) &\mapsto (\wtilde{x}_1\tau(\wtilde{x}_1)\ldots
\wtilde{x}_{n-1}\tau(\wtilde{x}_{n-1}))
\end{align*}
is holomorphic is a computation on charts in the symmetric product, we place it here so as not to disrupt the flow of the arguments above.

We claim $h$ is continuous; indeed, holomorphic. Let $\wbar{x} =
(x_1\ldots x_{n-1})$ be a point in $\mathrm{Sym}^{n-1}(S^2)$. Collect
repeated points, so that $(x_1\ldots x_{n-1})$ has the form
$$(y_1,\ldots,y_1,y_2,\ldots ,y_2,\ldots ,y_{\ell},\ldots ,y_{\ell})$$
with $n{-}1$ total
entries but $\ell$ unique entries. Moreover, list points which are not
branch points of the double branched cover (that is, points which are not
in $\{ \mathbf{w},\mathbf{z}\}$) first, so that $y_1,\ldots ,y_s \not\in
\{\mathbf{w},\mathbf{z}\}$ and $y_{s+1},\ldots ,y_{\ell} \in
\{\mathbf{w},\mathbf{z} \}$ for some $s$.

First consider $y_i$ such that $1 \leq i \leq s$, so that $y_i$ is not a
branch point of the double cover. Then let $\wtilde{y}_i$ be a lift of
$y_i$. There is a neighborhood $U_i$ of $y_i$ which admits a homeomorphic
lift to a neighborhood $\wtilde{U}_i$ of $\wtilde{y}_i$ such that
$(\pi|_{\wtilde{U}_i})^{-1}\co U_i \rightarrow \wtilde{U}_i$ is
holomorphic. Moreover, we may pick $U_i$ sufficiently small such that
there is a chart $f_i \co U_i \rightarrow D$, where $D$ is the unit disk
in the complex plane, and a corresponding chart $\wtilde{f}_i = f_i \circ
\pi \co \wtilde{U}_i \rightarrow D$. In total this gives a local
biholomorphism between $U_i$ and $\wtilde{U}_i$ expressed on charts as follows.
$$\bfig
\barrsquare/^{ (}->`->`->`->/<700,400>[\wtilde{U}_i`\wtilde{U}_i`D`D;
  \big(\pi|_{\wtilde{U}_i}\big)^{-1}` f_i` \wtilde{f}_i` \mathrm{Id}]
\efig$$
Similarly, there is a neighborhood $\tau(\wtilde{U}_i)$ of the second lift
$\tau(\wtilde{y}_i)$ of $y_i$  which is homeomorphic to $U_i$ via
$(\pi|_{\tau(\wtilde{U}_i)})^{-1}$ and has a chart $\tau(\wtilde{f}_i)\co
\tau(\wtilde{U}_i) \rightarrow D$.

Now suppose $s+1 \leq i \leq \ell$, so that $y_i$ is a branch point for the double branched cover map. Then there is a chart $f_i\co  U_i \rightarrow D$ around $y_i$ and a chart $g_i \co \pi^{-1}(U_i) \rightarrow D$ with respect to which $f_i \circ \pi \circ (g_i)^{-1}$ is $x \mapsto x^2$.
\begin{align*}
\xymatrix@C+20pt{
\pi^{-1}(U_i) \ar[d]^-{g_i} \ar[r]^-{\pi|_{\pi^{-1}(U_i)}} & U_i \ar[d]^-{f_i} \\
D \ar[r]^-{x \mapsto x^2} & D
}
\end{align*}
In particular, if $y \in U_i$ and $f_i(y) = x \in \mathbb D$, then if
$\wtilde{y}$ and $\tau{\wtilde{y}}$ the two lifts of $y$ in
$\pi^{-1}(U_i)$, then $g_i(\wtilde{y})$ and $g_i(\tau(\wtilde{y}))$ are $\sqrt{x}$ and $-\sqrt{x}$ in some order.

We will insist that all of our choices of neighborhoods $U_i$ about each
$y_i$ on which we choose preferred charts $\phi_i$ be made such that if
$i_1 \neq i_2$, $U_{i_1} \cap U_{i_2} = \emptyset$, by shrinking if
necessary. This in turn implies that
$\wtilde{U}_1,\tau(\wtilde{U}_1),\ldots ,\wtilde{U}_s,
\tau(\wtilde{U}_s),\pi^{-1}(U_{s+1}),\ldots,\pi^{-1}(U_{\ell})$ are all pairwise disjoint.

We are now ready to discuss the map $h$. Recall that we began with a point
$\wbar{x} = (x_1 \ldots x_{n-1}) = (y_1 \ldots y_1y_2 \ldots y_2
\ldots y_{\ell} \ldots y_{\ell})$ in $\mathrm{Sym}^{n-1}(S^2)$. Let $k_i$
be the number of times $x_i$ appears in $\wbar{y}$. Then
$\wbar{y}$ is contained in the open neighborhood
$$\mathrm{Sym}^{k_1}(U_1) \times \cdots \times
\mathrm{Sym}^{k_{\ell}}(U_{\ell}).$$
Here the product notation arises due
to our unwavering insistence that the $U_i$ be pairwise disjoint.
Moreover, we see $h(\wbar{x}) = (\wtilde{x}_1\tau(\wtilde{x}_1)\ldots
\wtilde{x}_{n-1}\tau(\wtilde{x}_{n-1}))$ is contained in the analogous open neighborhood 
\begin{align*}
\prod_{i=1}^{s}(\mathrm{Sym}^{k_i}(\wtilde{U}_i) \times
\mathrm{Sym}^{k_i}(\tau(\wtilde{U}_i)) \times \prod_{i=s+1}^{\ell} \mathrm{Sym}^{2k_{i}}(\pi^{-1}(U_{i})).
\end{align*}
We see that locally the map $h$ is a product of maps 
\begin{align*}
h_i \co \mathrm{Sym}^{k_i}(U_i) &\rightarrow
\mathrm{Sym}^{k_i}(\wtilde{U}_i) \times \mathrm{Sym}^{k_i}(\tau(\wtilde{U}_i)) \\
(y_1'\ldots y'_{k_i}) &\mapsto ((\wtilde{y}'_1 \ldots \wtilde{y}'_{n-1}),
(\tau(\wtilde{y}'_1), \ldots ,\tau(\wtilde{y}'_{k_i}))
\end{align*}
where $\wtilde{y}'_i$ is the lift of $y_i'$ in $\wtilde{U}_i$
and similarly for $\tau(\wtilde{U}_i)$, and maps 
\begin{align*}
h_i \co \mathrm{Sym}^{k_i}(U_i) &\rightarrow \mathrm{Sym}^{2k_i}(\pi^{-1}(U_i)) \\
(y_1' \ldots y_{k_i}') &\mapsto (\wtilde{y}_1'\tau(\wtilde{y}_1')\ldots
\wtilde{y}_{k_i}'\tau(\wtilde{y}_{k_i}'))
\end{align*}
Our goal is to show $h_i$ is holomorphic in each of these cases. We begin
with the first case, in which $x_i$ is not a branch point of the double
branch covering. In this case $h_i$ carries a point $(y_1'\ldots
y_{k_i}')$ in $\mathrm{Sym}^{k_i}(U_i)$ to the product of its lifts
$(\wtilde{y}_1\ldots \wtilde{y}_{k_i})$ in
$\mathrm{Sym}^{k_i}(\wtilde{U}_i)$ and $(\tau(\wtilde{y}_1)\ldots
\tau(\wtilde{y}_{k_i}))$ in $\mathrm{Sym}^{k_i}(\tau(\wtilde{U}_i))$. Ergo
$h_i = \mathrm{Sym}^{k_i}((\pi|_{\wtilde{U}_i})^{-1}) \times
\mathrm{Sym}^{k_i}( (\pi|_{\tau(\wtilde{U}_i)})^{-1})$. We already know
how to express this map in terms of the biholomorphisms
$\mathrm{Sym}^{k_i}(f_i) \co \mathrm{Sym}^{k_i}(U_i) \rightarrow
\mathrm{Sym}^{k_i}(D^{k_i})$ and the corresponding maps
$\mathrm{Sym}^{k_i}(\wtilde{f}_i)$ on $\mathrm{Sym}^{k_i}(\wtilde{U}_i)$
and $\mathrm{Sym}^{k_i}(\tau(\wtilde{f}_i))$ on $\tau(\wtilde{U}_i)$.
$$
\bfig\barrsquare|allb|<2400,500>[\Sym^{k_i}(U_i)`
\Sym^{k_i}(\wtilde{U}_i) \times \Sym^{k_i}(\tau(\wtilde{U}_i))`
\Sym^{k_i}(D)`
\Sym^{k_i}(D) \times \Sym^{k_i}(D);
h_i = \Sym((\pi|_{U_i})^{-1})\times\Sym^{k_i}((\pi|_{\tau(\wtilde{U}_i)})^{-1})`
\Sym^{k_i}(f_i)`
\Sym^{k_i}(\wtilde{f}_i)\times\Sym^{k_i}(\tau(\wtilde{f}_i))`
\Id\times\Id]
\efig
$$
We see $j_i$ is holomorphic. For the second case we will need to be slightly more subtle, producing actual charts for $\mathrm{Sym}^{k_i}(U_i)$ and $\mathrm{Sym}^{2k_i}(\pi^{-1}(U_i))$. We can assign $\mathrm{Sym}^{k_i}(U_j) \cong \mathrm{Sym}^{k_i}(D)$ a holomorphic chart using the familiar biholomorphism \eqref{biholomorphism} which maps a point $(r_1\ldots r_{k_i})$ in $\mathrm{Sym}^{k_i}(\mathbb D)$ to the $k_i$ elementary symmetric functions of its coordinates in $D$. Let's see what this produces in our particular case. Let $\phi_{k_i}(D) \subset \mathbb C^{k_i}$ be the image of $\mathrm{Sym}^{k_i}(D)$ under the map \eqref{biholomorphism} from $\mathrm{Sym}^{k_i}(\mathbb C) \rightarrow \mathbb C^{k_i}$, and similarly for $\phi_{2k_i}(D)$.
\begin{align*}
\xymatrix{
\mathrm{Sym}^{k_i}(U_i) \ar[r]^{h_i} \ar[d]^{\mathrm{Sym}^{k_i}(f_i)}  & \mathrm{Sym}^{2k_i}(\pi^{-1}(U_i)) \ar[d]^{\mathrm{Sym}^{2k_i}(g_i)}\\
\mathrm{Sym}^{k_i}(D) \ar[d]^-{\phi_{k_i}} \ar[r] & \mathrm{Sym}^{2k_i}(\mathbb C) \ar[d]^-{\phi_{2k_i}} \\
\phi_{k_i}(D) \ar[r] & \phi_{2k_i}(D)
}
\end{align*}
Let $(y_1' \ldots y_{k_i}')$ be an arbitrary point of $U_i$, so that
$$h_i(y_1' \ldots y_{k_i}') = (\wtilde{y}_1'\tau(\wtilde{y}_1') \ldots
\wtilde{y}_{k_i}'\tau(\wtilde{y}_{k_i}').$$
Then the middle horizontal map
carries
$(r_1 \ldots r_{k_i}) = (f_i(y_1')\ldots f_i(y_{k_i}')$ to
$$(g_i(\wtilde{y}_1')g_i(\tau(\wtilde{y}_1')) \ldots
g_i(\wtilde{y}_{k_i}')g_i(\tau(\wtilde{y}_{k_i}')) =
(\sqrt{r_1}-\sqrt{r_1}\ldots \sqrt{r_{k_i}}-\sqrt{r_{k_i}}).$$
Taking symmetric functions of both sides reveals that the bottom horizontal map is expressed in coordinates as
\begin{align*}
\phi_{k_i}(D) &\rightarrow \phi_{2k_i}(D)\\
(\sigma_1(r_1,\ldots ,r_{k_i}),\ldots ,\sigma_{k_i}(r_1,\ldots ,r_{k_i})) &\mapsto  (\sigma_1(\sqrt{r_1}, -\sqrt{r_1},\ldots , \sqrt{r_{k_i}}, -\sqrt{r_{k_i}}),\ldots, 
\\ & \hspace{44pt}\sigma_{2k_i}(\sqrt{r_1}, -\sqrt{r_1},\ldots ,\sqrt{r_{k_i}}, -\sqrt{r_{k_i}}))
\end{align*}
Let's consider the symmetric functions $\sigma_j$ of
$$(a_1,\ldots ,a_{2k_i}) = (\sqrt{r_1}, -\sqrt{r_1},\ldots ,
\sqrt{r_{k_i}}, -\sqrt{r_{k_i}})$$
of our set of complex roots and their opposites. Recall from \fullref{HomotopyCohomologySection} that the functions $\sigma_j$ are defined to be the sums
\begin{align*}
\sigma_j(a_1,a_2,\ldots ,a_{2k_i}) = \sum_{1\leq i_1 <\cdots <i_j \leq 2k_i} a_{i_1} \ldots a_{i_j}.
\end{align*}
If the set of indices $(i_1,\ldots ,i_j) \subset (i_1,\ldots ,i_{2k})$ contains only one of $i_{2t-1}$ and $i_{2t}$ for any $r$, then there is a set of indices $(i_1',\ldots ,i_j')$ identical to $(i_1,\ldots ,i_j)$ except that either $i_{2t}$ is replaced with $i_{2t-1}$ or vice versa. Moreover $a_{i_1}\ldots a_{i_j} = -a_{i_1'}'\ldots a_{i_j'}$, so these two terms cancel each other out in the sum which comprises $\sigma_j$. Therefore $\sigma_j(a_1,\ldots ,a_{2k_i})$ is a sum of terms $a_{i_1}\ldots a_{i_j}$ for sets $1 \leq i_1 < \cdots  < i_j < 1$ which contain either both $i_{2t-1}$ and $i_{2t}$ or neither, for every $1 \leq t \leq \frac{j}{2}$. In particular, $\sigma_j(\sqrt{r_1}, -\sqrt{r_1},\ldots , \sqrt{r_{k_i}}, -\sqrt{r_{k_i}}) = 0$ for $j$ odd. When $j$ is even we may make the following computation.
\begin{align*}
\sigma_j(a_1,-a_1,\ldots ,a_k,-a_k) &= \sum_{1 \leq i_1 < \cdots <i_j \leq k_i} a_{2i_1 -1}a_{2i_1} \ldots a_{2i_j -1}a_{2i_j} \\
							&=\sum_{1 \leq i_1 < \cdots <i_j \leq k_i} (\sqrt{r_{i_1}})(-\sqrt{r_{i_1}}) \ldots (\sqrt{r_{i_j}})(-\sqrt{r_{i_j}}) \\
							&= \sum_{1 \leq i_1 < \cdots <i_j \leq k_i} (-1)^j r_{i_1} \ldots r_{i_j} \\
							&= (-1)^j \sigma_j(r_1,\ldots ,r_{k_i})
\end{align*}
So the bottom horizontal map in the diagram above is of the form
\begin{align*}
\phi_{k_i}(D) &\rightarrow \phi_{2k_i}(D) \\
(\sigma_1(r_1,\ldots ,r_{k_i}),\ldots ,\sigma_{k_i}(r_1,\ldots ,r_{k_i}))  &\mapsto \\(0, -&\sigma_1(r_1,\ldots ,r_{k_i}), 0, \ldots , (-1)^{k_i}\sigma_{k_i}(r_1,\ldots ,r_{k_i}))
\end{align*}
This is holomorphic, implying that $h_i$ is as well. Since all the maps $h_i$ are holomorphisms, and $h$ is locally the product $h_1 \times \cdots \times h_{\ell}$, $h$ is holomorphic.

\bibliographystyle{gtart}
\bibliography{bib}

\end{document}